\documentclass[11pt]{article}

\usepackage{palatino}
\usepackage{rotating}
\usepackage{geometry}                
\usepackage[pst-pdf=md5]{gastex} 
\gasset{frame=false} 

\newtheorem{theorem}{Theorem}[section]
\newtheorem{lemma}[theorem]{Lemma}
\newtheorem{proposition}[theorem]{Proposition}
\newtheorem{corollary}[theorem]{Corollary}
\newtheorem{conjecture}[theorem]{Conjecture}

\newtheorem{definition}[theorem]{Definition}

\newenvironment{proof}[1][Proof:]{\begin{trivlist}
\item[\hskip \labelsep {\bfseries #1}]}{\end{trivlist}}

\usepackage{graphicx}
\usepackage{amssymb}
\usepackage{epstopdf}
\title{\bf A search for quantum coin-flipping protocols \\ using optimization techniques}

\author{%
Ashwin Nayak\thanks{%
Department of Combinatorics and Optimization,
and Institute for Quantum Computing, University of Waterloo.
Address: 200 University Ave.\ W.,
Waterloo, ON, N2L 3G1, Canada. \newline
Email: {
\tt
ashwin.nayak@uwaterloo.ca}.}
\and
Jamie Sikora\thanks{%
Laboratoire d'Informatique Algorithmique: Fondements et Applications, Universit\'e Paris Diderot.
Address: 5 rue Thomas-Mann 75205 Paris cedex 13, France.  \newline
Email: {
\tt
jamie.sikora@liafa.univ-paris-diderot.fr}.}
\and
Levent Tun\c cel\thanks{%
Department of Combinatorics and Optimization, University of Waterloo.
Address: 200 University Ave.\ W., Waterloo, ON, N2L 3G1, Canada. \newline
Email: {
\tt
ltuncel@uwaterloo.ca}.}
}

\date{March 3, 2014}                                           

\usepackage{amsmath}

\newcommand{\comment}[1]{{}}

\newcommand{\e}{\mathrm{e}}
\newcommand{\sqrtt}[1]{\sqrt{#1}\sqrt{#1}^\transpose}

\newcommand{\calA}{\mathcal{A}}

\newcommand{\calP}{\mathcal{P}}

\newcommand{\calH}{\mathcal{H}}
\newcommand{\prob}{\textup{Prob}}
\newcommand{\Prob}{\textup{Prob}}
\newcommand{\id}{\mathrm{I}}

\newcommand{\tr}{\mathrm{Tr}}
\newcommand{\kb}[1]{e_{#1} e_{#1}^*}

\newcommand{\dsum}{\displaystyle\sum}
\newcommand{\inner}[2]{\left\langle #1, #2 \right\rangle}

\newcommand{\set}[1]{\left\{ #1 \right\}}

\newcommand{\Lor}{\mathrm{SOC}}
\newcommand{\RL}{\mathrm{RSOC}}
\newcommand{\conv}{\mathrm{conv}}
\newcommand{\Diag}{\mathrm{Diag}}
\newcommand{\Herm}{\mathbb{S}}
\newcommand{\Pos}{\mathbb{S}_+}
\newcommand{\pos}{\mathbb{S}_+}

\newcommand{\SDP}{\mathrm{SDP}}

\newcommand{\supp}{\mathrm{supp}}

\newcommand{\conc}{\mathrm{conc}}

\newcommand{\diag}{\mathrm{diag}}

\newcommand{\nulll}{\textrm{Null}}

\newcommand{\C}{\mathbb{C}}
\newcommand{\R}{\mathbb{R}}
\newcommand{\Z}{\mathbb{Z}}

\newcommand{\forAll}{\textrm{for all }}
\newcommand{\bit}{\in \{ 0,1 \}}

\newcommand{\twovec}[2]{\left[ \begin{array}{c} #1\\ #2 \end{array} \right]}

\newcommand{\qed}{$\quad \square$}

\newcommand{\norm}[1]{\left\| #1 \right\|}

\newcommand{\tsigma}{\tilde{\sigma}}

\newcommand{\abort}{\textup{abort}}
\newcommand{\accept}{\textup{accept}}

\usepackage{graphicx}
\newcommand{\infconv}{\scalebox{0.9}{$\square$}}

\newcommand{\half}{\frac{1}{2}}
\newcommand{\quarter}{\frac{1}{4}}
\newcommand{\zo}{\set{0,1}}
\newcommand{\eps}{\varepsilon}
\newcommand{\lmax}[1]{\lambda_{\max} \left( #1 \right)}

\usepackage{hyperref}

\usepackage{color,graphicx}

\newcommand{\transpose}{\mathrm{T}}
\newcommand{\tensor}{\otimes}
\newcommand{\rF}{\mathrm{F}}
\newcommand{\rA}{\mathrm{A}}
\newcommand{\rB}{\mathrm{B}}
\newcommand{\rO}{\mathrm{O}}
\newcommand{\tgamma}{\tilde{\gamma}}
\newcommand{\bM}{{\mathbb M}}
\newcommand{\floor}[1]{\left\lfloor #1 \right\rfloor}
\newcommand{\size}[1]{\left| #1 \right|}

\begin{document}
\maketitle

\begin{abstract}
Coin-flipping is a cryptographic task in which two physically separated,
mistrustful parties wish to generate a fair coin-flip by
communicating with each other. Chailloux and Kerenidis~(2009) designed
quantum protocols that guarantee coin-flips with near optimal bias away
from uniform, even when one party deviates arbitrarily from the
protocol. The probability of any outcome in these protocols is provably
at most~$\tfrac{1}{\sqrt{2}} + \delta$ for any given~$\delta > 0$.
However, no explicit
description of these protocols is known, and the number of rounds in the
protocols tends to infinity as~$\delta$ goes to~$0$. In fact, the
smallest bias achieved by known explicit protocols is~$1/4$ (Ambainis,
2001).

We take a \emph{computational optimization\/} approach, based mostly
on convex optimization, to the search for simple and
explicit quantum strong coin-flipping protocols. We present a search
algorithm to identify protocols with low bias within a
natural class, protocols based on \emph{bit-commitment\/} (Nayak and
Shor, 2003). To make this search computationally feasible, we further
restrict to commitment states \emph{{\`a} la\/} Mochon (2005).
An analysis of the resulting protocols via semidefinite programs (SDPs)
unveils a simple structure. For example, we show that the SDPs reduce to
second-order cone programs. We devise novel cheating strategies in the
protocol by restricting the semidefinite programs and use the strategies
to prune the search.

The techniques we develop enable a computational search for protocols given by a mesh
over the corresponding parameter space. The protocols have up to six
rounds of communication, with messages of varying dimension and include
the best known explicit protocol (with bias~$1/4$). We conduct
two kinds of search: one for protocols with bias below~$0.2499$, and one
for protocols in the neighbourhood of protocols with bias~$1/4$. Neither
of these searches yields better bias. 
Based on the mathematical ideas behind the search algorithm, we prove a lower bound of~$0.2487$ on the bias of a class of four-round protocols.
\end{abstract}

\tableofcontents


\section{Introduction}

Some fundamental problems in the area of Quantum Cryptography
allow formulations in the language of convex optimization in the space
of hermitian matrices over the complex numbers, in particular,
in the language of semidefinite optimization.  These formulations enable
us to take a computational optimization approach towards solutions of
some of these problems.  In the rest of this section, we describe
\emph{quantum coin-flipping} and introduce our approach.

\subsection{Quantum coin-flipping}

Coin-flipping is a classic cryptographic task introduced by
Blum~\cite{Blu81}. In this task, two remotely situated parties,
Alice and Bob, would like to agree on a uniformly random bit by
communicating with each other. The complication is that neither party
trusts the other. If Alice were to toss a coin and send the
outcome to Bob, Bob would have no means to verify whether this was a
uniformly random outcome. In particular, if Alice wishes to cheat, she
could send the outcome of her choice without any possibility of being caught
cheating. We are interested in a communication protocol that is \emph{designed
to protect\/} an honest party from being cheated.

More precisely, a ``strong coin-flipping protocol'' with
bias~$\epsilon$ is a two-party communication protocol in the style
of Yao~\cite{Yao79,Yao93}. In the protocol, the two players, Alice and Bob,
start with no inputs and compute a value~$c_\rA, c_\rB \in
\set{0,1}$, respectively, or declare that the other player is
cheating. If both players are honest, i.e., follow the protocol, then
they agree on the outcome of the protocol ($c_\rA = c_\rB$), and
the coin toss is fair ($\Pr(c_\rA = c_\rB = b) = 1/2$, for any~$b
\in \set{0,1}$). Moreover, if one of the players deviates arbitrarily
from the protocol in his or her local computation, i.e., is ``dishonest''
(and the other party is honest), then the probability of either outcome~$0$
or~$1$ is at most~$1/2 + \epsilon$. Other variants of coin-flipping have
also been studied in the literature. However, in the rest of the
article, by ``coin-flipping'' (without any modifiers) we mean
\emph{strong\/} coin flipping.

A straightforward game-theoretic argument proves that if the two
parties in a coin-flipping protocol
communicate classically and are computationally unbounded,
at least one party can cheat perfectly (with bias~$1/2$).
In other words, there is at least one party, say Bob, and at least one
outcome~$b \in \set{0,1}$ such that Bob can ensure outcome~$b$ with
probability~$1$ by choosing his messages in the protocol appropriately.
Consequently, classical coin-flipping protocols with bias~$\epsilon <
1/2$ are only possible under complexity-theoretic assumptions, and
when Alice and Bob have limited computational resources.

Quantum communication offers the possibility of ``unconditionally
secure'' cryptography, wherein the security of a protocol
rests solely on the validity of
quantum mechanics as a faithful description of nature. The first few
proposals for quantum information processing, namely the Wiesner quantum
money scheme~\cite{Wiesner83} and the Bennett-Brassard quantum key expansion
protocol~\cite{BB84} were motivated by precisely this idea. These
schemes were indeed eventually shown to be unconditionally secure in
principle~\cite{M01,LC99,PS00,MVW12}. In light of these
results, several researchers have studied the possibility of
\emph{quantum\/} coin-flipping protocols, as a step towards studying
more general secure multi-party computations.

Lo and Chau~\cite{LC97} and Mayers~\cite{May97} were the first to
consider quantum protocols for coin-flipping without any computational
assumptions. They proved that no protocol with a finite number of rounds
could achieve~$0$ bias. Nonetheless, Aharonov, Ta-Shma, Vazirani, and
Yao~\cite{ATVY00} designed a simple, three-round quantum protocol that
achieved bias~$\approx 0.4143 < 1/2$. This is impossible classically, even
with an unbounded number of rounds. Ambainis~\cite{Amb01} designed
a protocol with bias~$1/4$ \emph{{\`a} la\/} Aharonov \emph{et al.\/},
and proved that it is optimal within a class (see also
Refs.~\cite{SR01,KN04} for a simpler version of the protocol and a
complete proof of security).  Shortly thereafter, Kitaev~\cite{Kit03}
proved that any strong coin-flipping protocol with a finite number of
rounds of communication has bias at least~$(\sqrt{2}-1)/2 \approx 0.207$ (see
Ref.~\cite{GW07} for an alternative proof). Kitaev's seminal work
uses semidefinite optimization in a central way.  This argument extends to
protocols with an unbounded number of rounds. This remained the state
of the art for several years, with inconclusive evidence in either
direction as to whether~$1/4 = 0.25$ or~$(\sqrt{2}-1)/2$
is optimal. In 2009, Chailloux and Kerenidis~\cite{CK09} settled this
question through an elegant protocol scheme that has
bias at most~$(\sqrt{2}-1)/2 + \delta$ for any~$\delta > 0$ of our
choice (building on~\cite{Moc07}, see below). We refer to this as
the CK protocol.

The CK protocol uses breakthrough work by Mochon~\cite{Moc07}, which
itself builds upon the ``point game'' framework proposed by Kitaev.
Mochon shows there are \emph{weak\/} coin-flipping protocols with
arbitrarily small bias. (This work has appeared only in the form of
an unpublished manuscript, but has been verified by experts on the
topic; see e.g.~\cite{AharonovCGKM13}.) A weak coin-flipping protocol
is a variant of
coin-flipping in which each party favours a distinct outcome, say Alice
favours~$0$ and Bob favours~$1$. The requirement when they are honest is
the same as before. We say it has bias~$\epsilon$ if the following
condition holds. When Alice is dishonest and Bob honest, we only
require that Bob's outcome is~$0$ (Alice's favoured outcome) with
probability at most~$1/2+\epsilon$. A similar condition to protect
Alice holds, when she is honest and Bob is dishonest.
The weaker requirement of security against a dishonest player allows us
to circumvent the Kitaev lower bound. While Mochon's work pins down
the optimal bias for weak coin-flipping, it does this in a non-constructive
fashion: we only know of the \emph{existence\/} of protocols with
arbitrarily small bias, not of its \emph{explicit description\/}.
Moreover, the number of rounds tends to infinity as the bias decreases to $0$.
As a consequence, the CK protocol for strong coin-flipping is also
existential, and the number of rounds tends to infinity as the bias
decreases to $(\sqrt{2}-1)/2$.  It is perhaps
very surprising that no progress on finding better explicit protocols has
been made in over a decade.

\subsection{Search for explicit protocols}
\label{sec-search}

This work is driven by the quest to find \emph{explicit\/} and
\emph{simple\/} strong coin-flipping protocols with bias smaller
than~$1/4$.  There are two main challenges in this quest. First, there
seems to be little insight into the structure (if any) that protocols with
small bias have; knowledge of such structure might help narrow our search
for an optimal
protocol. Second, the analysis of protocols, even those of a restricted
form, with more than three rounds of communication is technically quite
difficult. As the first step in deriving the~$(\sqrt{2}-1)/2$ lower bound,
Kitaev~\cite{Kit03} proved that the optimal cheating
probability of any dishonest party in a protocol with an explicit
description is characterized by a semidefinite program (SDP). While this
does not entirely address the second challenge, it reduces
the analysis of a protocol to that of a well-studied optimization problem.
In fact this formulation as an SDP enabled Mochon to analyze
an important class of weak coin-flipping protocols~\cite{Moc05},
and later discover the optimal weak coin flipping
protocol~\cite{Moc07}. SDPs resulting from strong coin-flipping protocols,
however, do not appear to be amenable to similar analysis.

We take a \emph{computational optimization\/} approach to the search for explicit strong
coin-flipping protocols. We focus on a class of protocols studied by Nayak and
Shor~\cite{NS03} that are based on ``bit commitment''. This is a natural
class of protocols that generalizes those due to Aharonov \emph{et al.\/}
and Ambainis, and provides a rich test bed for our search.  (See
Section~\ref{family} for a description of such protocols.) Early
proposals of multi-round protocols in this class were all shown to have
bias at least~$1/4$, without eliminating the possibility of smaller
bias (see, e.g., Ref.~\cite{NS03}).
A characterization of the smallest bias achievable in this class
would be significant progress on the problem: it would either lead to
simple, explicit protocols with bias smaller than~$1/4$, or we would
learn that protocols with smaller bias take some other, yet to be
discovered form.

Chailloux and Kerenidis~\cite{CK11} have studied a version of
quantum bit-commitment that may have implications for coin-flipping.
They proved that in any quantum bit-commitment protocol with
computationally unbounded players, at least one party can cheat with bias
at least~$\approx 0.239$. Since the protocols we study involve two
interleaved commitments to independently chosen bits, this lower bound
does not apply to the class. Chailloux and Kerenidis also give a
protocol scheme for
bit-commitment that guarantees bias arbitrarily close to~$0.239$.
The protocol scheme is non-constructive as it uses the Mochon weak
coin-flipping protocol. It is possible that any explicit protocols we
discover for coin-flipping could also lead to explicit bit-commitment with
bias smaller than~$1/4$.

We present an algorithm for finding protocols with low bias.
Each bit-commitment based coin-flipping protocol is specified by
a~$4$-tuple of quantum states.  At a high level, the algorithm
iterates through a suitably fine mesh of such~$4$-tuples,
and computes the bias of the resulting protocols. The size of the mesh
scales faster than~$1/\nu^{\kappa D}$, where~$\nu$ is a precision
parameter, $\kappa$ is a universal constant, and~$D$ is the dimension
of the states. The dimension itself scales as~$2^n$, where~$n$ is the
number of quantum bits involved. In order to minimize the doubly
exponential size of the set of $4$-tuples we examine, we further
restrict our attention to states of the form introduced by Mochon for
weak coin-flipping~\cite{Moc05}. The additional advantage of this kind of
state is that the SDPs in the analysis of the protocols simplify
drastically. In fact, all but a few constraints reduce to linear equalities so
that the SDPs may be solved more efficiently.

Next, we employ two
techniques to prune the search space of~$4$-tuples. First, we use a
sequence of strategies for dishonest players whose bias is given by a
closed form expression determined by the four states. The idea is that
if the bias for any of these strategies  is higher than~$1/4$ for
any~$4$-tuple of states, we may
safely rule it out as a candidate optimal protocol. This also has the
advantage of avoiding a call to the SDP solver, the computationally most
intensive step in the search algorithm. The second technique is to
invoke symmetries in the search space as well as in the problem to
identify protocols with the same bias. The idea here is to compute the
bias for as few members of an equivalence class of protocols as
possible.

These techniques enable a computational search for protocols with up to six
rounds of communication, with messages of varying dimension. The Ambainis
protocol with bias~$1/4$ has three rounds, and it is entirely possible
that a strong coin-flipping protocol with a small number of
rounds be optimal. Thus, the search non-trivially extends
our understanding of this cryptographic primitive. We elaborate on this
next.

\subsection{The results}
\label{sec-results}

We performed two types of search. The first was an optimistic search
that sought protocols within the mesh with bias at most $1/4$
minus a small constant.  We chose the constant to be $0.001$. The rationale
here was that if the mesh contains protocols with bias close to the
lower bound of $\approx 0.207$, we would find protocols that have
bias closer to~$0.25$ (but smaller than it) relatively quickly.
We searched for four-round protocols in which each message is of
dimension ranging from $2$ to $9$, each with varying fineness for the
mesh. We found that our heuristics, i.e., the filtering by fixed cheating
strategies, performed so well that they eliminated every protocol:
all of the protocols given by the mesh were found to have bias larger
than $0.2499$ without the need to solve any SDP.  Inspired by the search algorithm, we give an analytical proof that
four-round qubit protocols have bias at least~$0.2487$.

The initial search for four-round protocols helped us fine-tune the
filter by a careful selection of the order in which the cheating strategies
were tried. The idea was to eliminate most protocols with the least amount
of computation. This made it feasible for us to search for protocols in
finer meshes, with messages of higher dimension, and with a larger number
of rounds.
In particular, we were able to check six-round protocols with messages of dimension $2$ and $3$. Our heuristics again performed very well, eliminating almost every protocol before any SDP needed to be solved. Even during this search, not a single protocol with bias less than $0.2499$ was found. We also performed a search over
meshes shifted by a randomly chosen parameter. This was to avoid potential anomalies caused by any special properties of the mesh
we used. No protocols with bias less than $0.2499$ were found in this search either.

The second kind of search focused on protocols with bias close to~$0.25$.
We first identified protocols in the mesh with the least bias. Not
surprisingly, these protocols all had computationally verified bias $1/4$.
We zoned in on the neighbourhood of these protocols. The idea here was to see if there are perturbations to the $4$-tuple that lead
to a decrease in bias. This search revealed $2$ different equivalence classes of protocols for the four-round version and $6$ for the six-round version. Four of these eight protocols are equivalent
to optimal three-round protocols (within this class). However, the four remaining six-round protocols bear no resemblance to any known protocol
with bias $1/4$. A search in the neighbourhoods of all these protocols revealed no protocols with bias less than $1/4$ (details in Section~\ref{sect:numerical}).

It may not immediately be evident that the above searches involved
a computational examination of extremely large sets of protocols and that
the techniques described above in Section~\ref{sec-search}, were crucial
in enabling this search.
The symmetry arguments pruned the searches drastically, and in some cases only $1$ in every $1,000,000$ protocols needed to be checked. In most cases, the cheating strategies (developed in Section~\ref{sect:filter}) filtered out the rest of the protocols entirely. To give an example of the efficiency of our search, we were able to check $2.74 \times 10^{16}$ protocols in a matter of days. Without the symmetry arguments and the
use of cheating strategies as a filter, this same search would have taken well over $69$ million years, even using the very simplified forms of the SDPs.
Further refinement of these ideas may make a more thorough search of
protocols with four or more rounds feasible.

The search algorithm, if implemented with exact feasibility guarantees,
has the potential to give us \emph{computer aided proofs\/} that certain classes of protocols in the family \emph{do not
achieve optimal bias\/}.
Suppose we use a mesh such that given any $4$-tuple~$S$ of states,
there is a $4$-tuple $S'$ in the mesh such that the pairwise fidelity
between corresponding distributions is at least~$1 - \delta$. Further
suppose the numerical approximation to the bias for~$S'$ has additive
error~$\tau$ due to the filter or SDP solver,
and finite precision arithmetic\footnote{Note that in our experiments feasibility is guaranteed only
up to a tolerance, so as a result we do not have an independently verifiable upper bound on the additive
error in terms of the objective value.  Indeed, efficiently obtaining an exact feasible solution to SDPs,
in general, is still an open problem
at the time of this writing.}.
If the algorithm reports that there are no tuples in the mesh with bias
at most~$\epsilon^*$, then it holds that there are no~$4$-tuples, even
outside the mesh, with bias at most~$\epsilon^* - \sqrt{8 \, \delta} - \tau$.
The fineness of the mesh we are able to support currently is not sufficient for such proofs. A refinement of the search algorithm along the lines described above, however, would yield lower bounds for new classes
of bit-commitment based protocols. 

\comment{
The four-round search gives us computer aided \emph{proofs\/}
that certain classes of protocols in the family \emph{do not
achieve optimal bias\/}.
Suppose we use a mesh such that given any $4$-tuple~$S$ of states,
there is a $4$-tuple $S'$ in the mesh such that the pairwise fidelity
between corresponding distributions is at least~$1 - \delta$. Further suppose
the numerical approximation to the bias for~$S'$ has additive error~$\nu$
due to the filter or SDP solver, and finite precision arithmetic.
If the algorithm reports that there are no tuples in the mesh with bias
at most~$\epsilon^*$, then it holds that there are no~$4$-tuples, even
outside the mesh, with bias at most~$\epsilon^* - \sqrt{\delta} - \nu$.
The fineness of the mesh we are able to support for four-round protocols
with message dimension~$2$ is sufficient for us to conclude
that such protocols do not achieve optimal bias (see
Section~\ref{sec-cproof}). This class is incomparable with three round
bit-commitment based protocols, for which bias~$1/4$ is known to be
optimal~\cite{Amb01}. We thus obtain lower bounds
for a new set of bit-commitment based protocols. A refinement of
the search algorithm along the lines described above would extend the
lower bound to yet larger classes.
} 

Finally, based on our computational findings, we make the following conjecture:

\begin{conjecture}
\label{conj-bias}
Any strong coin-flipping protocol based on bit-commitment as defined
formally in Section~\ref{family} has bias at least~$1/4$.
\end{conjecture}

This conjecture, if true, would imply that we need to investigate new kinds
of protocols to find ones with bias less than $1/4$. Regardless of the truth
of the above conjecture, we hope that the new techniques developed for
analyzing protocols via modern optimization methods and for simplifying
semidefinite optimization problems with special structure will be helpful
in future work in the areas of quantum computing and semidefinite programming.

\paragraph{Organization of the paper.}

We begin with an introduction to the ideas contained in this paper in Section~\ref{sect:background}. Section~\ref{ssect:quantum} introduces quantum computing background and Section~\ref{ssect:opt} introduces semidefinite programming and related optimization classes. Section~\ref{sect:CF} defines strong coin-flipping protocols and the measure of their security
(namely, their bias). We define the notion of protocols based on bit-commitment in Section~\ref{family}. We model optimal cheating strategies for
such protocols using semidefinite programming in Section~\ref{sect:opt}. Sections~\ref{sect:filter} and~\ref{sect:symmetry} exploit the structure of the semidefinite programs in order to design the search algorithm presented in Section~\ref{sect:algo}. We conclude with computational results in Section~\ref{sect:numerical} and some final remarks in Section~\ref{sect:conclusions}.

The background material on quantum computation and optimization is
aimed at making this work accessible to researchers
in both communities.  Readers conversant with either topic need only skim
the corresponding sections to familiarize themselves with the notation used.
Proofs of most results are deferred to the appendix.


\section{Background and notation}\label{sect:background}

In this section, we establish the notation and the necessary background for this paper.

\subsection{Linear algebra}

For a finite set $A$, we denote by $\R^A$, $\R_+^A$, $\Prob^A$, and $\C^A$ the set of real vectors, nonnegative real vectors, probability vectors, and complex vectors, respectively, each indexed by $A$. We use $\R^n$, $\R_+^n$, $\Prob^n$, and $\C^n$ for the special case when $A = \set{1, \ldots, n}$. For~$x
\in A$, the
vectors~$e_x$ denote the standard basis vectors of~$\R^A$. The
vector~$e_A \in \R^A$ denotes the all~$1$ vector~$\sum_{x \in A} e_x$.

We denote by $\Herm^A$ and $\pos^A$ the set of Hermitian matrices and positive semidefinite matrices, respectively, each over
the reals with columns and rows indexed by $A$.

It is convenient to define $\sqrt{x}$ to be the element-wise square root of a nonnegative vector $x$. The element-wise square root of a probability vector yields a unit vector (in the Euclidean norm). This operation maps a probability vector to a quantum state, see Subsection~\ref{ssect:quantum}.

For vectors $x$ and $y$, the notation $x \geq y$ denotes that $x-y$ has nonnegative entries, $x > y$ denotes that $x-y$ has positive entries, and for matrices $X$ and $Y$, the notation $X \succeq Y$ denotes that  $X - Y$ is positive semidefinite, and $X \succ Y$ denotes $X - Y$ is positive definite when the underlying spaces are clear from context. When we say that a matrix is positive semidefinite or positive definite, it is assumed to be Hermitian which implies that $\Herm_+^A \subset \Herm^A$.

The Kronecker product of an~$n \times n$ matrix $X$ and another matrix $Y$,
denoted $X \otimes Y$, is defined as
\[ X \otimes Y := \left[ \begin{array}{cccc}
X_{1,1} \; Y & X_{1,2} \; Y & \cdots & X_{1,n} \; Y \\
X_{2,1} \; Y & X_{2,2} \; Y & \cdots & X_{2,n} \; Y \\
\vdots & \vdots & \ddots & \vdots \\
X_{n,1} \; Y & X_{n,2} \; Y & \cdots & X_{n,n} \; Y
\end{array} \right] \enspace. \]
Note that $X \otimes Y \in \pos^{A \times B}$ when $X \in \pos^A$ and $Y \in \pos^B$ and $\tr(X \otimes Y) = \tr(X) \cdot \tr(Y)$ when $X$ and $Y$ are square.

The \emph{Schatten $1$-norm}, or \emph{nuclear norm}, of a matrix $X$ is defined as
\[ \norm{X}_* := \tr(\sqrt{X^* X}), \]
where $X^*$ is the adjoint of $X$ and $\sqrt{X}$ denotes the square root of a positive semidefinite matrix $X$, i.e., the positive semidefinite matrix $Y$ such that $Y^2 = X$. Note that the $1$-norm of a matrix is the sum of its singular values. The $1$-norm of a vector $p \in \C^A$ is denoted as
\[ \norm{x}_1 := \sum_{x \in A} |p_x|. \]

We use the notation $\bar{a}$ to denote the complement of a bit $a$ with respect to $0$ and $1$ and $a \oplus b$ to denote the XOR of the bits $a$ and $b$. We use $\mathbb{Z}_2^n$ to denote the set of $n$-bit binary strings.

For a vector $p \in \R^A$, we denote by $\Diag(p) \in \mathbb{S}^A$ the diagonal matrix with $p$ on the diagonal. For a matrix $X \in \Herm^A$, we denote by $\diag(X) \in \R^A$ the vector on the diagonal of $X$.

For a vector $x \in \C^A$, we denote by $\supp(x)$ the set of indices of $A$ where $x$ is nonzero. We denote by $x^{-1}$ the element-wise inverse of $x$ (mapping the $0$ entries to $0$).

For a matrix $X$, we denote by $\nulll(X)$ the nullspace of $X$, by $\det(X)$ the determinant of $X$, and by $\lambda_{\max}(X)$ the largest eigenvalue of $X$. We denote by $\inner{X}{Y}$ the standard inner product~$\tr(X^*
Y)$ of matrices~$X,Y$ of the same dimension.

\subsection{Convex analysis} \label{convex}

A \emph{convex combination} of finitely many vectors $x_1, \ldots, x_n$ is any vector of the form $\sum_{i=1}^n \lambda_i x_i$, when $\lambda_1, \ldots, \lambda_n \in [0,1]$ satisfy $\sum_{i=1}^n \lambda_i = 1$. The \emph{convex hull} of a set $C$ is the set of convex combinations of elements of $C$, denoted $\mathrm{conv}(C)$. A set $C$ is \emph{convex} if $C = \mathrm{conv}(C)$.

A \emph{convex function} $f : \R^n \to \R \cup \set{\infty}$ is one that satisfies
\[ f(\lambda x + (1-\lambda) y) \leq \lambda f(x) + (1-\lambda) f(y), \; \forAll x, y \in \R^n, \lambda \in [0,1]. \]
A convex function is \emph{strictly convex} if
\[ f(\lambda x + (1-\lambda) y) < \lambda f(x) + (1-\lambda) f(y), \; \forAll x \neq y, \, x, y \in \R^n, \lambda \in (0,1). \]
We say that a convex function is proper if $f(x) < + \infty$ for some $x \in \R^n$.
The \emph{epigraph} of a function $f$ is the set
\[ \mathrm{epi}(f) := \set{(x,t): f(x) \leq t} \]
which are the points above the graph of the function~$f$.
A function is convex if and only if its epigraph is a convex set.

A function $f : \R^n \to \R \cup \set{-\infty}$ is \emph{(strictly) concave} if $-f$ is (strictly) convex, and proper when $f(x) > - \infty$ for some $x \in \R^n$.
The \emph{hypograph} of a function $f$ is the set
\[ \mathrm{hypo}(f) := \set{(x,t): f(x) \geq t} \]
which are the points below the graph of the function~$f$.
A function is concave if and only if its hypograph is a convex set.

Let ${f_1, \ldots, f_n} : \R^m \to \R \cup \set{\infty}$ be proper, convex functions. We denote the  \emph{convex hull} of the functions $\{ f_1, \ldots, f_n \}$ by $\conv \{f_1, \ldots, f_n \}$ which is the greatest convex function $f$ such that $f(x) \leq f_1(x), \ldots, f_n(x)$ for every $x \in \R^m$. The convex hull can be written in terms of the epigraphs
\[ \conv \{f_1, \ldots, f_n \}(x) := \inf \set{t: (x,t) \in \mathrm{conv}( \cup_{i=1}^n \mathrm{epi}(f_i))}. \]
We denote the \emph{concave hull} of $\{ f_1, \ldots, f_n \}$ by $\conc \{f_1, \ldots, f_n \}$ which can be written as
\[ \conc \set{f_1, \ldots, f_n} := - \conv \set{-f_1, \ldots, -f_n} \]
when $f_1, \ldots, f_n : \R^m \to \R \cup \{ - \infty \}$ are proper, concave functions. The concave hull is the least concave function $f$ such that $f(x) \geq f_1(x), \ldots, f_n(x)$ for every $x \in \R^m$ and can be written as
\[ \conc \{f_1, \ldots, f_n \}(x) := \sup \set{t: (x,t) \in \mathrm{conv}( \cup_{i=1}^n \mathrm{hypo}(f_i))}. \]

A \emph{convex optimization problem} or \emph{convex program} is one of the form
\[ \inf_{x \in C} f(x), \]
where $f$ is a convex function and $C$ is a convex set. Alternatively, one could maximize a concave function over a convex set.

\comment{
A set $C \subseteq \R^n$ is closed if it contains every limit point of sequences within $C$. In a complex Euclidean space, a set is compact if and only if it is closed and bounded.

\comment{We define the \emph{relative interior} of a nonempty, convex set $C$ as
\[ \mathrm{relint}(C) := \set{x \in C : \forall y \in C, \exists \lambda > 1 : \lambda x + (1 - \lambda) y \in C}. \]
} 

We call a convex set $K$ a cone if $\lambda K \subseteq K$, for all $\lambda > 0$. This thesis concerns the optimization of linear functions over closed, convex cones, see Subsection~\ref{ssect:opt}.
Given a set $C \subseteq \R^n$, its dual cone, denoted $C^*$, is defined as
\[ C^* := \set{ x \in \R^n : \inner{x}{y} \geq 0, \, \forAll y \in C}. \]
One can check that the dual cone is always a closed, convex cone.
Also, we have that $C_1 \subseteq C_2$ implies $C_1^* \supseteq C_2^*$ and the converse holds if $C_1$ and $C_2$ are closed convex cones.

A function $f: \Herm^n \to \Herm^m$ is said to be \emph{operator monotone} if
\[ f(X) \succeq f(Y) \quad \textup{ when } \quad X \succeq Y. \]
The set of operator monotone functions is a convex cone.

A function $f: \R^n \to \R$ is said to be \emph{positively homogeneous} if
\[ f(\lambda x) = \lambda f(x), \quad \forAll \lambda > 0. \]

A \emph{polyhedron} is the solution set of a system of finitely many linear inequalities (or equalities). A \emph{polytope} is a bounded polyhedron.

} 

\subsection{Quantum information}\label{ssect:quantum}

In this subsection, we give a brief introduction to quantum information. For a more thorough treatment of the subject, we refer the reader to \cite{NC00}.

\paragraph{Quantum states} \quad \\

Quantum states are a description of the state of a physical system, such as the spin of an electron. In
the simplest case, such a state is a \emph{unit vector} in a finite-dimensional Hilbert space (which is a complex Euclidean space). For example, the following vectors are quantum states in $\C^2$
\[ e_0 := \twovec{1}{0}, \; e_1 := \twovec{0}{1}, \; e_+ := \dfrac{1}{\sqrt 2} \twovec{1}{1}, \; e_- := \dfrac{1}{\sqrt 2} \twovec{1}{-1}. \]
The first two are standard basis vectors and can be thought of as the logical states of a standard computer. In general, a qubit can be written as
\[ \psi := \alpha_0 \, e_0 + \alpha_1 \, e_1, \]
where $\alpha_0, \alpha_1 \in \C$ satisfy $|\alpha_0|^2 + |\alpha_1|^2 = 1$. This condition ensures that $\psi$ has norm equal to $1$. Up
to factor of modulus~$1$,  the set of pairs~$(\alpha_0, \alpha_1)$ defining a two-dimensional
quantum state is in one-to-one correspondence with the unit sphere in $\R^3$.

Systems with a two dimensional state space are called \emph{quantum bits} or \emph{qubits}. The
state space of a sequence of~$n$ qubits is given by the~$n$-fold tensor
product~$(\C^2)^{\tensor n} \cong \C^{2^n}$. Higher
dimensional systems, say, of dimension~$d \le 2^n$, may be viewed as being
composed of a sequence of~$n$ qubits via a canonical isometry~$\C^d
\rightarrow \C^{2^n}$.

Notice that $e_+ = \frac{1}{\sqrt 2} e_0 + \frac{1}{\sqrt 2} e_1$ and $e_- = \frac{1}{\sqrt 2} e_0 - \frac{1}{\sqrt 2} e_1$. These states are said to be in a \emph{superposition} of the states $e_0$ and $e_1$ and exhibit properties of being in both states at the same time. This is in part what gives quantum computers the power to efficiently tackle hard problems such as factoring \cite{Sho94}.

In general, a system may be in a random superposition according
to some probability distribution. Suppose a quantum system is in such a state
drawn from the ensemble of states $(\psi_0, \psi_1, \ldots, \psi_n)$ with
probabilities $(p_0, p_1, \ldots, p_n)$, respectively. This quantum state may
be described more succinctly as a \emph{density matrix}, defined as
\[ \sum_{i=0}^n p_i \, \psi_i \psi_i^*. \]
Notice that this matrix is positive semidefinite and has unit trace. Moreover, any positive semidefinite matrix with unit trace can be written in the above form using its spectral decomposition.

Two different probability distributions over superpositions may have the
same density matrix. For example, density matrices do not record ``phase
information'', i.e., the density matrix of state~$\psi$ is the same as
that of~$-\psi$. However, two ensembles with the same density matrix
behave identically under all allowed physical operations. Therefore,
there is no loss in working with density matrices, and we identify an
ensemble with its density matrix.

A quantum superposition given by the vector $\psi$ corresponds to the rank $1$ density matrix $\psi \psi^*$ and we call it
a \emph{pure state}.  States with a density matrix of rank $2$ or more
are said to be \emph{mixed}.

\paragraph{Quantum operations} \quad \\

The most basic quantum operation is specified by a unitary
transformation.
Suppose $U$ is a unitary operator acting on $\C^A$ and $\psi \in \C^A$ is a quantum state. If we apply $U$ to $\psi$ then the resulting quantum state is $U \psi \in \C^A$. Note this is a well-defined quantum state since unitary operators preserve Euclidean norm.

Suppose we are given a state drawn from the ensemble~$(\psi_0, \psi_1, \ldots, \psi_n)$ with probabilities $(p_0, p_1, \ldots, p_n)$. Then if we apply a unitary matrix $U$ to the state, the resulting state is given
by the
ensemble~$(U \psi_0, U \psi_1, \ldots, U \psi_n)$ with the same probabilities. The
new density matrix is thus
\[ \sum_{i=0}^n p_i \, U \psi_i \psi_i^* U^* = U \left( \sum_{i=0}^n p_i \, \psi_i \psi_i^* \right) U^*, \]
where $U^*$ is the adjoint of $U$.
Thus, if we apply the unitary $U$ to a state (with density matrix)~$\rho$, then the
resulting quantum state is $U \rho U^*$. Note that this matrix is still positive semidefinite with unit trace.

We assume that parties capable of quantum information processing
have access to qubits initialized to a fixed quantum state, say~$e_0$,
can apply arbitrary unitary operations, and can physically
transport (``send'') qubits without disturbing their state.
We use the phrase ``prepare a quantum state~$\psi \in \C^A$'' to mean
that we start with sufficiently many qubits (say~$n$ such that~$\C^A
\subseteq \C^{2^n}$) in
state~$e_0^{\otimes n}$ and apply any unitary transformation that
maps~$e_0^{\otimes n}$ to~$\psi$.

\paragraph{Quantum measurement} \quad \\

Measurement is a means of extracting classical information from a quantum state.
A \emph{quantum measurement\/} on space~$\C^{A}$ is a sequence of positive semidefinite operators $(\Pi_1, \ldots, \Pi_n)$,
with~$\Pi_i \in \Pos^A$ for each~$i \in \set{1,\dotsc,n}$, satisfying $\sum_{i=1}^n \Pi_i = \id$. This sequence of operators is also called a \emph{positive operator valued measure} or a \textup{POVM}
in the literature. If we have
some qubits in state~$\rho$ and we apply the measurement $(\Pi_1, \ldots, \Pi_n)$
(or ``observe the qubits according to the measurement''), we obtain \emph{outcome\/}~``$i$'' with probability $\inner{\Pi_i}{\rho}$,
and the state of the qubits becomes $\Pi_i \rho \Pi_i/ \inner{\Pi_i}{\rho}$.
The definitions of density matrices and measurements establish $(\inner{\Pi_i}{\rho})$ as a well-defined probability distribution over the indices. The
alteration of state resulting from a measurement is referred to as a
\emph{collapse\/}.
Due to this restricted kind of access,
in general only a limited amount of classical information may be extracted from
a given quantum state.

For example, if we apply the measurement $\set{\Pi_0 := e_0 e_0^*, \Pi_1 := e_1 e_1^*}$ to the state $e_+ e_+^*$, we obtain the outcomes:
\[ \left\{ \begin{array}{rcl} ``0" & \text{ with probability } & \inner{\Pi_0}{e_+ e_+^*} = 1/2, \\ ``1" & \text{ with probability } & \inner{\Pi_1}{e_+ e_+^*} = 1/2. \end{array} \right. \]

\comment{ 
We note here a special measurement which is optimal for distinguishing two quantum states called the \emph{Helstrom measurement}. If Alice gives Bob one of two quantum states $\rho_0$ or $\rho_1$ uniformly at random then this measurement gives Bob a probabilistic method to learn which state was sent such that the probability of an incorrect inference is minimized. This measurement succeeds with probability $\half + \quarter \norm{\rho_0 - \rho_1}_{*}$. Note if $\rho_0 = \rho_1$ then this is no better than randomly guessing and if $\inner{\rho_0}{\rho_1} = 0$ then this succeeds with probability $1$.
} 

\paragraph{Multiple quantum systems} \quad \\

For convenience, we refer to a quantum system with state space~$\C^A$ by
the index set~$A$.  Suppose we have two quantum systems~$A_1, A_2$
that are independently in pure states $\psi_1 \in \C^{A_1}$ and $\psi_2 \in \C^{A_2}$. Their combined state is $\psi_1 \otimes \psi_2 \in \C^{A_1} \otimes \C^{A_2} \cong \C^{A_1 \times A_2}$ where $\otimes$ denotes the Kronecker (or tensor) product. Note that the Kronecker product has the property that $\norm{x \otimes y}_2 = \norm{x}_2 \norm{y}_2$ so unit norm is preserved. It is not always possible to decompose a vector
in~$\C^{A_1} \tensor \C^{A_2}$ as a Kronecker product of vectors
in~$\C^{A_1}$ and~$\C^{A_2}$; a state with this property is said to be
\emph{entangled\/}. For example, the state $\Phi^+ = [1/\sqrt{2}, 0, 0, 1/\sqrt{2}]^\transpose$
is entangled; it cannot be expressed as $\psi_1 \otimes \psi_2$ for any choice of $\psi_1, \psi_2 \in \C^2$.

These concepts extend to mixed states as well.
If two disjoint quantum systems are independently in states $\rho_1 \in \pos^{A_1}$ and $\rho_2 \in \Pos^{A_2}$, then the joint state of the combined system is the density matrix~$\rho_1 \otimes \rho_2 \in \Pos^{A_1 \times A_2}$. We make use of the properties that Kronecker products preserve positive semidefiniteness and that $\tr (A \otimes B) = \tr (A)
\, \tr(B)$. It is not always possible to write a density matrix $\rho \in \Pos^{A_1 \times A_2}$ as $\rho_1 \otimes \rho_2$ where $\rho_1 \in \Pos^{A_1}$ and $\rho_2 \in \pos^{A_2}$,
or more generally, as a convex combination of such Kronecker products.
In the latter case, the state is said to be \emph{entangled\/}, and
otherwise, it is said to be \emph{unentangled\/}.

\comment{However, there is a way to describe the ``part" of the state which is in ${A_1}$ using the notion of \emph{partial trace}.
}

We typically consider systems consisting of two-dimensional particles
(qubits), but it is sometimes convenient to work with higher dimensional
particles. Since higher dimensional spaces may be viewed as subspaces of
suitable tensor powers of~$\C^2$, we continue to describe such systems
in terms of qubits.

\paragraph{Partial trace} \quad \\

The \emph{partial trace over ${A_1}$} is the unique linear transformation $\tr_{{A_1}}: \Herm^{A_1 \times A_2} \to \Herm^{A_2}$, which satisfies
\[ \tr_{A_1}(\rho_1 \otimes \rho_2) = \tr(\rho_1) \cdot \rho_2, \] for all $\rho_1 \in \Herm^{A_1}$ and $\rho_2 \in \Herm^{A_2}$. More explicitly, given any matrix $X \in \pos^{A_1 \times A_2}$ we define $\tr_{A_1}$ as
\[ \tr_{A_1}(X) := \sum_{x_1 \in A_1} \left( e_{x_1}^* \otimes \id_{A_2} \right) X \left( e_{x_1} \otimes \id_{A_2} \right), \]
where $\set{ e_{x_1}: x_1 \in A_1}$ is the standard basis for $\C^{A_1}$. In
fact, the definition is independent of the choice of basis, so long as
it is orthonormal.
Note that the partial trace is positive, i.e., $\tr_{A_1} (X) \in \pos^{A_2}$ when $X \in \pos^{A_1 \times A_2}$, and also trace-preserving.
(In fact, it is a \emph{completely positive\/} operation.) This ensures that the image
of any density matrix under this operation,
called its \emph{reduced state\/}, is a well-defined density matrix.

Consider the scenario where two parties, Alice and Bob, hold parts of a
quantum system which are jointly in some state~$\rho$, i.e., they ``share'' a quantum state~$\rho$ over the space $\C^A
\otimes \C^B$.
Then the partial trace of~$\rho$ over one space characterizes the quantum state
over the remaining space (if we are interested only in operations on the
latter space). For example, $\tr_{A}(\rho)$ is the density matrix representing Bob's half of the state and $\tr_{B}(\rho)$ represents Alice's half. Note that $\rho$
may not equal $\tr_B(\rho) \otimes \tr_A(\rho)$ in general.

Suppose we are given the density matrix $\rho \in \pos^{A}$. We call the pure state $\psi \in \C^A \otimes \C^B$ a \emph{purification} of $\rho$ if $\tr_{B} \left( \psi \psi^* \right) = \rho$. A purification always exists if $|B| \geq |A|$,
and is in general not unique.
An important property of purifications of the same state is that they
are related to each other by a unitary operator:
if $\tr_{B} \left( \psi \psi^* \right) = \tr_{B} \left( \phi \phi^* \right)$, then there exists a unitary $U$ acting on $\C^B$ alone such that $\psi = (\id_{A} \otimes U) \, \phi$.

The partial trace operation is the quantum analogue of calculating marginal probability distributions. Consider
the linear operator~$\tr_A : \R^{A \times B} \rightarrow \R^B$ defined
by
\[ [\tr_{A}(v)]_y = \sum_{x \in A} v_{x,y} \enspace, \]
for~$y \in B$.  This is called the partial trace over~$A$.
Note that~$\tr_A(p)$ gives the marginal distribution over $B$
 of the probability distribution~$p \in \Prob^{A \times B}$.
One may view probability distributions as diagonal positive semidefinite matrices with unit trace. Then, taking the partial trace (as defined for quantum states) corresponds exactly to the computation of marginal distributions.

\paragraph{Distance measures} \quad \\

Notions of distance between quantum states and probability distributions are very important in quantum cryptography. Here, we discuss two measures used in this paper and how they are related.

We define the \emph{fidelity} of two nonnegative vectors $p,q \in \R_+^A$ as
\[ \rF(p,q) := \left( \sum_{x \in A} \sqrt{p_x} \sqrt{q_x} \right)^2 \]
and the fidelity of two positive semidefinite matrices $\rho_1$ and $\rho_2$ as
\[ \rF(\rho_1, \rho_2) := \norm{\sqrt{\rho_1}\sqrt{\rho_2}}_*^2. \]
Notice, $\rF(\rho_1, \rho_2) \geq 0$ with equality if and only if $\inner{\rho_1}{\rho_2} = 0$ and, if $\rho_1$ and $\rho_2$ are quantum states, $\rF(\rho_1, \rho_2) \leq 1$ with equality if and only if $\rho_1 = \rho_2$. An analogous statement can be made for fidelity over probability vectors.

Fidelity has several useful properties, which we later use in this paper. We have occasion to consider fidelity only of probability distributions,
and state the properties in terms of these.
However, the following properties hold for quantum states as well.
Fidelity is symmetric, positively homogeneous in both arguments, i.e., $\lambda \, \rF(p,q) = \rF(\lambda p, q) = \rF(p, \lambda q)$ for all $\lambda > 0$, and
is concave, i.e., $\rF \left( \sum_{i=1}^n \lambda_i \, p_i, q \right) \geq \sum_{i=1}^n \lambda_i \, \rF \left( p_i, q \right)$, for all $\lambda \in \Prob^n$.

Another distance measure is the \emph{trace distance}. We define the trace distance between two probability vectors $p$ and $q$, denoted $\Delta(p, q)$, as
\[ \Delta(p, q) := \half \norm{p - q}_1. \]
We similarly define the trace distance between two quantum states $\rho_1$ and $\rho_2$ as
\[ \Delta(\rho_1, \rho_2) := \frac{1}{2} \norm{\rho_1 - \rho_2}_*. \]
Notice $\Delta(\rho_1, \rho_2) \geq 0$ with equality if and only if $\rho_1 = \rho_2$, and $\Delta(\rho_1, \rho_2) \leq 1$ with equality if and only if $\inner{\rho_1}{\rho_2} = 0$. An analogous statement can be made for the trace distance between probability vectors.

It is straightforward to show that for any~$\Pi$ with~$0 \preceq \Pi \preceq
\id$,
\begin{equation}
\label{eqn-trbound}
\tr(\Pi(\rho_1 - \rho_2)) \quad \le \quad \Delta(\rho_1,\rho_2)
\enspace.
\end{equation}

We now discuss two important notions in quantum cryptography. The first is how easily two states can be distinguished from each other. For example, if Alice gives to Bob one of two states $\rho_1$ or $\rho_2$ chosen uniformly at random, then Bob can measure to learn whether he has been given $\rho_1$ or $\rho_2$ with maximum probability
\[ \frac{1}{2} + \frac{1}{4} \norm{\rho_1 - \rho_2}_* = \frac{1}{2} + \frac{1}{2} \Delta(\rho_1, \rho_2). \]
The second notion is \emph{quantum steering}. Suppose Alice has given to Bob the $A_1$ part (i.e.,
the subsystem~$A_1$ of qubits) of $\phi \in \C^{A_1 \times A_2}$. Now suppose she wants to modify
and send the $A_2$ part in a way so as to convince Bob that a different state was sent, say $\psi \in \C^{A_1 \times A_2}$. Her most general strategy is to apply a quantum operation on $A_2$ (i.e.,
a sequence of unitary operations and measurements) before sending it to Bob. If
Bob measures according to the POVM~$(\psi\psi^*, \id - \psi\psi^*)$, Alice can convince him that the state is $\psi$ with maximum probability
\[ \rF(\tr_{A_2} (\psi \psi^*), \tr_{A_2} (\phi \phi^*)) \enspace. \]

Trace distance and fidelity are closely related. The Fuchs-van de Graaf inequalities \cite{FvdG99} illustrate this relationship:
\begin{proposition}
\label{thm-fvdg}
For any finite dimensional quantum states~$\rho_1,\rho_2 \in \Pos^D$, we
have
\[ 1 - \sqrt{\rF(\rho_1, \rho_2)}
    \quad \leq \quad \Delta (\rho_1, \rho_2)
    \quad \leq \quad \sqrt{1 - \rF(\rho_1, \rho_2)}
\enspace. \]
\end{proposition}


\subsection{Semidefinite programming}\label{ssect:opt}

A natural model of optimization when studying quantum information is {semidefinite programming}. A \emph{semidefinite program}, abbreviated
as $\SDP$, is an optimization problem of the form

\[ \begin{array}{rrrcllllllllllllll}
\textrm{(P)} & \sup                         & \inner{C}{X} \\
                     & \textrm{subject to} & \calA(X) & = & b, \\
                     &                                   & X & \in & \Pos^n, \\
\end{array} \]
where $\calA: \Herm^n \to \R^m$ is linear, $C \in \Herm^n$, and $b \in \R^m$.
The SDPs that arise in quantum computation involve optimization
over complex matrices.
However, they may be transformed to the above standard form in a
straightforward manner, by observing that Hermitian matrices form
a real subspace of the vector space of~$n \times n$ complex matrices.
We give direct arguments as to why we may restrict ourselves to SDPs
over real matrices when they arise in this article.

Similar to linear programs, every SDP has a dual. We can write the dual of (P) as

\[ \begin{array}{rrrcllllllllllllll}
\textrm{(D)} & \inf                         & \inner{b}{y} \\
                     & \textrm{subject to} & \calA^*(y) - S & = & C, \\
                     &                                   & S & \in & \Pos^n, \\
\end{array} \]
where $\calA^*$ is the adjoint of $\calA$. We refer to (P) as the primal problem and to (D) as its dual. We
say~$X$ is \emph{feasible\/} for (P) if it satisfies the
constraints~$\calA(X) = b$ and~$X \in \Pos^n$, and~$(y,S)$ is feasible
for~(D) if~$\calA^*(y) - S = C, S \in \Pos^n$.
The usefulness of defining the dual in the above manner is apparent in the following lemmas.

\begin{lemma}[Weak duality]
For every $X$ feasible for $\textup{(P)}$ and $(y,S)$ feasible for $\textup{(D)}$ we have
\[ \inner{C}{X} \leq \inner{b}{y}. \]
\end{lemma}

Using weak duality, we can prove bounds on the optimal objective value of $\textup{(P)}$ and $\textup{(D)}$, i.e., the objective function value of any primal feasible solution yields a lower bound on $\textup{(D)}$ and the objective function value of any dual feasible solution yields an upper bound on $\textup{(P)}$.

Under mild conditions, we have that the optimal objective values of $\textup{(P)}$ and $\textup{(D)}$ coincide.

\begin{lemma}[Strong duality]
If the objective function of $\textup{(P)}$
is bounded from above
on the set of feasible solutions of $\textup{(P)}$ and there exists a strictly feasible solution, i.e., there exists $\bar{X} \succ 0$ such that $\calA(\bar{X}) = b$, then $\textup{(D)}$ has an optimal solution and the optimal objective values of $\textup{(P)}$ and $\textup{(D)}$ coincide.
\end{lemma}
A strictly feasible solution as in the above lemma is also called a
\emph{Slater point}.

Semidefinite programming has a powerful and rich duality theory and the interested reader is referred to $\cite{SDP, TW08}$ and the references therein.

\subsubsection{Second-order cone programming}

The \emph{second-order cone} (or \emph{Lorentz cone}) in~$\R^n$, $n \ge
2$, is defined as
\[ \Lor^n := \set{(x,t) \in \R^{n} : t \geq \norm{x}_2}. \]
A \emph{second-order cone program}, denoted SOCP, is an optimization problem of the form
\[ \begin{array}{rrrcllllllllllllll}
\textrm{(P)} & \sup                         & \inner{c}{x} \\
                     & \textrm{subject to} & Ax & = & b, \\
                     &                                   & x & \in & \Lor^{n_1} \oplus \cdots \oplus \Lor^{n_k}, \\
\end{array} \]
where $A$ is an $m \times (\sum_{i=1}^k n_k)$ matrix, $b \in \R^m$, and $k$ is finite. We say that a feasible solution $\bar{x}$ is strictly feasible if $\bar{x}$ is in the interior of $\Lor^{n_1} \oplus \cdots \oplus \Lor^{n_k}$.

An SOCP also has a dual which can be written as
\[ \begin{array}{rrrcllllllllllllll}
\textrm{(D)} & \inf                         & \inner{b}{y} \\
                     & \textrm{subject to} & A^\transpose y - s & = & c, \\
                     &                                   & s & \in & \Lor^{n_1} \oplus \cdots \oplus \Lor^{n_k}. \\
\end{array} \]
Note that weak duality and strong duality also hold for SOCPs for the properly modified definition of a strictly feasible
solution.

A related cone, called the \emph{rotated second-order cone}, is defined as
\[ \RL^n := \set{(a,b,x) \in \R^{n} : a, b \geq 0, \, 2 a b \geq \norm{x}_2^2}. \]
Optimizing over the rotated second-order cone is also called second-order cone programming because $(x,t) \in \Lor^{n}$ if and only if $(t/2,t,x) \in \RL^{n+1}$ and $(a,b,x) \in \RL^n$ if and only if $(x,a,b,a+b) \in \Lor^{n+1}$
and~$a,b \ge 0$. In fact, both second-order cone constraints can be cast as positive semidefinite constraints:
\[ t \geq \norm{x}_2 \iff \left[ \begin{array}{cc} t & x^\transpose \\ x & t \, \id \end{array} \right] \succeq 0 \quad \textup{ and } \quad a,b \geq 0, \, 2ab \geq \norm{x}_2^2 \iff \left[ \begin{array}{cc} 2a & x^\transpose \\ x & b \, \id \end{array} \right] \succeq 0. \]

There are some notable differences between semidefinite programs and second-order cone programs. One is that the algorithms for solving second-order cone programs can be  more efficient and robust than those for solving semidefinite programs. We refer the interested reader to~\cite{Stu99, Stu02, Mit03, AG03} and the references therein.

\comment{
Mittelmann, H. D.
An independent benchmarking of SDP and SOCP solvers.
Computational semidefinite and second order cone programming:
the state of the art.
Math. Program. 95 (2003), no. 2, Ser. B, 407--430

Sturm, J.F. (1999): Using SeDuMi 1.02,
a MATLAB toolbox for optimization over symmetric cones.
Optimization Methods and Software 11, 625--653

Sturm, Jos F.
Implementation of interior point methods for mixed semidefinite
and second order cone optimization problems.
Optim. Methods Softw. 17 (2002), no. 6, 1105--1154.

F. Alizadeh and D. Goldfarb, Second-order cone programming,
Math. Program., 95 (2003), pp. 3-51.
} 


\section{Coin-flipping protocols}

\subsection{Strong coin-flipping}
\label{sect:CF}

A strong coin-flipping protocol is a two-party quantum
\emph{communication protocol\/} in the style of Yao~\cite{Yao93}.
We concentrate on a class of communication protocols relevant to
coin-flipping.  Informally, in such protocols, two parties Alice and Bob
hold some number of qubits; the qubits with each party are initialized
to a fixed pure state. The initial joint state is therefore unentangled
across Alice and Bob. The two parties then ``play'' in turns.
Suppose it is Alice's turn to play. Alice applies a unitary transformation on
her qubits and then sends one or more qubits to Bob. Sending qubits does
not change the overall superposition, but rather changes the ownership of the
qubits. This allows Bob to apply his next unitary transformation on the
newly received qubits.  At the end of the protocol, each player
makes a measurement of their qubits and announces the outcome as the result
of the protocol.

Formally, the players Alice and Bob, hold some number of qubits, which
initially factor into a tensor product~$\C^{A_0} \tensor \C^{B_0}$ of Hilbert
spaces. The qubits corresponding to~$\C^{A_0}$ are in Alice's possession,
and those in~$\C^{B_{0}}$ are in Bob's possession. When the
protocol starts, the qubits in~$\C^{A_0}$ are initialized to some
superposition~$\psi_{\rA,0}$ and those in~$\C^{B_0}$ to~$\psi_{\rB,0}$,
both of which specified by the protocol.
The communication consists of~$t \ge 1$ alternations of message
exchange (``rounds''), in which the two players ``play''. Either party may play
first. The protocol specifies a factorization of the joint state space
just before each round, corresponding to the ownership of the qubits.
In the~$i$th round, $i \geq 1$, suppose it is Alice's turn to
play. Suppose the factorization of the state space just before the~$i$th
round is~$\C^{A_{i-1}} \tensor \C^{B_{i-1}}$. Alice
applies a unitary operator~$U_{\rA,i}$ to the qubits in~$\C^{A_{i-1}}$.
Then, Alice sends some of her qubits to Bob. Formally, the
space~$\C^{A_{i-1}}$ is expressed as~$\C^{A_{i}} \tensor \C^{M_i}$,
where~$\C^{A_{i}}$ is Alice's state space after the~$i$th message
is sent and $\C^{M_i}$ is the state space for the~$i$th message.
Consequently, Bob's state space after receiving the~$i$th message
is~$\C^{B_{i}} = \C^{M_i} \tensor \C^{B_{i-1}}$. In the next round,
Bob may thus apply a unitary operation to the qubits previously in Alice's
control.

At the end of the~$t$ rounds of play, Alice and Bob observe the qubits in
their possession according to some measurement. The outcomes of these
measurements represent their outputs.
We emphasize that there are no measurements until all rounds of
communication are completed. A protocol with
intermediate measurements may be transformed into this form by appealing
to standard techniques~\cite{BernsteinV97}.

\begin{definition}[Strong coin-flipping]
\label{def-scf}
A \emph{strong coin-flipping protocol\/} is a two-party communication
protocol as described above, in which the measurements of Alice and
Bob are given by three-outcome POVMs $(\Pi_{\rA,0}, \Pi_{\rA,1},
\Pi_{\rA, \abort})$ and~$(\Pi_{\rB,0}, \Pi_{\rB,1}, \Pi_{\rB, \abort})$,
respectively. When both parties follow the protocol, they do not abort,
i.e., only get outcomes in~$\set{0,1}$.
Further, each party outputs the same bit $c \in \zo$ and each binary
outcome occurs with probability 1/2.
\end{definition}

\comment{
A protocol with four messages is depicted in Figure~\ref{fig-scf}.

\begin{figure}[!h]
\begin{center}
  \unitlength=3pt

\gasset{frame=false}
  \begin{gpicture}(80, 25)

    \gasset{Nw=20,Nh=10,Nmr=00, Nframe=n}

    \node[Nw=50,Nh=10,Nmr=00](Alice)(0,1600){}

    \node(calA)(-30,130){{$\C^{A}$}}
    \node(calM)(0,130){{$\C^{M}$}}
    \node(calB)(30,130){{$\C^B$}}

    \node(PsiA)(-30,120){{$\psi_{\rA,0}$}}
    \node(PsiM)(0,120){{$\psi_{M,0}$}}
    \node(PsiB)(30,120){{$\psi_{\rB,0}$}}

    \gasset{Nw=20,Nh=10,Nmr=00, Nframe=y}

    \node[Nw=50,Nh=10,Nmr=00](U1)(-15,100){{$U_{\rA,1}$}}
    \node[Nw=50,Nh=10,Nmr=00](U3)(-15,60){{$U_{\rA,3}$}}

    \node[Nw=50,Nh=10,Nmr=00](U2)(15,80){{$U_{\rB,2}$}}
    \node[Nw=50,Nh=10,Nmr=00](U4)(15,40){{$U_{\rB,4}$}}

    \gasset{Nw=20,Nh=10,Nmr=00, Nframe=n}

    \node[Nw=50,Nh=10,Nmr=00](Alice)(0,1600){}

    \node(PiA)(-30,20){{$\{ \Pi_{\rA,0}, \Pi_{\rA,1}, \Pi_{\rA, \textup{abort}} \}$}}
    \node(PiB)(30,20){{$\{ \Pi_{\rB,0}, \Pi_{\rB,1}, \Pi_{\rB, \textup{abort}} \}$}}

    \drawedge[sxo=0,exo=-15, ELpos=30](PsiA,U1){}
    \drawedge[sxo=0,exo=15, ELpos=30](PsiM,U1){}
    \drawedge[sxo=0,exo=15, ELpos=30](PsiB,U2){}

    \drawedge[sxo=-15,exo=-15, ELpos=30](U1,U3){}

    \drawedge[sxo=15,exo=15, ELpos=30](U2,U4){}

    \drawedge[sxo=15,exo=-15, ELpos=30](U1,U2){}
    \drawedge[sxo=-15,exo=15, ELpos=30](U2,U3){}
    \drawedge[sxo=15,exo=-15, ELpos=30](U3,U4){}

    \drawedge[sxo=-15,exo=0, ELpos=30](U3,PiA){}
    \drawedge[sxo=15,exo=0, ELpos=30](U4,PiB){}

  \end{gpicture}
\end{center}
\caption{A strong coin-flipping protocol with four messages, all of
which have the same dimension.}
\label{fig-scf}
\end{figure}
} 

We are interested in the probabilities of the different outcomes in a
coin-flipping protocol, when either party ``cheats''. Suppose Alice and
Bob have agreed upon a protocol, i.e., a set of rules for the state
initialization, communication, quantum operations, and
measurements. What if Alice or Bob do not follow protocol? Suppose Alice
is dishonest and would like to force an outcome of $0$. She may
use a different number of qubits for her private operations, so that her
space $\C^{A'_i}$ may be much larger than
$\C^{A_i}$. She may create any initial state she wants. During the
communication, the only restriction is that she send a state of the
correct dimension, e.g., if the protocol requires a message with~$3$
qubits in the first message, then Alice sends $3$ qubits. Between messages,
she may apply any quantum operation she wants on the qubits in her
possession.
At the end of the protocol, she may use a different measurement of her
choice. For example, she may simply output ``$0$'' as this is her desired
outcome (which corresponds to a trivial measurement).
The rules that Alice chooses to follow instead of the protocol
constitute a \emph{cheating strategy\/}.

We would like to quantify the extent to which a cheating player can
convince an honest one of a desired outcome, so we focus on runs of the
protocol in which at most one party is dishonest. We analyze in this paper
the maximum probability with which Alice (or Bob) can force a
desired outcome in terms of the ``bias'', i.e., the advantage
over~$1/2$ that a cheating party can achieve.

\begin{definition}[Bias]
\label{def-bias}
For a given strong coin-flipping protocol, for each~$c \bit$, define
\begin{itemize}
\item $P_{\rA,c}^* := \sup \; \{\Pr[\text{honest Bob outputs } c \text{ when Alice may cheat}]\}$,
\item $P_{\rB,c}^* := \sup \; \{\Pr[\text{honest Alice outputs } c \text{ when Bob may cheat}]\}$,
\end{itemize}
where the suprema are taken over all cheating strategies of the dishonest
player. The bias~$\epsilon$ of the protocol is defined as
\[
\epsilon := \max\{P_{\rA,0}^*, P_{\rA,1}^*, P_{\rB,0}^*, P_{\rB,1}^* \} - 1/2
\enspace.
\]
\end{definition}

\comment{ 
Strong coin-flipping protocols are designed to protect Alice and Bob from biasing either outcome. However, sometimes all that is needed is to protect Alice from biasing the outcome toward ``$0$'' and Bob from biasing the outcome toward ``$1$'', i.e., they desire opposing outcomes. This variant is called \emph{weak coin-flipping} and is defined below.

\begin{definition}[Weak coin-flipping]
A weak coin-flipping protocol with \emph{bias} $\varepsilon > 0$ is a protocol that satisfies:
\begin{itemize}
\item If both parties are honest, i.e, follow the protocol, they do not abort and each party outputs the same bit $c \in \zo$ and each outcome occurs with probability 1/2,
\item $P_{\rA,0}^* := \sup \{\Pr[\text{honest Bob outputs } 0 \text{ when Alice may cheat}]\}$,
\item $P_{\rB,1}^* := \sup \{\Pr[\text{honest Alice outputs } 1 \text{ when Bob may cheat}]\}$,
\item $\eps := \max\{P_{\rA,0}^*, P_{\rB,1}^* \} - 1/2$,
\end{itemize}
where the suprema are taken over all cheating strategies.
\end{definition}
} 

\subsection{An example protocol} \label{ex}

Here we describe a construction of strong coin-flipping protocols based on
quantum \emph{bit-commitment\/}~\cite{ATVY00, Amb01, SR01, KN04} that
consists of three messages.
First, Alice chooses a uniformly random bit $a$, creates a state of the form
\[ \psi_a \in \C^A \otimes \C^{A'} \]
and sends $A$ to Bob, i.e., the first message consists of qubits
corresponding to the space~$\C^A$. (For ease of exposition, we use this
language throughout,
i.e., refer to qubits by the labels of the corresponding spaces.)
This first message is the \emph{commit stage} since she potentially
gives some information about the bit $a$, for which she may be held
accountable later. Then Bob chooses a uniformly random bit $b$ and sends it to Alice. Alice then sends $a$ and $A'$ to
Bob. Alice's last message is the \emph{reveal stage}. Bob checks to see if the qubits
he received  are in state $\psi_a$ (we give more details about this step below). If Bob is convinced that the state is correct, they both output $0$ when $a=b$, or $1$ if $a \neq b$, i.e., they output the XOR of $a$ and $b$.

This description can be cast in the form of a quantum protocol as
presented in Section~\ref{sect:CF}: we can encode $0$ as basis state
$e_0$ and $1$ as $e_1$,  we can simulated the generation
of a uniformly random bit
by preparing a uniform superposition over the two basis states, and we
can ``send'' qubits by permuting their order (a unitary operation) so that they
are part of the message subsystem.
In fact, we can encode an entirely classical protocol using a quantum
one in this manner.  A more general protocol of this kind is described
formally in Section~\ref{family}.

We present a protocol from~\cite{KN04} which follows the above framework.

\begin{definition}[Coin-flipping protocol example] \quad \\
Let $A := \{ 0, 1, 2 \}$, $A' := A$, and let $\C^A$ and $\C^{A'}$ be spaces for Alice's two messages.
\begin{itemize}
\item Alice chooses $a \in \zo$ uniformly at random and creates the state
\[ \psi_a = \frac{1}{\sqrt{2}} \, e_a \otimes e_a + \frac{1}{\sqrt{2}} \, e_2 \otimes e_2 \in \C^A \otimes \C^{A'}, \]
where $\{ e_0, e_1, e_2 \}$ are standard basis vectors. Alice sends the $A$ part of $\psi_a$ to Bob.
\item Bob chooses $b \in \zo$ uniformly at random and sends it to Alice.
\item Alice reveals $a$ to Bob and sends the rest of $\psi_a$, i.e., she sends ${A'}$.
\item Bob checks to see if the state sent by Alice is $\psi_a$, i.e., he checks to see if Alice has tampered with the state during the protocol. The measurement on $\C^A \otimes \C^{A'}$ corresponding to this check is
\[ ( \Pi_{\accept} := \psi_a \psi_a^*, \quad \Pi_{\abort} := \id - \Pi_{\accept} ). \]
If the measurement outcome is ``$\abort$'' then Bob aborts the protocol.
\item Each player outputs the $\mathrm{XOR}$ of the two bits, i.e.,
Alice outputs~$a \oplus b'$, where~$b'$ is the bit she received in
the second round, and if he does not abort, Bob outputs~$a' \oplus b$,
where~$a'$ is the bit received by him in the third round.
\end{itemize}
\end{definition}

In the honest case, Bob does not abort since $\inner{\Pi_{\abort}}{\psi_a \psi_a^*}
= 0$. Furthermore, Alice and Bob get the same outcome which is uniformly random. Therefore, this is a well-defined coin-flipping protocol. We now sketch a proof that this  protocol has bias $\epsilon = 1/4$.

\paragraph{Bob cheating:} We consider the case when Bob cheats
towards~$0$; the analysis of cheating towards $1$ is similar.
If Bob wishes to maximize the probability of outcome~$0$,
he has to maximize
the probability that the bit~$b$ he sends equals~$a$. In other words,
he may only cheat by measuring Alice's first message to try to learn $a$,
then choose $b$ suitably to force the desired outcome.
Define $\rho_a := \tr_{A'} \left( \psi_a \psi_a^* \right)$. This is the
reduced state of the~$A$-qubits Bob has after the first message.
Recall Bob can learn the value of $a$ with probability
\[ \half + \frac{1}{2} \Delta(\rho_0, \rho_1) = 3/4 \enspace, \]
and this bound can be achieved.
This strategy is independent of the outcome Bob desires, thus $P_{\rB,0}^* = P_{\rB,1}^* = 3/4$.

\paragraph{Alice cheating:} Alice's most general cheating strategy is to send a state in the first message such that she can decide the value of $a$ after receiving $b$,
and yet pass Bob's cheat detection step with maximum probability.  For example, if Alice wants outcome $0$ then she returns $a = b$ and if she wants outcome $1$, she returns $a = \bar{b}$. Alice always gets the desired outcome as long as Bob does not detect her cheating. As a primer for more complicated protocols, we show an SDP formulation for a cheating Alice based on the above cheating strategy description. There are three important quantum states to consider here. The first is Alice's first message, which we denote as $\sigma \in \Pos^{A}$. The other two states are the states Bob has at the end of the protocol depending on whether $b=0$ or $b=1$, we denote them by $\sigma_b \in \Pos^{A \otimes A'}$. Note that $\tr_{A'}(\sigma_0) = \tr_{A'}(\sigma_1) = \sigma$ since they are consistent with the first message $\sigma$---Alice does not know $b$ when $\sigma$ is sent. However, they could be different on $A'$ because Alice may apply some quantum operation depending
upon~$b$ before sending the $A'$ qubits.  Then Alice can cheat with probability given by the optimal objective value of the following SDP:
\[ \begin{array}{rrrcllllllllllllll}
& \quad \sup                         & \half \langle \psi_0 \psi_0^* , \sigma_0 \rangle & + & \half \langle \psi_1 \psi_1^* , \sigma_1 \rangle \\
                     & \textup{subject to} & \tr_{A'}(\sigma_b) & = & \sigma, & \forAll b \in \{ 0,1 \}, \\
                       &      & \tr(\sigma) & = & 1, \\
                            &      & \sigma & \in & \Pos^{A}, \\
                                                    &      & \sigma_b & \in & \Pos^{A \otimes A'}, & \forAll b \bit,
\end{array} \]
recalling that the partial trace is trace-preserving, any unit trace, positive semidefinite matrix represents a valid quantum state,
and that two purifications of the same density matrix are related to
each other by a unitary transformation on the part that is traced out.

A few words about the above optimization problem are in order here.
First, the restriction to real positive semidefinite matrices does not
change the optimum: the real part of any feasible set of complex
matrices~$\sigma, \sigma_0, \sigma_1$ is also feasible, and has the
same objective function value. Second, using straightforward transformations,
we may verify that the problem is an SDP of the form defined in
Section~\ref{ssect:opt}.

It has been shown~\cite{SR01,Amb01,NS03} that the optimal objective value of this problem is
\[ \half + \half \sqrt{\rF(\rho_0, \rho_1)} = 3/4 \]
given by the optimal solution $(\sigma_0, \sigma_1, \sigma) = (\psi \psi^*, \psi \psi^*, \tr_{A'}(\psi \psi^*))$, where
\[ \psi = \sqrt{\frac{1}{6}} \, e_0 \otimes e_0 + \sqrt{\frac{1}{6}} \,  e_1 \otimes e_1 + \sqrt{\frac{2}{3}} \, e_2 \otimes e_2
\enspace. \]
Therefore, the bias of this protocol is $\max \{ P_{\rA,0}^*, P_{\rA,1}^*, P_{\rB,0}^*, P_{\rB,1}^* \} - 1/2 = 3/4 - 1/2 = 1/4$.
Using Proposition~\ref{thm-fvdg}, it was shown in~\cite{Amb01} that for any $\rho_0$ and $\rho_1$, we have 
\[
\max \set{\half + \half \sqrt{\rF(\rho_0, \rho_1)}, \half + \frac{1}{2} \Delta(\rho_0, \rho_1)} - 1/2 \quad \geq \quad  1/4
\enspace.
\]
Thus, we cannot improve the bias by simply changing the starting states in this type of protocol, suggesting a substantial change of the form of the protocol is necessary to find a smaller bias.

\subsection{A family of protocols} \label{family}

We now consider a family of protocols which generalizes the above idea. Alice and Bob each flip a coin and commit
to their respective bits by exchanging quantum states. Then they reveal
their bits and send the remaining part of the commitment state.
Each party checks the received state against the one they expect, and
abort the protocol if they detect an inconsistency. They output the
XOR of the two bits otherwise. We see that this is uniformly random,
when $a$ and $b$ are uniformly random.

The difficulty in designing a good protocol is in deciding how Alice and Bob
commit
to their bits. If Alice or Bob leaks too much information early, then the other party has more freedom to cheat. Thus, we try to maintain a balance between the two parties so as to minimize the bias they can achieve by cheating.

Consider the following Cartesian product of finite sets $A = A_1 \times \cdots \times A_n$. These are used for Alice's first $n$ messages to Bob. Suppose we are given two probability distributions $\alpha_0, \alpha_1 \in \Prob^A$. Define the following two quantum states
\[ \psi_a = \sum_{x \in A} \sqrt{\alpha_{a,x}} \, e_x \otimes e_x \in \C^A \otimes \C^{A'} \quad \text{ for } \quad a \bit, \]
where $A' = A$. The reason we define the state over $\C^A$ and a copy is because in the protocol, Alice sends states in $\C^A$ while retaining copies in $\C^{A'}$ for herself. We
may simulate Alice's choice of uniformly random~$a$ and the
corresponding messages by preparing the initial state
\[ \psi := \sum_{a \bit} \frac{1}{\sqrt 2} \, e_a \otimes e_a \otimes \psi_a \in \C^{A_0} \otimes \C^{A'_0} \otimes \C^A \otimes \C^{A'}, \]
where $A_0 = A'_0 = \{ 0, 1 \}$ are used for two copies of Alice's bit $a$, one for Bob and a copy for herself.

We now describe the setting for Bob's messages. Consider the following Cartesian  product of finite sets $B = B_1 \times \cdots \times B_n$ used for Bob's first $n$ messages to Alice. Suppose we are given two probability distributions $\beta_0, \beta_1 \in \Prob^B$. Define the following two quantum states
\[ \phi_b = \sum_{y \in B} \sqrt{\beta_{b,y}} \, e_y \otimes e_y \in \C^B \otimes \C^{B'} \quad \text{ for } \quad b \bit, \]
where $B' = B$. Bob's choice of uniformly random $b$, and the
corresponding messages may be simulated
by preparing the initial state
\[ \phi := \sum_{b \bit} \frac{1}{\sqrt 2} \, e_b \otimes e_b \otimes \phi_b \in \C^{B_0} \otimes \C^{B'_0} \otimes \C^B \otimes \C^{B'}, \]
where $B_0 = B'_0 = \{ 0, 1 \}$ are used for two copies of Bob's bit $b$, one for Alice and a copy for himself.

We now describe the communication and cheat detection in the protocol.

\begin{definition}[Coin-flipping protocol based on bit-commitment]
A \emph{coin-flipping protocol based on bit-commitment\/} is specified
by a $4$-tuple of probability distributions~$(\alpha_0, \alpha_1,
\beta_0, \beta_1)$ that define states~$\psi,\phi$ as above.
\begin{itemize}
\item Alice prepares the state $\psi$ and Bob prepares the state $\phi$
as defined above.
\item For $i$ from $1$ to $n$: Alice sends $\C^{A_i}$ to Bob who replies with $\C^{B_i}$.
\item Alice fully reveals her bit by sending $\C^{A'_0}$. She also sends $\C^{A'}$ which Bob uses later to check if she was honest. Bob then reveals his bit by sending $\C^{B'_0}$. He also sends $\C^{B'}$ which Alice uses later to check if he was honest.
\item Alice observes the qubits in her possession according to the measurement $(\Pi_{\rA,0}, \Pi_{\rA,1}, \Pi_{\rA, \abort})$ defined
on the space $\Pos^{A_0 \times B'_0 \times B \times B'}$, where
\[ \Pi_{\rA,0} := \sum_{b \bit} e_b e_b^* \otimes e_b e_b^* \otimes \phi_b \phi_b^*, \quad
\Pi_{\rA,1} := \sum_{b \bit} e_{\bar b} e_{\bar b}^* \otimes e_b e_b^* \otimes \phi_b \phi_b^*, \]
and $\Pi_{\rA, \abort} := \id - \Pi_{\rA,0} - \Pi_{\rA,1}$. \\
\item Bob observes the qubits in his possession according to the measurement $(\Pi_{\rB,0}, \Pi_{\rB,1}, \Pi_{\rB, \abort})$
defined on the space~$\Pos^{B_0 \times A'_0 \times A \times A'}$, where
\[ \Pi_{\rB,0} := \sum_{a \bit} e_a e_a^* \otimes e_a e_a^* \otimes \psi_a \psi_a^*, \quad
\Pi_{\rB,1} := \sum_{a \bit} e_{\bar a} e_{\bar a}^* \otimes e_a e_a^* \otimes \psi_a \psi_a^*, \]
and $\Pi_{\rB, \abort} := \id - \Pi_{\rB,0} - \Pi_{\rB,1}$. (These last two steps can be interchanged.)
\end{itemize}
\end{definition}

Note that the measurements check two things. First, they check whether the outcome is $0$ or $1$. The first two terms determine this, i.e., whether $a = b$ or if $a \neq b$. Second, they check whether the other party was honest. For
example, if Alice's measurement projects onto a subspace where $b=0$ and Bob's messages are not in
state $\phi_0$, then Alice knows Bob has cheated and aborts. A six-round protocol is depicted in Figure~\ref{fig:6Rprotocol}, on the next page. 

\begin{figure}[htbp] 
   \centering
   \includegraphics[width=5.5in]{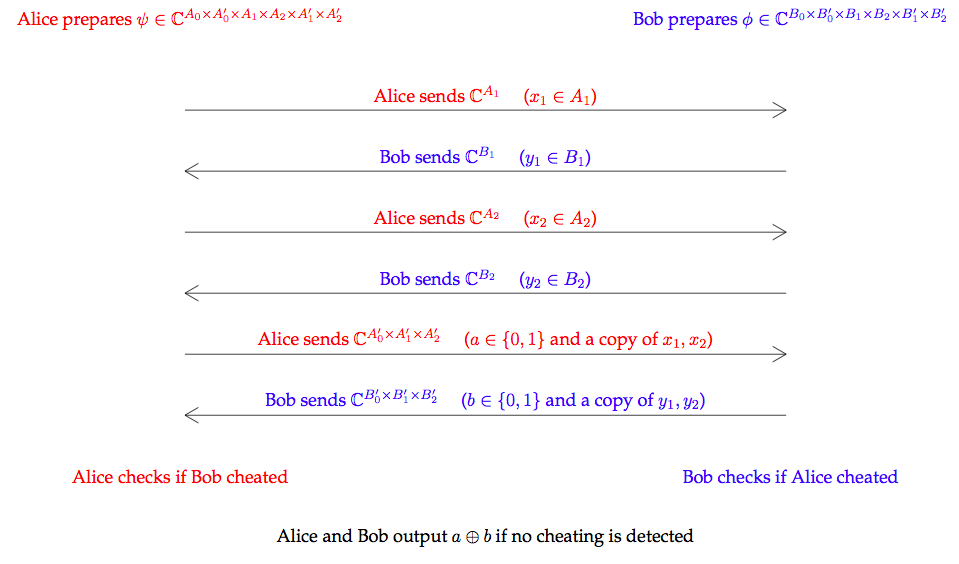} 
  \caption{Six-round coin-flipping protocol based on bit-commitment. Alice's actions are in red and Bob's actions are in blue.}
  \label{fig:6Rprotocol}
\end{figure} 

We could also consider the case where Alice and Bob choose $a$ and $b$ with different probability distributions, i.e., we could change the $1/\sqrt{2}$ in the definitions of $\psi$ and $\phi$ to other values depending on $a$ or $b$. This causes the honest outcome probabilities to not be uniformly random and this no longer falls into our definition of a coin-flipping protocol. However, sometimes such ``unbalanced'' coin-flipping protocols are useful, see~\cite{CK09}. We note that our optimization techniques in Section~\ref{sect:opt} are robust enough to handle the analysis of such modifications.

Notice that our protocol is parameterized by the four probability distributions $\alpha_0$, $\alpha_1$, $\beta_0$, and~$\beta_1$. It seems to be a very difficult problem to solve for the choice of these parameters
that gives us the least bias. Indeed, we do not even have an upper bound on the dimension
of these parameters in an optimal protocol. However, we can solve for the bias of a protocol once these parameters are fixed using the optimization techniques in Section~\ref{sect:opt}. Once we have a means for computing the bias given some choice of fixed parameters, we then turn our attention to solving for the best choice of parameters. We use the heuristics in Sections~\ref{sect:filter} and~\ref{sect:symmetry} to design an algorithm in Section~\ref{sect:algo} to search for these.


\section{Cheating strategies as optimization problems} \label{sect:opt}

In this section, we show that the optimal cheating strategy of a player
in a coin-flipping protocol is characterized by highly structured
semidefinite programs.

\subsection{Characterization by semidefinite programs}

We start by formulating strategies for cheating Bob and cheating Alice as
semidefinite optimization problems as proposed by Kitaev~\cite{Kit03}.
The extent to which Bob can cheat is captured by the following lemma.

\newpage 
\begin{lemma}
\label{thm-bob-sdp}
The maximum probability with which cheating Bob can force honest Alice to accept $c \bit$ is given by the optimal objective value of the following SDP:
\[ \begin{array}{rrrcllllllllllllll}
& \quad \sup                         & \langle \, \rho_F , \Pi_{\rA,c} \, \rangle \\
                     & \textup{subject to} & \tr_{B_1}(\rho_1) & = & \tr_{A_1} \left( \psi \psi^* \right), \\
                     &                                   & \tr_{B_j} (\rho_j) & = & \tr_{A_j} (\rho_{j-1}), \quad \forAll j \in \{ 2, \ldots, n \}, \\
                     &                                   & \tr_{B' \times B'_0}(\rho_F) & = & \tr_{A' \times A'_0}(\rho_n), \\
                     &                                   & \rho_j & \in & \pos^{A_0 \times A'_0 \times B_1 \times \cdots \times B_j \times A_{j+1} \times \cdots \times A_n \times A'}, \quad  \forAll j \in \{ 1, \ldots, n \}, \\
                     &                                   & \rho_F & \in & \pos^{A_0 \times B'_0 \times B \times B'}. \\
\end{array} \]
Furthermore, an optimal cheating strategy for Bob may be derived from an
optimal feasible solution of this SDP.
\end{lemma}

We depict Bob cheating, and the context of the SDP variables, in a six-round protocol in Figure~\ref{fig:6RBobprotocolwv}, below.

\begin{figure}[htbp] 
   \centering
   \includegraphics[width=6in]{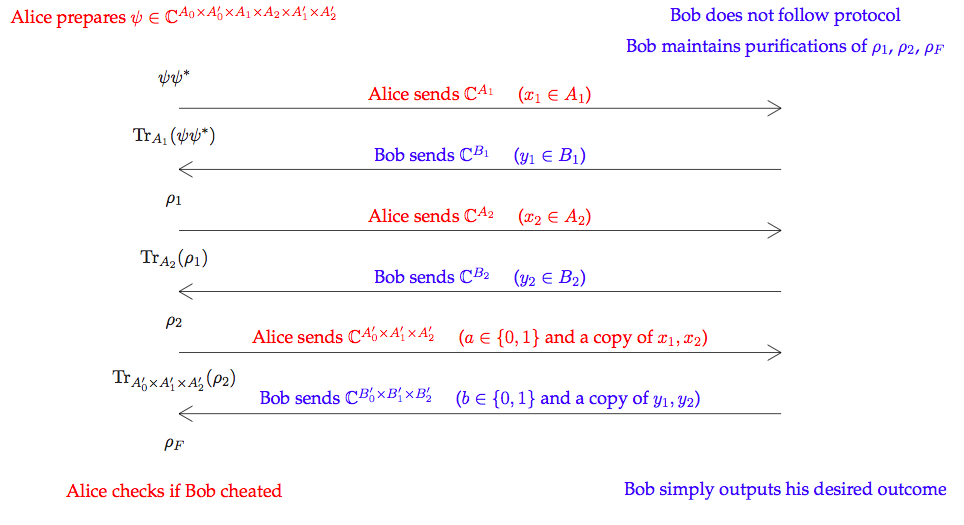} 
  \caption{Bob cheating in a six-round protocol.}
  \label{fig:6RBobprotocolwv}
\end{figure}

We call the SDP Lemma~\ref{thm-bob-sdp} Bob's cheating SDP. In a similar fashion, we can formulate Alice's cheating SDP.

\newpage 
\begin{lemma}
\label{thm-alice-sdp}
The maximum probability with which cheating Alice can force honest Bob to accept $c \bit$ is given by the optimal objective value of the following SDP:
\[ \begin{array}{rrrcllllllllllllll}
& \sup
& \inner{\sigma_F}{\Pi_{\rB,c} \otimes \id_{B'_0 \times B'}} \\
                     & \textup{subject to} & \tr_{A_1}(\sigma_1) & = & \phi \phi^*, \\
                     &                               & \tr_{A_2}(\sigma_2) & = & \tr_{B_1}(\sigma_1), \\
                     & & & \vdots \\
                     &                               & \tr_{A_n}(\sigma_n) & = & \tr_{B_{n-1}}(\sigma_{n-1}), \\
                     &                               & \tr_{A' \otimes A'_0}(\sigma_F) & = & \tr_{B_n}(\sigma_n), \\
                     &                               & \sigma_{j} & \in & \pos^{B_{0} \times B'_{0} \times A_1 \times \cdots \times A_j \times B_j \times \cdots \times B_n \times B'}, \\
                     & & & & \forAll j \in \{ 1, \ldots, n \}, \\
                     & & \sigma_{F} & \in & \pos^{B_{0} \times B'_{0} \times A'_{0} \times A \times A' \times B'}.
\end{array} \]
Furthermore, we may derive an optimal cheating strategy for Alice from
an optimal feasible solution to this SDP.
\end{lemma}

For completeness, we present proofs of these lemmas in
Appendix~\ref{sec-cheating-sdp}.

We depict Alice cheating, and the context of her SDP variables, in a six-round protocol in Figure~\ref{fig:6RAliceprotocolwv}, below. 

\begin{figure}[htbp] 
   \centering
   \includegraphics[width=6in]{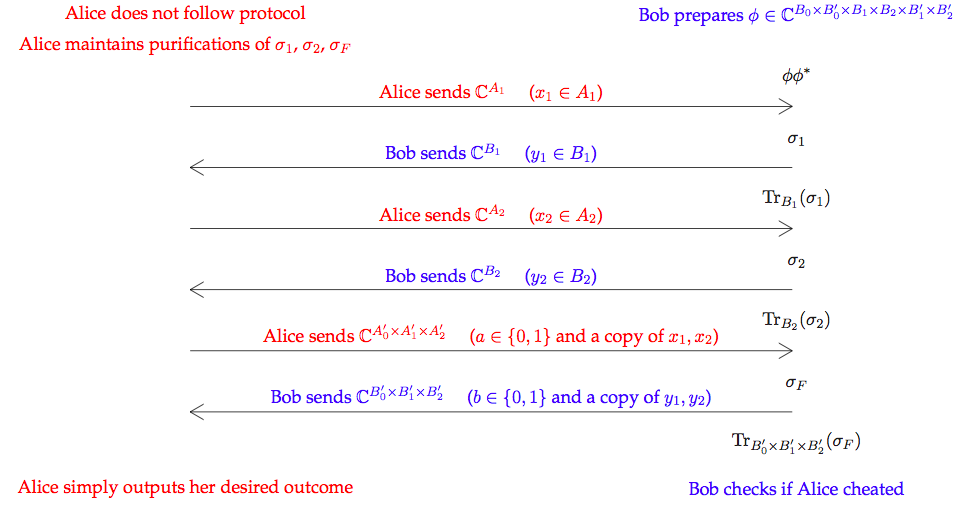} 
  \caption{Alice cheating in a six-round protocol.}
  \label{fig:6RAliceprotocolwv}
\end{figure}

Analyzing and solving these problems computationally gets increasingly difficult and time consuming as~$n$
increases, since
the dimension of the variables increases exponentially in~$n$. In the analysis
of the bias, we make use of the following results which simplify the underlying optimization problems without changing their optimal objective values.

\begin{definition}
We define \emph{Bob's cheating polytope}, denoted as $\calP_{\rB}$, as the set of all vectors
$(p_{1}, p_{2}, \ldots, p_{n})$ such that
\[ \begin{array}{rrrcllllllllllllll}

                     &  & \tr_{B_1}(p_1) & = & e_{A_{1}}, \\
                     &  & \tr_{B_2} (p_2) & = & p_{1} \otimes e_{A_{2}}, \\
                     & & & \vdots \\
                     &  & \tr_{B_n} (p_n) & = & p_{n-1} \otimes e_{A_{n}}, \\
                     & & p_j & \in & \R_+^{A_{1} \times B_{1} \times \cdots \times A_{j} \times B_{j}}, \; \forAll j \in \{ 1, \ldots, n \},
\end{array} \]
where $e_{A_j}$ denotes the vector of all ones on the corresponding space $\C^{A_j}$.
\end{definition}

We can now define a simpler ``reduced'' problem that captures Bob's optimal
cheating probability.

\begin{theorem}[Bob's Reduced Problem]
\label{thm-brp}
The maximum probability with which cheating Bob can force honest Alice to accept outcome~$c
\in \set{0,1}$ is given by the optimal objective function value of the following convex optimization problem
\[ P_{\rB,c}^* = \max \left\{ \half \sum_{a \bit} \, \rF \left( (\alpha_{a} \otimes \id_{B})^{\transpose}p_{n}, \, \beta_{a
\oplus c} \right) : (p_{1}, \ldots, p_{n}) \in \calP_{\rB} \right\}, \]
where the arguments of the fidelity functions are probability distributions over $B$.
\comment{
Furthermore, he can force an outcome of $1$ with probability given by the optimal objective function value of the following convex optimization problem
\[ P_{\rB,1}^* = \max \left\{ \half \sum_{a \bit} \, \rF \left( (\alpha_{a} \otimes \id_{B})^{\transpose}p_{n}, \, \beta_{\bar{a}} \right) : (p_{1}, \ldots, p_{n}) \in \calP_{\rB} \right\}. \]
}
\end{theorem}

The connection between the fidelity function and semidefinite programming is detailed in the next subsection. A proof of the above theorem is presented in Appendix~\ref{Alice}.

We can also define Alice's cheating polytope.

\begin{definition}
We define \emph{Alice's cheating polytope}, denoted as $\calP_{\rA}$, as the set of all vectors
$(s_{1}, s_{2}, \ldots, s_{n}, s)$ satisfying
\[ \begin{array}{rrrcllllllllllllll}
                     &                               & \tr_{A_1}(s_1) & = & 1, \\
                     &                               & \tr_{A_2}(s_2) & = & s_1 \otimes e_{B_{1}}, \\
                     & & & \vdots \\
                     &                               & \tr_{A_n}(s_n) & = & s_{n-1} \otimes e_{B_{n-1}}, \\
                     &                               & \tr_{A'_{0}}(s) & = & s_n \otimes e_{B_{n}}, \\
                     &                               & s_{1} & \in & \R_{+}^{A_{1}}, \\
                     &                               & s_{j} & \in & \R_{+}^{A_{1} \times B_{1} \times \cdots \times B_{j-1} \times A_{j}}, \; \forAll j \in \{ 2, \ldots, n \}, \\
                     & & s & \in & \R_{+}^{A'_{0} \times A \times B},
\end{array} \]
where $e_{B_j}$ denotes the vector of all ones on the corresponding space $\C^{B_j}$.
\end{definition}

Now we can define Alice's reduced problem.

\begin{theorem}[Alice's Reduced Problem]
\label{thm-arp}
The maximum probability with which
cheating Alice can force honest Bob to accept outcome $c \in \set{0,1}$ is given by the optimal objective function value of the following convex optimization problem
\[ P_{\rA,c}^* = \max \left\{ \half \sum_{a \bit} \sum_{y \in B} \beta_{a
\oplus c,y} \; \rF(s^{(a,y)}, \alpha_{a}) \; : \;
(s_{1}, \ldots, s_{n}, s) \in \calP_{\rA} \right\}, \]
where $s^{(a,y)} \in \R_+^A$ is the restriction of $s$ with the indices $(a,y)$ fixed, i.e., $[s^{(a,y)}]_x := s_{a,x,y}$.
\comment{
Similarly, she can force honest Bob to accept $1$ with probability given by the optimal objective function value of the following convex optimization problem
\[ P_{\rA,1}^* = \max \left\{ \half \sum_{a \bit} \sum_{y \in B} \beta_{\bar{a},y} \; \rF(s^{(a,y)}, \alpha_{a}) \; : \;
(s_{1}, \ldots, s_{n}, s) \in \calP_{\rA} \right \}. \]
}
\end{theorem}

We postpone a proof of the above theorem until Appendix~\ref{Alice}.

We note here that we can get similar SDPs and reductions if Alice chooses $a$ with a non-uniform probability distribution and similarly for Bob. It only changes the multiplicative factor $1/2$ in the reduced problems to something that depends on $a$ (or $b$) and the proofs are nearly identical to those in the appendix.

\comment{
\subsection{Optimizing the fidelity function using semidefinite programming}

In this subsection, we show that the reduced problems are also semidefinite programs.
}

We point out that the reduced problems are also semidefinite programs.
The containment of the variables in a polytope is captured by linear
constraints, so it suffices to express the objective function as a
linear functional of an appropriately defined positive semidefinite
matrix variable.

\begin{lemma} \label{FidelityLemma}
For any $p, q \in \R_+^{A}$, we have
\[ \rF(p,q) = \max \left\{ \inner{X}{\sqrtt{p}} : \diag(X) = q, \, X \in \pos^{A} \right\}. \]
\end{lemma}

\begin{proof}
Notice that $\bar{X} := \sqrtt{q}$ is a feasible solution to the SDP with objective function value $\rF(p,q)$. All that remains to show is that it is an optimal solution. If $p=0$, then we are done, so assume $p \neq 0$. The dual can be written as
\[ \inf \{ \inner{y}{q} : \Diag(y) \succeq \sqrtt{p}, y \in \R^A \}. \]
Define $y$, as a function of $\eps > 0$, entry-wise for each~$x \in A$
as
\[ y_x(\eps) := \left\{ \begin{array}{rcl}
(\sqrt{ \rF(p,q)} + \eps) \frac{\sqrt{p_x}}{\sqrt{q_x}} & \textup{ if } & p_x, q_x > 0, \\
\frac{(\sqrt{\rF(p,q)} + \eps) \norm{p}_1}{\eps} & \textup{ if } & q_x = 0, \\
\eps & \textup{ if } & p_x = 0, q_x > 0. \\
\end{array} \right. \]
We can check that $\inner{y(\eps)}{q} \to  \rF(p,q)$ as $\eps \to 0$, so it suffices to show  that $y(\eps)$ is dual feasible for all $\eps > 0$. For any $y > 0$,
\begin{eqnarray*}
\Diag(y) \succeq \sqrtt{p}
& \iff & \id_A \succeq \Diag(y)^{-1/2} \sqrtt{p} \Diag(y)^{-1/2} \\
& \iff & 1 \geq \sqrt{p}^{\transpose} \Diag(y)^{-1} \sqrt{p} \\
& \iff & 1 \geq \sum_{x \in A} \frac{p_x}{y_x},
\end{eqnarray*}
noting $\Diag(y)^{-1/2} \sqrtt{p} \Diag(y)^{-1/2}$ is rank $1$ so the largest eigenvalue is equal to its trace.
From this, we can check that $y(\eps)$ is feasible for all $\eps > 0$. \qed
\end{proof}

\comment{
The proof above shows that
\begin{eqnarray*}
\inf \{ \inner{y}{q} : \Diag(y) \succeq \sqrtt{p}, y \in \R^A \}
& = & \inf \{ \inner{y}{q} : \inner{y^{-1}}{p} \leq 1, \, y > 0 \} \\
& = & \inf_{y > 0} \{ \inner{y}{q} \inner{y^{-1}}{p} \},
\end{eqnarray*}
which is the classical version of Alberti's Theorem~\cite{A83}, which states that
\[ \rF (\rho, \sigma) = \inf_{X \succ 0} {\inner{X}{\rho} \inner{X^{-1}}{\sigma}}, \]
for any quantum states $\rho$ and $\sigma$.
} 

The optimization problem in Lemma~\ref{FidelityLemma} remains an SDP if we replace $q$ with a variable constrained to be in a polytope. Therefore, the reduced problems in
Theorems~\ref{thm-brp} and~\ref{thm-arp} can
be modelled as semidefinite programs.


\comment{
\subsection{Optimizing the fidelity function using second-order cone programming}
}

\subsection{SOCP formulations for the reduced problems}

In this section, we show that the reduced SDPs can be modelled using a simpler class of optimization problems,
second-order cone programs.
We elaborate on this below and explain the significance to solving these problems computationally.

We start by first explaining how to model fidelity as an SOCP. Suppose we are given the problem
\[ \max \left\{ \sqrt{\rF (p,q)} : q \in \R_+^n \cap S \right\}
= \max \left\{ \sum_{i = 1}^n {\sqrt{p_i}} \, t_i : t_i^2 \leq q_i, \, \forall i \in \{ 1, \ldots, n \}, \, q \in \R_+^n \cap S \right\}, \]
where $p \in \R_+^n$ and $S \subseteq \R^n$. We can replace $t_i^2 \leq q_i$ with the equivalent constraint $(1/2, q_i, t_i) \in \RL^3$, for all $i \in \set{1, \ldots, n}$. Therefore, we can maximize the fidelity using~$n$ rotated second-order cone constraints.

For the same reason, we can use second-order cone programming to solve a problem of the form
\[ \max \left\{ \sum_{j = 1}^m a_j \sqrt{\rF (p_j, q_j)} : (q_1, \ldots, q_m) \in \R_+^{mn} \cap S' \right\}, \]
where $a \in \R_+^m$ and $S' \subseteq \R^{mn}$.
However, this does not apply directly to the reduced problems since we need to optimize over a linear combination of fidelities and $f(x) = x^2$ is not a concave function. For example, Alice's reduced problem is of the form
\[ \max \left\{ \sum_{j = 1}^m a_j \, {\rF (p_j, q_j)} : (q_1, \ldots, q_m) \in \R_+^{mn} \cap S' \right\}. \]
The root of this problem arises from the fact that the fidelity function, which is concave, is a composition of a concave function with a convex function, thus we cannot break it into these two steps. Even though the above analysis does not work to capture the reduced problems as SOCPs, it does have a desirable property that it only uses $O(n)$ second-order cone constraints and perhaps this formulation will be useful for future applications.

We now explain how to model the reduced problems as SOCPs directly.

\begin{lemma} \label{soclemma}
For $p,q \in \R_+^n$, we have
\[ \rF (p,q) = \max \set{ \frac{1}{\sqrt 2} \sum_{i,j = 1}^n \sqrt{p_i p_j} \, t_{i,j} : \left( q_i, q_j, t_{i,j} \right) \in \RL^3, \, \forAll i,j \in \set{1, \ldots, n} }. \]
\end{lemma}

\begin{proof}
For every $i,j \in \set{1, ,\ldots, n}$, we have $\left( q_i, q_j, t_{i,j} \right) \in \RL^3$ if and only if $q_i, q_j \geq 0$, and $2 q_i q_j \geq t_{i,j}^2$. Thus, $t_{i,j} = \sqrt{2 q_i q_j}$ is optimal with objective function value $\rF (p,q)$. \qed
\end{proof}

This lemma provides an SOCP
representation for the hypograph of the fidelity function. Recall that the hypograph of a concave function is a convex set. Also, the dimension of the hypograph of $\rF (\cdot, q) : \R_+^n \to \R$ is equal to $n$ (assuming $q > 0$). Since the hypograph is $O(n)$-dimensional and convex, there exists a \emph{self-concordant barrier function} for the set with complexity parameter $\rO(n)$, shown by Nesterov and Nemirovski~\cite{NN94}.
This allows the derivation of interior-point methods for the underlying convex optimization problem which use $\rO(\sqrt{n} \log(1/\eps))$ iterations, where $\eps$ is an accuracy parameter.
The above lemma uses $\Omega(n^2)$ second-order cone constraints and the usual
treatment of these ``cone constraints'' with optimal self-concordant barrier functions
lead to interior-point methods with an iteration complexity bound of $\rO(n \log(1/\eps))$. It is
conceivable that there exist better convex representations of the hypograph
of the fidelity function than the one we provided in Lemma~\ref{soclemma}.

\comment{
\subsection{Finding efficient SOCP formulations for the reduced problems}
}

We can further simplify the reduced problems using fewer SOC constraints than
derived above. We first consider the dual formulation of the reduced problems,
so as to avoid the hypograph of the fidelity function.

Using Lemma~\ref{FidelityLemma}, we write Alice's reduced problem for
forcing outcome $0$ as an SDP. The dual of this SDP is
\[ \begin{array}{rrrcllllllllllllll}
\textrm{} & \inf                         & z_1 \\
                     & \textup{subject to} & z_1 \cdot e_{A_1} & \geq & \tr_{B_1}(z_2), \\
                     &                               & z_2 \otimes e_{A_2} & \geq & \tr_{B_2}(z_3), \\
                     & & & \vdots \\
                     &                               & z_n \otimes e_{A_n} & \geq & \tr_{B_n}(z_{n+1}), \\
                     & & \Diag(z_{n+1}^{(y)}) & \succeq & \half \beta_{a,y}
\, \sqrtt{\alpha_a}, & \forAll a \in \zo, y \in B \enspace, \\
& & z_1 & \in & \R, \\
& & z_i & \in & \R^{A_1 \times B_1 \times \dotsb \times A_{i-1} \times B_{i-1}}, & \forAll i \in \set{2, \dotsc, n+1} \enspace, \\
 & \textrm{where} & z^{(y)}_{n+1,x} & = & z_{n+1, x_1 y_1 x_2 y_2 \dotsb,
x_n y_n}, & \forAll x \in A, y \in B \enspace.
\end{array} \]

\comment{
\[ \begin{array}{rrrcllllllllllllll}
& \sup                         & \dfrac{1}{2} \dsum_{a \in \zo} \dsum_{y \in B} \inner{\sigma^{(a,y)}}{\sqrtt{\alpha_a}} \\
& \textup{s. t.} & \Diag(\sigma^{(a,y)}) & = & s^{(a,y)}, & \forAll a \in \zo, \\
& & & & & \phantom{\forAll} y \in B, \\
                     & & \sigma^{(a,y)} & \in & \Herm_+^{A}, & \forAll a \in \zo, \\
& & & & & \phantom{\forAll} y \in B, \\
                     &                               & \tr_{A_1}(s_1) & = & 1, \\
                     &                               & \tr_{A_2}(s_2) & = & s_1 \otimes e_{B_{1}}, \\
                     & & & \vdots \\
                     &                               & \tr_{A_n}(s_n) & = & s_{n-1} \otimes e_{B_{n-1}}, \\
                     &                               & \tr_{A'_{0}}(s) & = & s_n \otimes e_{B_{n}}, \\
                     &                               & s_{1} & \in & \R_{+}^{A_{1}}, \\
                     &                               & s_{j} & \in & \R_{+}^{A_{1} \times B_{1} \times \cdots \times B_{j-1} \times A_{j}}, & \forAll j \in \{ 2, \ldots, n \}, \\
                     & & s & \in & \R_{+}^{A'_{0} \times A \times B}.
\end{array} \]
} 

The only nonlinear constraint in the above problem is of the form
\[ \Diag(z) \succeq \sqrtt{q}, \]
for some fixed $q \geq 0$.
From the proof of Lemma~\ref{FidelityLemma}, we see that for~$z$ which
is positive in every coordinate
\[ \Diag(z) \succeq \sqrtt{q} \iff \inner{z^{-1}}{q} \leq 1. \]
So, it suffices to characterize inverses using SOCP constraints which can be done by considering
\[ (z_i, r_i, \sqrt{2}) \in \RL \iff r_i \geq z_i^{-1}. \]

With this observation, we can write the dual of Alice and Bob's reduced problems using $\rO(n)$ $\RL$ constraints for each fidelity function in the objective function as opposed to $\Omega(n^2)$ constraints
as above.

\comment{
\begin{corollary}
The four cheating probabilities can be written as
\[ P_{\rB,0}^* = \displaystyle\inf_{\substack{\inner{\xi_a}{\beta_a} \leq 1 \\ (\xi_{a,y}, v_{a,y}, \sqrt{2}) \in \RL^3}} \left\{ \sum_{x_1} \max_{y_1} \sum_{x_2} \max_{y_2} \cdots \sum_{x_n} \max_{y_n} \sum_a \half \alpha_{a,x} v_{a,y} \right\}, \]

\[ P_{\rB,1}^* = \displaystyle\inf_{\substack{\inner{\xi_a}{\beta_{\bar{a}}} \leq 1 \\ (\xi_{a,y}, v_{a,y}, \sqrt{2}) \in \RL^3}} \left\{ \sum_{x_1} \max_{y_1} \sum_{x_2} \max_{y_2} \cdots \sum_{x_n} \max_{y_n} \sum_a \half \alpha_{a,x} v_{a,y} \right\}, \]

\[ P_{\rA,0}^* = \displaystyle\inf_{\substack{\frac{\beta_{a,y}}{2} \inner{\mu^{(y)}}{\alpha_a} \leq 1 \\ (\mu_{x,y}, z_{n+1,x,y}, \sqrt{2}) \in \RL^3}} \left\{ \max_{x_1} \sum_{y_1} \cdots \max_{x_n} \sum_{y_n} z_{n+1,x,y} \right\}, \]

\[ P_{\rA,1}^* = \displaystyle\inf_{\substack{\frac{\beta_{\bar{a},y}}{2} \inner{\mu^{(y)}}{\alpha_a} \leq 1 \\ (\mu_{x,y}, z_{n+1,x,y}, \sqrt{2}) \in \RL^3}} \left\{ \max_{x_1} \sum_{y_1} \cdots \max_{x_n} \sum_{y_n} z_{n+1,x,y} \right\}. \]
\end{corollary}
} 

\subsection{Numerical performance of SDP formulation vs. SOCP formulation} \label{SDPvsSOCP}

Since the search algorithm designed in this paper examines the optimal cheating probabilities of many protocols (more than $10^{16}$) we are concerned with the efficiency of solving the reduced problems. In this subsection, we discuss the efficiency of this computation. Our computational platform is an SGI XE C1103 with 2x 3.2 GHz 4-core Intel X5672 x86 CPUs processor, and 10 GB memory, running Linux. The reduced problems were solved using SeDuMi 1.3, a program for solving semidefinite programs and rotated second-order cone programs in Matlab (Version 7.12.0.635) \cite{Stu99, Stu02}.

Table~\ref{RSOCvsSDP1} (on the next page) compares the computation of Alice's reduced problem in a four-round protocol for forcing an outcome of $0$ with $5$-dimensional messages. The top part of the table presents the average running time, the maximum running time, and the worst gap (the maximum of the extra time needed to solve the problem compared to the other formulation). The bottom part of the table presents the average number of iterations, the average feasratio, the average timing (the time spent in preprocessing,  iterations, and postprocessing, respectively), and the average cpusec.

Table~\ref{RSOCvsSDP1} suggests that solving the rotated second-order cone programs are comparable to solving the semidefinite programs. However, before testing the other three cheating probabilities, we test the performance of the two formulations from Table~\ref{RSOCvsSDP1} in a setting that appears more frequently in the search. In particular,  the searches detailed in Section~\ref{sect:numerical} deal with many protocols with very sparse parameters. We retest the values in Table~\ref{RSOCvsSDP1} when we force the first entry of $\alpha_0$, the second entry of $\alpha_1$, the third entry of $\beta_0$, and the fourth entry of $\beta_1$ to all be $0$. The results are shown in Table~\ref{RSOCvsSDP2}.

\begin{table}
\caption{Comparison of solving the $\mathrm{SOCP}$ and $\SDP$ formulations of Alice forcing outcome $0$ with $5$-dimensional messages in four-rounds (averaged over $1,000$ randomly
selected protocols).}
\label{RSOCvsSDP1}
\quad
\begin{center}
\begin{tabular}{|r|c|c|}
\hline
\comment{Strategy Description & } INFO parameters & $\mathrm{SOCP}$ & SDP \\
\hline
\hline
\quad & \quad & \quad \\
Average running time (s) & $0.1551$ & $0.1529$ \\
Max running time (s) & $0.7491$ & $0.2394$ \\
Worst gap (s) & $+ \, 0.5098$ & $+ \, 0.0927$ \\
\quad & \quad & \quad \\
\hline
\quad & \quad & \quad \\
Average iteration & $14.4420$ & $12.2940$ \\
Average feasratio & $0.9990$ & $1.0000$ \\
Average timing & $[0.0270, 0.1267, 0.0010]^\transpose$ & $[0.0024, 0.1494, 0.0009]^\transpose$ \\
Average cpusec & $0.9283$ & $0.6588$ \\
\quad & \quad & \quad \\
\hline
\end{tabular}
\end{center}
\end{table}

\begin{table}
\caption{Comparison of solving the $\mathrm{SOCP}$ and $\SDP$ formulations of Alice forcing outcome $0$ with $5$-dimensional messages in four-rounds (averaged over $1,000$ randomly
selected protocols with forced $0$ entries).}
\label{RSOCvsSDP2}
\quad
\begin{center}
\begin{tabular}{|r|c|c|}
\hline
\comment{Strategy Description & } INFO parameters & $\mathrm{SOCP}$ & SDP \\
\hline
\hline
\quad & \quad & \quad \\
Average running time (s) & $0.4104$ & $0.1507$ \\
Max running time (s) & $0.7812$ & $0.2084$ \\
Worst gap (s) & $+ \, 0.6323$ & $ + \, 0$ \\
\quad & \quad & \quad \\
\hline
\quad & \quad & \quad \\
Average iterations & $32.7370$ & $12.2530$ \\
Average feasratio & $0.5172$ & $1.0000$ \\
Average timing & $[0.0279, 0.3814, 0.0010]^\transpose$ & $[0.0023, 0.1473, 0.0009]^\transpose$ \\
Average cpusec & $2.4953$ & $0.5605$ \\
\quad & \quad & \quad \\
\hline
\end{tabular}
\end{center}
\end{table}

As we can see, the second-order cone programming formulation stumbles when the data does not have full support. Since we search over many vectors without full support, we use the semidefinite programming formulation to solve the reduced problems and for the analysis throughout the rest of this paper.


\newpage
\section{Protocol filter}
\label{sect:filter}

In this section, we describe ways to bound the optimal cheating probabilities from below by finding feasible solutions to Alice and Bob's reduced cheating problems. In
the search for parameters that lead to the lowest bias, our algorithm tests many protocols. The idea
is to devise simple tests to check whether a protocol is a good candidate for being optimal. For example, suppose we can
quickly compute the success probability of a certain cheating strategy for Bob. If this strategy succeeds with too high
a probability for a given set of parameters, then we can rule out these parameters as being good choices. This
saves the time it would have taken to solve the SDPs (or SOCPs).

We illustrate this idea using the Kitaev lower bound below.
\begin{theorem}[\cite{Kit03, GW07}]
For \emph{any} coin-flipping protocol, we have
\[ P_{\rA,0}^* P_{\rB,0}^* \geq \half \quad \text{ and } \quad
P_{\rA,1}^* P_{\rB,1}^* \geq \half. \]
\end{theorem}
Suppose that we find that $P_{\rA,0}^* \approx 1/2$, that is, the protocol is very secure against
dishonest Alice cheating towards $0$. Then, from the Kitaev bound, we infer that $P_{\rB,0} \approx 1$ and the protocol is highly insecure against cheating Bob. Therefore, we can
avoid solving for $P_{\rB,0}^*$.

The remainder of this section is divided according to the party that is
dishonest. We discuss cheating strategies for the two parties for the special cases of $4$-round and $6$-round protocols.

\paragraph{Cheating Alice} \quad \\

\comment{
Before proceeding, we need the following definition.

\begin{definition}
Let ${f_1, \ldots, f_n} : \R^n \to \R \cup \set{\infty}$ be proper, convex functions (convex functions not equal to the $+ \infty$ constant function). We define the \emph{convex hull} of $\{ f_1, \ldots, f_n \}$, denoted $\conv \{f_1, \ldots, f_n \}$, as the greatest convex function $f$ such that $f(x) \leq f_1(x), \ldots, f_n(x)$ for every $x \in \R^n$.
\end{definition}

We can similarly define the \emph{concave hull}, denoted $\conc$, as
\[ \conc \set{f_1, \ldots, f_n} := - \conv \set{-f_1, \ldots, -f_n} \]
when $f_1, \ldots, f_n$ are proper, concave functions (concave functions not equal to the $- \infty$ constant function).

} 

We now present a theorem which captures some of Alice's cheating strategies.

\begin{theorem}\label{AliceFilter}
For a protocol parameterized by $\alpha_0, \alpha_1 \in \prob^A$ and $\beta_0, \beta_1 \in \prob^B$, we can bound Alice's optimal cheating probability as follows:
\begin{eqnarray}
P_{\rA,0}^*
& \geq & \half \sum_{y \in B} \conc \set{\beta_{a,y} \rF(\cdot, \alpha_a)
: {a \bit}}(v) \label{A1} \\
& \geq & \half \lambda_{\max} \left( \eta \sqrtt{\alpha_0} + \tau \sqrtt{\alpha_1} \right) \label{A2} \\
& \geq & \left( \half + \half \sqrt{\rF(\alpha_0, \alpha_1)} \right) \left( \half + \half \Delta(\beta_0, \beta_1) \right), \label{A3}
\end{eqnarray}
where
\[ \eta := \sum_{y \in B: \atop \beta_{0,y} \geq \beta_{1,y}} \beta_{0,y} \quad \textup{ and } \quad \tau := \sum_{y \in B: \atop \beta_{0,y} < \beta_{1,y}} \beta_{1,y}
\enspace, \]
and $\sqrt{v}$ is the normalized principal eigenvector of $\eta \sqrtt{\alpha_0} + \tau \sqrtt{\alpha_1}$.

Furthermore, in a six-round protocol, we have
\begin{eqnarray}
P_{\rA,0}^*
& \geq & \half \lambda_{\max} \left( \eta' \sqrtt{\tr_{A_2}(\alpha_0)} + \tau' \sqrtt{\tr_{A_2}(\alpha_1)} \right) \label{A4} \\
& \geq & \left( \half + \half \sqrt{\rF(\tr_{A_2}(\alpha_0), \tr_{A_2}(\alpha_1))} \right) \left( \half + \half \Delta (\tr_{B_2}(\beta_0), \tr_{B_2}(\beta_1)) \right), \label{A5}
\end{eqnarray}
where
\[ \eta' := \sum_{y_1 \in B_1: \atop [\tr_{B_2}(\beta_{0})]_{y_1} \geq  [\tr_{B_2}(\beta_{1})]_{y_1}}  [\tr_{B_2}(\beta_{0})]_{y_1} \quad \textup{ and } \quad
\tau' := \sum_{y_1 \in B_1: \atop [\tr_{B_2}(\beta_{0})]_{y_1} <  [\tr_{B_2}(\beta_{1})]_{y_1}}  [\tr_{B_2}(\beta_{1})]_{y_1}
\enspace. \]
We have analogous bounds for $P_{\rA,1}^*$, which are obtained by
interchanging $\beta_0$ and $\beta_1$ in the above expressions.
\end{theorem}
We call \textup{(\ref{A1})} Alice's \emph{improved eigenstrategy\/},
\textup{(\ref{A2})} her \emph{eigenstrategy\/}, and \textup{(\ref{A3})} her \emph{three-round strategy\/}.
For six-round protocols, we call \textup{(\ref{A4})} Alice's \emph{eigenstrategy\/} and
\textup{(\ref{A5})} her \emph{measuring strategy\/}.

Note that only the improved eigenstrategy is affected by switching $\beta_0$ and $\beta_1$ (as long as we are willing to accept a slight modification to how we break ties in the definitions of $\eta, \eta', \tau,$ and $\tau'$).

We now briefly describe the strategies that yield the corresponding cheating probabilities in Theorem~\ref{AliceFilter}. Her three-round strategy is to prepare
the qubits~$AA'$ in the state~$\psi' = (\psi_0 + \psi_1)/\norm{\psi_0 + \psi_1}$ instead of
$\psi_0$ or~$\psi_1$, send the
first~$n$ messages accordingly, then measure the qubits received from Bob
to try to learn $b$, and reply with a bit $a$ using the measurement outcome
(along with the rest of the state~$\psi'$), to
bias the coin towards her desired output. Her eigenstrategy is the same as her three-round strategy, except that the first message is further optimized. The improved eigenstrategy has the same first message as in
her eigenstrategy, but the last message is further optimized.
\comment{
These strategies work for the general case as well, where Alice treats her first $n$ messages like the first one in a four-round protocol.
}

For a six-round protocol, Alice's measuring strategy is to prepare the qubits~$AA'$ in the following state $\psi' = (\psi_0' + \psi_1')/\norm{\psi_0' + \psi_1'}$
where~$\psi_0'$ and $\psi_1'$ are
purifications of~$\tr_{A_2, A'}(\psi_0 \psi_0^*)$ and $\tr_{A_2, A'}(\psi_1 \psi_1^*)$,
respectively. She measures Bob's first message to try to learn $b$, then depending on the outcome, she applies a (fidelity achieving) unitary before sending the rest of her messages.
Her six-round eigenstrategy is similar to her measuring strategy, except her first message is optimized in a way described in the proof.

\comment{ (figures depicting Alice cheating in four and six-round protocols are shown  below).

\begin{figure}[!h]
  \centering
  \scalebox{0.75}{
  \unitlength=.65mm
  \begin{gpicture}
\gasset{Nh=1,Nadjust=w,Nframe=n,linewidth=0.1,AHnb=0}

\node(A)(-100,110){\textcolor{red}{Alice does not follow protocol}}
\node(B)(100,110){\textcolor{blue}{Bob prepares ${\phi} \in \C^{B_0 \times B'_0 \times B \times B'}$}}

\node(M1A)(-100,80){}
\node(M1B)(100,80){}
\node(M2A)(-100,60){}
\node(M2B)(100,60){}
\node(M3A)(-100,40){}
\node(M3B)(100,40){}
\node(M4A)(-100,20){}
\node(M4B)(100,20){}

  \drawedge[AHnb=1,AHangle=30,AHLength=5,AHlength=0,ATnb=1,ATangle=30,ATLength=0,ATlength=0](M1A,M1B){\textcolor{red}{Alice sends $\C^{A}$ \quad ($x \in A$)}}
\drawedge[AHnb=1,AHangle=30,AHLength=0,AHlength=0,ATnb=1,ATangle=30,ATLength=5,ATlength=0](M2A,M2B){\textcolor{blue}{Bob sends $\C^{B}$ \quad ($y \in B$)}}
\drawedge[AHnb=1,AHangle=30,AHLength=5,AHlength=0,ATnb=1,ATangle=30,ATLength=0,ATlength=0](M3A,M3B){\textcolor{red}{Alice sends $\C^{A'_0 \times A'}$ \quad ($a \in \{ 0, 1 \}$ and a copy of $x$)}}
\drawedge[AHnb=1,AHangle=30,AHLength=0,AHlength=0,ATnb=1,ATangle=30,ATLength=5,ATlength=0](M4A,M4B){\textcolor{blue}{Bob sends $\C^{B'_0 \times B'}$ \quad ($b \in \{ 0, 1 \}$ and a copy of $y$)}}

\node(Ameasure)(-100,-5){\textcolor{red}{Alice simply outputs her desired outcome}}
\node(Bmeasure)(100,-5){\textcolor{blue}{Bob checks if Alice cheated}}

   \end{gpicture}
  }
  \label{4RAliceprotocolnv}
  \caption{Alice cheating in a four-round protocol.}
\end{figure}

\begin{figure}[!h]
  \centering
  \scalebox{0.75}{
  \unitlength=.65mm
  \begin{gpicture}
\gasset{Nh=1,Nadjust=w,Nframe=n,linewidth=0.1,AHnb=0}

\node(A)(-100,110){\textcolor{red}{Alice does not follow protocol}}
\node(B)(100,110){\textcolor{blue}{Bob prepares ${\phi} \in \C^{B_0 \times B'_0 \times B_1 \times B_2 \times B'_1 \times B'_2}$}}

\node(M1A)(-100,80){}
\node(M1B)(100,80){}
\node(M2A)(-100,60){}
\node(M2B)(100,60){}
\node(M3A)(-100,40){}
\node(M3B)(100,40){}
\node(M4A)(-100,20){}
\node(M4B)(100,20){}
\node(M5A)(-100,0){}
\node(M5B)(100,0){}
\node(M6A)(-100,-20){}
\node(M6B)(100,-20){}

  \drawedge[AHnb=1,AHangle=30,AHLength=5,AHlength=0,ATnb=1,ATangle=30,ATLength=0,ATlength=0](M1A,M1B){\textcolor{red}{Alice sends $\C^{A_1}$ \quad ($x_1 \in A_1$)}}
\drawedge[AHnb=1,AHangle=30,AHLength=0,AHlength=0,ATnb=1,ATangle=30,ATLength=5,ATlength=0](M2A,M2B){\textcolor{blue}{Bob sends $\C^{B_1}$ \quad ($y_1 \in B_1$)}}
\drawedge[AHnb=1,AHangle=30,AHLength=5,AHlength=0,ATnb=1,ATangle=30,ATLength=0,ATlength=0](M3A,M3B){\textcolor{red}{Alice sends $\C^{A_2}$ \quad ($x_2 \in A_2$)}}
\drawedge[AHnb=1,AHangle=30,AHLength=0,AHlength=0,ATnb=1,ATangle=30,ATLength=5,ATlength=0](M4A,M4B){\textcolor{blue}{Bob sends $\C^{B_2}$ \quad ($y_2 \in B_2$)}}
\drawedge[AHnb=1,AHangle=30,AHLength=5,AHlength=0,ATnb=1,ATangle=30,ATLength=0,ATlength=0](M5A,M5B){\textcolor{red}{Alice sends $\C^{A'_0 \times A'_1 \times A'_2}$ \quad ($a \in \{ 0, 1 \}$ and a copy of $x_1, x_2$)}}
\drawedge[AHnb=1,AHangle=30,AHLength=0,AHlength=0,ATnb=1,ATangle=30,ATLength=5,ATlength=0](M6A,M6B){\textcolor{blue}{Bob sends $\C^{B'_0 \times B'_1 \times B'_2}$ \quad ($b \in \{ 0, 1 \}$ and a copy of $y_1, y_2$)}}

\node(Ameasure)(-100,-45){\textcolor{red}{Alice simply outputs her desired outcome}}
\node(Bmeasure)(100,-45){\textcolor{blue}{Bob checks if Alice cheated}}

   \end{gpicture}
  }
  \caption{Alice cheating in a six-round protocol.}
  \label{6RAliceprotocolnv}
\end{figure}

} 

We prove Theorem~\ref{AliceFilter} in the appendix.


\paragraph{Cheating Bob} \quad \\

We turn to strategies for a dishonest Bob. 

\newpage 
\begin{theorem}\label{BobFilter}
For a protocol parameterized by $\alpha_0, \alpha_1 \in \prob^A$ and $\beta_0, \beta_1 \in \prob^B$, we can bound Bob's optimal cheating probability as follows:

\begin{equation}
P_{\rB,0}^* \geq \half + \half \sqrt{\rF(\beta_0, \beta_1)},
\label{B1}
\end{equation}
and
\begin{equation}
P_{\rB,0}^* \geq \half + \half \Delta(\tr_{A_2 \times \cdots \times A_n}(\alpha_0), \tr_{A_2 \times \cdots \times A_n}(\alpha_1)).
\label{B2}
\end{equation}

In a four-round protocol, we have
\begin{eqnarray}
P_{\rB,0}^*
& \geq & \half \sum_{a \bit} \rF\!\left( \sum_{x \in A} \alpha_{a,x} v_x, \beta_a \right) \label{B3} \\
& \geq & \half \sum_{x \in A} \lambda_{\max} \left( \sum_{a \bit} \alpha_{a,x} \sqrtt{\beta_a} \right) \label{B4} \\
& \geq & \max \left\{ \half + \half \Delta(\alpha_0, \alpha_1), \; \half + \half \sqrt{\rF(\beta_0, \beta_1)}
\right\} \enspace, \nonumber
\end{eqnarray}
where $\sqrt{v_x}$ is the normalized principal eigenvector of $\sum_{a \bit} \alpha_{a,x} \sqrtt{\beta_a}$.

In a six-round protocol, we have
\begin{eqnarray}
P_{\rB,0}^* & \geq &
\half \sum_{a \in A'_0} \rF \left( \sum_{x \in A} \alpha_{a,x} \, \tilde{p_2}^{(x)}, \beta_a \right) \label{B5} \\
& \geq & \half \, \lambda_{\max} \! \left(\kappa \sqrtt{\tr_{B_2}(\beta_0)} + \zeta \sqrtt{\tr_{B_2}(\beta_1)} \right) \label{B6} \\
& \geq & \left( \half + \half \sqrt{\rF(\tr_{B_2}(\beta_0), \tr_{B_2}(\beta_1))} \right) \left( \half + \half \Delta (\alpha_0, \alpha_1) \right), \label{B7}
\end{eqnarray}
where
\[ [\tilde{p_2}^{(x)}]_{y_1, y_2} := \left\{ \begin{array}{ccl}
c_{y_1} \frac{\beta_{g(x), y_1, y_2}}{[\tr_{B_2}(\beta_{g(x)})]_{y_1}} & \text{ if }~ [\tr_{B_2}(\beta_{g(x)})]_{y_1} > 0
\enspace, \\
\quad \\
c_{y_1} \frac{1}{|B_2|} & \text{ if }~ [\tr_{B_2}(\beta_{g(x)})]_{y_1} = 0
\enspace,
\end{array} \right. \]
\[ \kappa = \dsum_{x \in A: \atop \alpha_{0,x} \geq \alpha_{1,x}} \alpha_{0,x}
\enspace, \quad \text{} \quad
\zeta = \dsum_{x \in A: \atop \alpha_{0,x} < \alpha_{1,x}} \alpha_{1,x}
\enspace, \quad \text{} \quad g(x) = \arg\max_{a} \set{\alpha_{a,x}}
\enspace, \]
and $\sqrt{c}$ is the normalized principal eigenvector of
\[ \half \, \lambda_{\max} \! \left( \kappa \sqrtt{\tr_{B_2}(\beta_{0})} + \zeta \sqrtt{\tr_{B_2}(\beta_{1})} \right). \]

Furthermore, if $|A_i| = |B_i|$ for all $i \in \{ 1, \ldots, n \}$, then
\begin{equation}
\label{B8}
P_{\rB,0}^* \geq \half \sum_{a \bit} \rF(\alpha_a, \beta_a) \enspace.
\end{equation}

We get analogous lower bounds for $P_{\rB,1}^*$ by switching the roles of $\beta_0$ and $\beta_1$
in the above expressions.
\end{theorem}
We prove Theorem~\ref{BobFilter} in the appendix.
We call \textup{(\ref{B1})} Bob's \emph{ignoring strategy\/} and \textup{(\ref{B2})} his \emph{measuring strategy\/}.
For four-round protocols, we call \textup{(\ref{B3})} Bob's \emph{eigenstrategy\/} and
\textup{(\ref{B4})} his \emph{eigenstrategy lower bound\/}.
For six-round protocols, we call \textup{(\ref{B5})} Bob's \emph{six-round eigenstrategy\/}, \textup{(\ref{B6})} his \emph{eigenstrategy lower bound\/}, and \textup{(\ref{B7})} his \emph{three-round strategy\/}.
We call \textup{(\ref{B8})} Bob's \emph{returning strategy\/}.

Note that the only strategies that are affected by switching $\beta_0$ and $\beta_1$ are the eigenstrategy and the returning strategy.

\comment{

\begin{figure}[!h]
  \centering
  \scalebox{0.75}{
  \unitlength=.65mm
  \begin{gpicture}
\gasset{Nh=1,Nadjust=w,Nframe=n,linewidth=0.1,AHnb=0}

\node(A)(-100,110){\textcolor{red}{Alice prepares ${\psi} \in \C^{A_0 \times A'_0 \times A \times A'}$}}
\node(B)(100,110){\textcolor{blue}{Bob does not follow protocol}}

\node(M1A)(-100,80){}
\node(M1B)(100,80){}
\node(M2A)(-100,60){}
\node(M2B)(100,60){}
\node(M3A)(-100,40){}
\node(M3B)(100,40){}
\node(M4A)(-100,20){}
\node(M4B)(100,20){}

  \drawedge[AHnb=1,AHangle=30,AHLength=5,AHlength=0,ATnb=1,ATangle=30,ATLength=0,ATlength=0](M1A,M1B){\textcolor{red}{Alice sends $\C^{A}$ \quad ($x \in A$)}}
\drawedge[AHnb=1,AHangle=30,AHLength=0,AHlength=0,ATnb=1,ATangle=30,ATLength=5,ATlength=0](M2A,M2B){\textcolor{blue}{Bob sends $\C^{B}$ \quad ($y \in B$)}}
\drawedge[AHnb=1,AHangle=30,AHLength=5,AHlength=0,ATnb=1,ATangle=30,ATLength=0,ATlength=0](M3A,M3B){\textcolor{red}{Alice sends $\C^{A'_0 \times A' }$ \quad ($a \in \{ 0, 1 \}$ and a copy of $x$)}}
\drawedge[AHnb=1,AHangle=30,AHLength=0,AHlength=0,ATnb=1,ATangle=30,ATLength=5,ATlength=0](M4A,M4B){\textcolor{blue}{Bob sends $\C^{B'_0 \times B'}$ \quad ($b \in \{ 0, 1 \}$ and a copy of $y$)}}

\node(Ameasure)(-100,-5){\textcolor{red}{Alice checks if Bob cheated}}
\node(Bmeasure)(100,-5){\textcolor{blue}{Bob simply outputs his desired outcome}}

   \end{gpicture}
  }
  \label{fig:xmlcar}
  \caption{Bob cheating in a four-round protocol.}
\end{figure}

} 

We now briefly describe the strategies that yield the corresponding cheating probabilities in Theorem~\ref{BobFilter}. Bob's ignoring strategy is to
prepare the qubits~$BB'$ in the state~$\phi' = (\phi_0 + \phi_1)/\norm{\phi_0 + \phi_1}$
instead of~$\phi_0$ or~$\phi_1$, send the first~$n$ messages
accordingly, then send a value for $b$ that
favours his desired outcome (along with the rest of~$\phi'$). His measuring strategy is to measure Alice's first message, choose $b$ according to his best guess for $a$
and run the protocol with~$\phi_b$. His returning strategy is to send Alice's messages right back to her. For the four-round eigenstrategy, Bob's commitment state is a principal eigenvector depending on Alice's first message.

For a six-round protocol, Bob's three-round strategy is to prepare the qubits~$BB'$ in the following state $\phi' = (\phi_0' + \phi_1')/\norm{\phi_0' + \phi_1'}$
where~$\phi_0'$ and $\phi_1'$ are
purifications of~$\tr_{B_2, B'}(\phi_0 \phi_0^*)$ and $\tr_{B_2, B'}(\phi_1 \phi_1^*)$,
respectively. He measures Alice's second message to try to learn $a$, then depending on the outcome, he applies a (fidelity achieving) unitary before sending the rest of his messages. His six-round eigenstrategy is similar to his three-round strategy except that the first message is optimized in a way described in the proof.

\comment{

\begin{figure}[!h]
  \centering
  \scalebox{0.75}{
  \unitlength=.65mm
  \begin{gpicture}
\gasset{Nh=1,Nadjust=w,Nframe=n,linewidth=0.1,AHnb=0}

\node(A)(-100,110){\textcolor{red}{Alice prepares ${\psi} \in \C^{A_0 \times A'_0 \times A_1 \times A_2 \times A'_1 \times A'_2}$}}
\node(B)(100,110){\textcolor{blue}{Bob does not follow protocol}}

\node(M1A)(-100,80){}
\node(M1B)(100,80){}
\node(M2A)(-100,60){}
\node(M2B)(100,60){}
\node(M3A)(-100,40){}
\node(M3B)(100,40){}
\node(M4A)(-100,20){}
\node(M4B)(100,20){}
\node(M5A)(-100,0){}
\node(M5B)(100,0){}
\node(M6A)(-100,-20){}
\node(M6B)(100,-20){}

  \drawedge[AHnb=1,AHangle=30,AHLength=5,AHlength=0,ATnb=1,ATangle=30,ATLength=0,ATlength=0](M1A,M1B){\textcolor{red}{Alice sends $\C^{A_1}$ \quad ($x_1 \in A_1$)}}
\drawedge[AHnb=1,AHangle=30,AHLength=0,AHlength=0,ATnb=1,ATangle=30,ATLength=5,ATlength=0](M2A,M2B){\textcolor{blue}{Bob sends $\C^{B_1}$ \quad ($y_1 \in B_1$)}}
\drawedge[AHnb=1,AHangle=30,AHLength=5,AHlength=0,ATnb=1,ATangle=30,ATLength=0,ATlength=0](M3A,M3B){\textcolor{red}{Alice sends $\C^{A_2}$ \quad ($x_2 \in A_2$)}}
\drawedge[AHnb=1,AHangle=30,AHLength=0,AHlength=0,ATnb=1,ATangle=30,ATLength=5,ATlength=0](M4A,M4B){\textcolor{blue}{Bob sends $\C^{B_2}$ \quad ($y_2 \in B_2$)}}
\drawedge[AHnb=1,AHangle=30,AHLength=5,AHlength=0,ATnb=1,ATangle=30,ATLength=0,ATlength=0](M5A,M5B){\textcolor{red}{Alice sends $\C^{A'_0 \times A'_1 \times A'_2}$ \quad ($a \in \{ 0, 1 \}$ and a copy of $x_1, x_2$)}}
\drawedge[AHnb=1,AHangle=30,AHLength=0,AHlength=0,ATnb=1,ATangle=30,ATLength=5,ATlength=0](M6A,M6B){\textcolor{blue}{Bob sends $\C^{B'_0 \times B'_1 \times B'_2}$ \quad ($b \in \{ 0, 1 \}$ and a copy of $y_1, y_2$)}}

\node(Ameasure)(-100,-45){\textcolor{red}{Alice checks if Bob cheated}}
\node(Bmeasure)(100,-45){\textcolor{blue}{Bob simply outputs his desired outcome}}

   \end{gpicture}
  }
  \caption{Bob cheating in a six-round protocol.}
  \label{fig:6RBobprotocolnv}
\end{figure}

} 


\section{Protocol symmetry} \label{sect:symmetry}

In this section, we discuss equivalence between protocols due to symmetry
in the states used in them. Namely, we identify transformations on
states under which the bias remains unchanged. This allows us to prune
the search space of parameters needed to specify a protocol in the
family under scrutiny. As a result, we significantly reduce the time
required for our searches.

\subsection{Index symmetry}

We show that if we permute the elements of $A_i$ or $B_i$, for any $i \in \set{1, \ldots, n}$, then this does not change the bias of the protocol. We first show that cheating Bob is unaffected.

\paragraph{Cheating Bob} \quad \\

Bob's reduced problems are to maximize $\half \sum_{a \in A'_0} \rF\!\left( (\alpha_{a} \otimes \id_{B})^{\transpose} p_{n}, \, \beta_a \right)$, for forcing outcome $0$, and $\half \sum_{a \in A'_0} \rF\!\left( (\alpha_{a} \otimes \id_{B})^{\transpose} p_{n}, \, \beta_{\bar{a}} \right)$, for forcing outcome $1$, over the polytope $\calP_\rB$ defined as the set of all vectors
$(p_{1}, p_{2}, \ldots, p_{n})$ that satisfy
\[ \begin{array}{rrrcllllllllllllll}

                     &  & \tr_{B_1}(p_1) & = & e_{A_{1}}, \\
                     &  & \tr_{B_2} (p_2) & = & p_{1} \otimes e_{A_{2}}, \\
                     & & & \vdots \\
                     &  & \tr_{B_n} (p_n) & = & p_{n-1} \otimes e_{A_{n}}, \\
                     & & p_j & \in & \R_+^{A_{1} \times B_{1} \times \cdots \times A_{j} \times B_{j}}, \; \forAll j \in \{ 1, \ldots, n \}
\enspace.
\end{array} \]

Suppose we are given a new protocol where the elements of $A_i$ have been permuted, for some $i \in \set{1, \ldots, n}$ (and therefore the entries of $\alpha_a$ for both $a \bit$). We can write the entries of $(\alpha_{a} \otimes \id_{B})^{\transpose} p_{n}$ as
\[ [(\alpha_{a} \otimes \id_{B})^{\transpose} p_{n}]_y = \sum_{x \in A} \alpha_{a,x} [p_{n}]_{x,y}, \]
for each $y \in B$. For any feasible solution for the original protocol,
we construct a feasible solution by permuting the elements of $p_j$ corresponding to $A_i$.
This gives us a bijection, and the feasible solution so constructed has the
same objective function value as the original one. Thus, dishonest Bob
cannot cheat more or less than in the original protocol.

Now suppose we are given a new protocol where the elements of $B_i$ have been permuted for some $i \in \set{1, \ldots, n}$. We can write
\[ \rF\! \left( (\alpha_{a} \otimes \id_{B})^{\transpose} p_{n}, \, \beta_a \right) = \left( \sqrt{(\alpha_{a} \otimes \id_{B})^{\transpose} p_{n}}^\transpose \sqrt{\beta_a} \right)^2. \]
If we permute the entries in $p_n$ corresponding to $B_i$ (and likewise for every variable in the polytope) we get the same objective function value.

Similar arguments hold for $P_{\rB,1}^*$. In both cases, Bob's  cheating probabilities are unaffected.

\paragraph{Cheating Alice} \quad \\

To show that the bias of the protocol remains unchanged, we still need to check that cheating Alice is unaffected by a permutation
of the elements of~$A_i$ or~$B_i$.  Alice's reduced problem is to maximize $\half \sum_{a \in A'_0} \sum_{y \in B} \beta_{a,y} \; \rF(s^{(a,y)}, \alpha_{a})$
for forcing outcome $0$, and $\half \sum_{a \in A'_0} \sum_{y \in B} \beta_{\bar{a},y} \; \rF(s^{(a,y)}, \alpha_{a})$ for forcing outcome $1$, over the set of all vectors $(s_{1}, s_{2}, \ldots, s_{n}, s)$ that satisfy
\[ \begin{array}{rrrcllllllllllllll}
                     &                               & \tr_{A_1}(s_1) & = & 1
\enspace, \\
                     &                               & \tr_{A_2}(s_2) & = & s_1 \otimes e_{B_{1}}
\enspace, \\
                     & & & \vdots \\
                     &                               & \tr_{A_n}(s_n) & = & s_{n-1} \otimes e_{B_{n-1}}
\enspace, \\
                     &                               & \tr_{A'_{0}}(s) & = & s_n \otimes e_{B_{n}}
\enspace, \\
                     &                               & s_{j} & \in & \R_{+}^{A_{1} \times B_{1} \times \cdots \times B_{j-1} \times A_{j}}
\enspace, \; \forAll j \in \{ 1, \ldots, n \} \enspace, \\
                     & & s & \in & \R_{+}^{A'_{0} \times A \times B}
\enspace.
\end{array} \]

By examining the above problem, we see that the same arguments that apply to cheating Bob also apply to cheating Alice. We can simply permute any
feasible solution to account for any permutation in $A_i$ or $B_i$.

Note that these arguments only hold for ``local'' permutations, i.e.,
we cannot in general permute the indices in $A_i \times A_{i'}$
without affecting the bias.

\subsection{Symmetry between probability distributions}

We now identify a different kind of symmetry in the protocols. Recall the four objective functions
\[
P_{\rB,0}^* = \half \sum_{a \in A'_0} \, \rF\! \left( (\alpha_{a} \otimes \id_{B})^{\transpose} p_{n}, \, \beta_a \right) \quad \text{ and } \quad
P_{\rB,1}^* = \half \sum_{a \in A'_0} \, \rF\! \left( (\alpha_{a} \otimes \id_{B})^{\transpose} p_{n}, \, \beta_{\bar{a}} \right) \]
for Bob and
\[
P_{\rA,0}^* = \half \sum_{y \in B} \sum_{a \bit} \beta_{a,y}\, \rF(s^{(a,y)}, \alpha_a) \quad \text{ and } \quad
P_{\rA,1}^* = \half \sum_{y \in B} \sum_{a \bit} \beta_{\bar{a},y}\, \rF(s^{(a,y)}, \alpha_a) \]
for Alice.

We argue that the four quantities above are not affected if we switch $\beta_0$ and $\beta_1$ and simultaneously switch $\alpha_0$ and $\alpha_1$. This is immediate for cheating Bob, but requires explanation for cheating Alice. The only constraints involving $s^{(a,y)}$ can be written as
\[ \sum_{a \in A'_0} s^{(a,y)} = s_n^{(y_1, \ldots, y_{n-1})}, \]
for all $y = (y_1, \ldots, y_{n-1}, y_n) \in B$. Since this constraint is symmetric about $a$, the result follows.

It is also evident that switching $\beta_0$ and $\beta_1$ switches $P_{\rA,0}^*$ and $P_{\rA,1}^*$ and it also switches $P_{\rB,0}^*$ and $P_{\rB,1}^*$. With these symmetries, we can effectively switch the roles of $\alpha_0$ and $\alpha_1$ and the roles of $\beta_0$ and $\beta_1$ independently and the bias is unaffected.

\subsection{The use of symmetry in the search algorithm}

Since we are able to switch the roles of $\alpha_0$ and $\alpha_1$, we assume $\alpha_0$ has the largest entry out of $\alpha_0$ and $\alpha_1$ and similarly that $\beta_0$ has the largest entry out of $\beta_0$ and $\beta_1$.

In four-round protocols, since we can permute the elements of $A = A_1$, we also assume $\alpha_0$ has entries that are non-decreasing. This allows us to upper bound all the entries of $\alpha_0$ and $\alpha_1$ by the last entry in $\alpha_0$. We do this simultaneously for $\beta_0$ and $\beta_1$.

In the six-round version, we need to be careful when applying the index symmetry, we cannot permute all of the entries in $\alpha_0$. The index symmetry only applies to local permutations so we only partially order them. We order $A_2$ such that the entries $\alpha_{0, \tilde{x}_1 {x}_2}$ do not decrease for \emph{one particular index} $\tilde{x}_1 \in A_1$. It is convenient to choose the index corresponding to the largest entry. Then we order the last block of entries in $\alpha_0$ such that they do not decrease. Note that the last entry in $\alpha_0$ is now the largest among all the entries in $\alpha_0$ and $\alpha_1$. We do this simultaneously for $\beta_0$ and $\beta_1$. Note that the search algorithm does not stop all symmetry; for example if $\alpha_0$ and $\alpha_1$ both have an entry of largest magnitude, we do not compare the second largest entries. But, as will be shown in the computational tests, we have a dramatic reduction in the number of protocols to be tested using the symmetry in the way described above. 


\section{Search algorithm}\label{sect:algo}
\label{algo}

In this section, we develop an algorithm for finding coin-flipping protocols with small bias within our parametrized family.

\comment{
We start with the following search algorithm prototype.
\begin{center}
\quad \\
\begin{tabular}{|l|}
\hline \\
\quad
\underline{Pseudo-algorithm for finding a protocol with small bias}
\quad \quad \\
\quad \\
\quad For each choice of $\alpha_0$, $\alpha_1$, $\beta_0$, and $\beta_1$ (modulo the symmetry): \quad \\
\quad $\bullet$ Use the Protocol Filter to test if the protocol has
too ``high'' a bias. \quad \\
\quad $\bullet$ Solve the necessary SDPs characterising the bias. \quad \\
\quad $\bullet$ If the protocol has a desired bias: stop and output the protocol parameters. \quad \\
\quad \\
\hline
\end{tabular}
\quad \\
\quad \\
\end{center}
We now discuss a few key steps to speed up the search and to make it finite.
} 

To search for protocols, we first fix a dimension $d$ for the parameters
\[ \alpha_0, \alpha_1, \beta_0, \beta_1 \in \prob^d. \]
We then create a finite mesh over these parameters by creating a mesh over the entries in the probability vectors $\alpha_0$, $\alpha_1$, $\beta_0$, and $\beta_1$. We do so by increments of a precision parameter $\nu \in (0,1)$. For example, we range over the values
\[ \set{0, \nu, 2 \nu, \ldots, 1 - \nu, 1} \]
for $[\alpha_0]_0$, the first entry of $\alpha_0$. For the second entry of $\alpha_0$, we range over
\[ \{ 0, \nu, 2 \nu, \ldots, 1 - [\alpha_0]_0 \} \]
and so forth. Note that we only consider $\nu = 1/N$ for some positive integer $N$ so that we use the endpoints of the intervals.

This choice in creating the mesh makes it very easy to exploit the symmetry discussed in Section~\ref{sect:symmetry}. We show computationally (in Section~\ref{sect:numerical}) that this symmetry helps by dramatically reducing the number of protocols to be tested. This is important since there are ${{d+N-1}\choose{N}}^4$ protocols to test (before applying symmetry considerations). 
 
Each point in this mesh is a set of candidate parameters for an optimal protocol. As described in Section~\ref{sect:filter}, the protocol filter can be used to expedite the process of checking whether the protocol has high bias or is a good candidate for an optimal protocol. There are two things to be considered at this point which we now address.

First, we have to determine the order in which the cheating strategies in the protocol filter are applied. It is roughly the case that the computationally cheaper tasks give a looser lower bound to the optimal cheating probabilities. Therefore, we start with these easily computable probabilities, i.e. the probabilities involving norms and fidelities, then check the more computationally expensive tasks such as largest eigenvalues and calculating principal eigenvectors. We lastly solve the semidefinite programs. Another heuristic that we use is alternating between Alice and Bob's strategies. Many protocols with high bias seem to prefer either cheating Alice or cheating Bob. Having cheating strategies for both Alice and Bob early in the filter removes the possibility of checking many of Bob's strategies when it is clearly insecure concerning cheating Alice and vice versa. Starting with these heuristics, we then ran preliminary tests to see which order seemed to perform the best. The order (as well as the running times for the filter strategies) is shown in Tables~\ref{4round_running_times1} and \ref{4round_running_times2} for the four-round version and Tables~\ref{6round_running_times1} and \ref{6round_running_times2} for the six-round version.

Second, we need to determine a threshold for what constitutes a ``high bias.'' If a filter strategy has success probability $0.9$, do we eliminate this candidate protocol? The lower the threshold, the more quickly the filter eliminates protocols. However, if the threshold is too low, we may be too ambitious and not find any protocols. To determine a good threshold, consider the following protocol parameters
\[ \alpha_0 = \half \left[ 1, 0, 1 \right]^\transpose, \quad \alpha_1 = \half \left[ 0, 1, 1 \right]^\transpose, \quad \beta_0 = \left[ 1,0 \right]^\transpose, \quad  \beta_1 =\left[ 0,1 \right]^\transpose. \]
This is the four-round version of the optimal three-round protocol in Subsection~\ref{ex}. Numerically solving for the cheating probabilities for this protocol shows that
\[ P_{\rA,0}^* = P_{\rA,1}^* = P_{\rB,0}^* = P_{\rB,1}^* = 3/4. \]
Thus, there exists a protocol with the same bias as the best-known explicit
coin-flipping protocol constructions. This suggests that we use a threshold around $0.75$. Preliminary tests show that using a threshold of $0.75$ or larger  is much slower than a value of $0.7499$. This is because using the larger threshold allows protocols with optimal cheating probabilities (or filter cheating probabilities) of $0.75$ to slip through the filter and these protocols are no better than the one mentioned above (and many are just higher dimensional embeddings of it). Therefore, we use a threshold of $0.7499$. (Tests using a threshold of slightly larger than $0.75$ are considered in Subsection~\ref{ssect:zoningin}.)

Using these ideas, we now state the search algorithm.

\begin{center}
\quad \\
\begin{tabular}{|l|}
\hline \\
\quad \underline{Search algorithm for finding the best protocol parameters} \quad \quad \\
\quad \\
\quad Fix a dimension $d$ and mesh precision $\nu$. \quad \\
\quad For each protocol in the mesh (modulo the symmetry): \quad \\
\quad $\bullet$ Use the Protocol Filter to eliminate (some) protocols with bias above $0.2499$. \quad \\
\quad $\bullet$ Calculate the optimal cheating probabilities by solving the SDPs. \quad \\
\quad \phantom{$\bullet$} If any are larger than $0.7499$, move on to the next protocol. \quad \\
\quad \phantom{$\bullet$} Else, output the protocol parameters with bias $\epsilon < 1/4$. \\
\quad \\
\hline
\end{tabular}
\quad \\
\quad \\
\end{center}

We test the algorithm on the cases of four and six-round protocols and for certain dimensions and precisions for the mesh. These are presented in detail next.


\section{Numerical results}\label{sect:numerical}

\paragraph{Computational Platform.}
We ran our programs on Matlab, Version 7.12.0.635, on an SGI XE C1103 with 2x 3.2 GHz 4-core Intel X5672 x86 CPUs processor, and 10 GB memory, running Linux.

We solved the semidefinite programs using SeDuMi 1.3, a program for solving semidefinite programs in Matlab \cite{Stu99, Stu02}.

Sample programs can be found at the following link: \\ \quad \\ 
\url{http://www.math.uwaterloo.ca/~anayak/coin-search/}

\subsection{Four-round search}

We list the filter cheating strategies in Tables~\ref{4round_running_times1} and \ref{4round_running_times2}
which also give an estimate of how long it
takes the program to compute the success probability for each strategy based
on the average over $1000$ random instances (i.e. four randomly chosen probability vectors $\alpha_0$, $\alpha_1$, $\beta_0$, and $\beta_1$.)

\begin{sidewaystable}
\caption{Average running times for filter strategies for a $4$-round protocol when $d=5$ over
random protocol states (1 of 2).}
\label{4round_running_times1}
\quad
\begin{center}
\begin{tabular}{|r|c|c|}
\hline
\comment{Strategy Description & } Success Probability & Comp. Time (s) & Code \\
\hline
\hline
\quad & \quad & \quad \\
\comment{Bob's ignoring strategy & } $\half + \half \sqrt{\rF(\beta_0, \beta_1)}$
&            $0.000034429$               & F1 \\
\quad & \quad & \quad \\
\comment{Bob's measuring strategy & } $\half + \half \Delta(\alpha_0, \alpha_1)$
&       $0.000004640$                & F2 \\
\quad & \quad & \quad \\
\comment{Alice's three-round strategy & } $\left( \half + \half \sqrt{\rF(\alpha_0, \alpha_1)} \right) \left( \half + \frac{1}{2} \Delta(\beta_0, \beta_1) \right)$
&           $0.000025980$              & F3 \\
\quad & \quad & \quad \\
\comment{Bob's returning strategies & } \quad $\half \sum_{a \bit} \rF(\alpha_a, \beta_a)$
&      $0.000023767$            & F4 \\
\quad & \quad & \quad \\
\comment{ & } \quad $\half \sum_{a \bit} \rF(\alpha_a, \beta_{\bar{a}})$
&            $0.000018019$           & F5 \\
 \quad & \quad & \quad \\
\comment{Alice's eigenstrategy & } $\half \lambda_{\max} \left( \left( \displaystyle\sum_{y: \beta_{0,y} \geq \beta_{1,y}} \beta_{0,y} \right) \sqrtt{\alpha_{0}} + \left( \displaystyle\sum_{y: \beta_{0,y} < \beta_{1,y}} \beta_{1,y} \right) \sqrtt{\alpha_{1}} \right)$
&      $0.000036613$          & F6  \\
($\sqrt{v}$ is the principal eigenvector) & & \\
\quad & \quad & \quad \\
\comment{Bob's eigenstrategy lower bound & } $\frac{1}{2} \sum_{x \in A} \lambda_{\max} \left( \sum_{a \bit} \alpha_{a,x} \sqrtt{\beta_a}  \right)$
&            $0.000073010$            & F7  \\
($\sqrt{v_x}$ is the principal eigenvector for each $x \in A$) & & \\
\quad & \quad & \quad \\
\comment{Bob's eigenstrategy & } $\half \sum_{a \bit} \rF\! \left( \sum_{x \in A} \alpha_{a,x} (v_x), \beta_a \right)$
&    $0.000697611$       & F8 \\
\quad & \quad & \quad \\
\comment{Bob's eigenstrategy & } $\half \sum_{a \bit} \rF\! \left( \sum_{x \in A} \alpha_{a,x} (v_x), \beta_{\bar a} \right)$
&       $0.000532954$    & F9 \\
& & \\
\hline
\end{tabular}
\end{center}
\end{sidewaystable}

\begin{sidewaystable}
\caption{Average running times for filter strategies for a $4$-round protocol when $d=5$ over random protocol states (2 of 2).}
\label{4round_running_times2}
\quad
\begin{center}
\begin{tabular}{|r|c|c|}
\hline
\comment{Strategy Description & } Success Probability & Comp. Time (s) & Code \\
\hline
\hline
\quad & \quad & \quad \\
\comment{Alice's improved eigenstrategy & } $\sum_{y \in B} \conc \set{ \frac{1}{2} \beta_{0,{y}} \rF(\cdot, \alpha_0), \frac{1}{2} \beta_{1,{y}} \rF(\cdot, \alpha_1)}(v)$
&        $0.122971205$            & F10 \\
\quad & \quad & \quad \\
\comment{Alice's improved eigenstrategy & } $\sum_{y \in B} \conc \set{ \frac{1}{2} \beta_{1,{y}} \rF(\cdot, \alpha_0), \frac{1}{2} \beta_{0,{y}} \rF(\cdot, \alpha_1)}(v)$
&         $0.123375678$           & F11 \\
\quad & \quad & \quad \\
\comment{Alice's optimal strategy & } $P_{\rA,0}^*$
&             $0.149814373$          & SDPA0 \\
\quad & \quad & \quad \\
\comment{Bob's implicit strategies & } $\dfrac{1}{2 P_{\rA,0}^*}$
&      $0.000000947$                & F12 \\
\quad & \quad & \quad \\
\comment{Bob's optimal strategies & } $P_{\rB,0}^*$
&            $0.070846378$          & SDPB0 \\
\quad & \quad & \quad \\
\comment{Alice's optimal strategy & } $P_{\rA,1}^*$
&        $0.149176117$               & SDPA1 \\
\quad & \quad & \quad \\
\comment{Bob's implicit strategies & } $\dfrac{1}{2 P_{\rA,1}^*}$
&      $0.000000760$                & F13 \\
\quad & \quad & \quad \\
\comment{ & } $P_{\rB,1}^*$
&         $0.070479449$          & SDPB1 \\
\quad & \quad & \quad \\
\hline
\end{tabular}
\end{center}
\end{sidewaystable}

Notice the two strategies with
codes F1 and F2 are special because they only involve two of the four probability distributions. Preliminary tests show that first generating $\beta_0$ and $\beta_1$ and checking with F1 is much faster than first generating $\alpha_0$ and $\alpha_1$ and checking with F2, even though F2 is much faster to compute.

We can similarly justify the placement of $P_{\rA,0}^*$ before $P_{\rB,0}^*$ or $P_{\rB,1}^*$. The strategies F8 and F9 perform very well and the cheating probabilities are empirically very close to $P_{\rB,0}^*$ and $P_{\rB,1}^*$. Thus, if a protocol gets through the F8 and F9 filter strategies, then  it is likely that $P_{\rB,0}^*$ and $P_{\rB,1}^*$ are also less than $0.7499$. This is why we place $P_{\rA,0}^*$ first (although it will be  shown that the order of solving the SDPs does not matter much).

Recall from Subsection~\ref{SDPvsSOCP} that we solve for $P_{\rB,0}^*$, $P_{\rB,1}^*$, $P_{\rA,0}^*$, and $P_{\rA,1}$ using the semidefinite programming formulations of the reduced problems.

We then give tables detailing how well the filter performs for
four-round protocols, by counting the number of protocols that are
\emph{not\/} determined to have bias greater than~$0.2499$ by each prefix
of cheating strategies. We test four-round protocols with message
dimension $d \in \set{2, \ldots, 9}$ and precision $\nu$ ranging up to $1/2000$ (depending on $d$).

\begin{sidewaystable}
\caption{The number of protocols that get past each strategy in the filter for $d=2$.}
\quad
\begin{center}
\begin{tabular}{|r||r|r|r|r|r|r|r|r|r|r|r|r|r|r|r|r|}
\hline
$d = 2$ 	&  \comment{p = 1/100 &  		 p = 1/200 			 			& p = 1/350 &}			 $\nu = 1/500$ &	$\nu = 1/1000$ & $\nu = 1/1250$ & $\nu = 1/1500$ & $\nu = 1/2000$ 	\\
\hline
\hline
Protocols 	& \comment{1.46e+04 			& 1.04e+08 &	 1.63e+09 &}	 $6.30 \; \e \! + \! 10$ & $1.00 \; \e \! + \! 12$ & 	$2.44 \; \e \! + \! 12$ & $5.07 \; \e \! + \! 12$ & $1.60 \; \e \! + \! 13$ \\
\hline
Symmetry & \comment{6,765,201  & 104,060,401  & 959,512,576 &}   $3,969,126,001$  & $63,001,502,001$ &  $153,566,799,376$  & $318,097,128,001$  & $1,004,006,004,001$ \\
\hline
F1 &    \comment{195,075  & 2,713,466  & 23,851,520 &} $96,706,535$  & $1,499,479,974$  & $3,636,609,280$  & $7,506,289,309$  & $23,607,143,560$ \\
\hline
F2 &   \comment{144,375  & 2,021,600  & 17,820,880 &}  $72,336,875$  & $1,123,112,000$  & $2,724,552,320$  & $5,624,716,125$  & $17,693,560,000$ \\
\hline
F3 &                   \comment{0      &             0  & 3 &}   $5$ & $27$ & $50$ & $67$ & $124$ \\
\hline
F4 &                   \comment{0      &             0      &             0  &}                 $0$     &              $0$      &             $0$      &             $0$       &            $0$ \\
\hline
F5 &                   \comment{0      &             0      &             0  &}                 $0$     &              $0$      &             $0$      &             $0$       &            $0$ \\
\hline
F6 &                   \comment{0      &             0      &             0  &}                 $0$     &              $0$      &             $0$      &             $0$       &            $0$ \\
\hline
F7 &                   \comment{0      &             0      &             0  &}                 $0$     &              $0$      &             $0$      &             $0$       &            $0$ \\
\hline
F8 &                   \comment{0      &             0      &             0  &}                 $0$     &              $0$      &             $0$      &             $0$       &            $0$ \\
\hline
\end{tabular}
\end{center}

\caption{The number of protocols that get past each strategy in the filter for $d=3$.}
\quad
\begin{center}
\begin{tabular}{|r||r|r|r|r|r|r|r|r|r|r|r|r|r|r|r|r|}
\hline
$d = 3$ 	& $\nu = 1/5$ & $\nu = 1/10$ & $\nu = 1/20$ & $\nu = 1/30$ & $\nu = 1/50$ \\		
\hline
\hline
Protocols 	&  $1.94 \; \e \! + \! 05$ &  $1.89 \; \e \! + \! 07$ 	& $2.84 \; \e \! + \! 09$ 	& $6.05 \; \e \! + \! 10$ & $3.09 \; \e \! + \! 12$ \\
\hline
Symmetry &   $4,356$ &  $272,484$ &  $29,430,625$ &  $55,436,7025$ &  $25,475,990,544$ \\
\hline
F1 &  $1,254$ &  $37,584$ &  $2,175,425$ &  $30,985,220$  & $1,020,080,292$ \\
\hline
F2  & $665$ &  $19,656$ &  $1,300,042$  & $19,366,256$ & $662,158,728$ \\
\hline
F3  & $49$ &  $470$ & $22,282$ &  $225,098$  & $4,414,994$ \\
\hline
F4  & $29$ &  $261$ &  $11,667$  & $110,931$ &  $2,028,518$ \\
\hline
F5  & $28$ &  $258$  & $11,495$  & $109,515$ &  $2,009,141$ \\
\hline
F6  & $28$ &  $241$ &  $10,405$  & $96,464$  & $1,765,114$ \\
\hline
F7   &                $0$    &    $3$  & $54$  & $148$  & $1,158$ \\
\hline
F8    &               $0$        &           $0$           &        $0$      &             $0$      &             $0$ \\
\hline
\end{tabular}
\end{center}
\end{sidewaystable}

\begin{sidewaystable}
\caption{The number of protocols that get past each strategy in the filter for $d=4$.}
\quad
\begin{center}
\begin{tabular}{|r||r|r|r|r|r|r|r|r|r|r|r|r|r|r|r|r|}
\hline
$d = 4$	& \comment{$\nu = 1/8$		&} $\nu = 1/10$ 			& $\nu = 1/12$ 			& $\nu = 1/16$ &  $\nu = 1/20$ & $\nu = 1/24$ & $\nu = 1/30$	\\
\hline
\hline
Protocols 	  & $  \comment{7.41e+08 	 $ & $ }  6.69 \; \e \! + \! 09 	 $ & $  4.28 \; \e \! + \! 10 	 $ & $  8.81 \; \e \! + \! 11  $ & $   	9.83 \; \e \! + \! 12  $ & $  7.31 \; \e \! + \! 13  $ & $  8.86 \; \e \! + \! 14	 $ 	\\
\hline
Symmetry   & $  \comment{2,053,489  $ & $ }  13,498,276  $ & $   74,166,544  $ & $   1,154,640,400  $ & $   10,334,552,281  $ & $   69,927,455,844  $ & $   736,486,643,344  $ \\
\hline
F1   & $   \comment{490,086  $ & $ }  2,432,188   $ & $  12,616,580   $ & $  146,114,000  $ & $   934,856,164  $ & $   5,916,006,936  $ & $      49,798,933,264  $ \\
\hline
F2   & $   \comment{183,312  $ & $ }  1,036,030   $ & $  5,616,810   $ & $  71,246,700  $ & $   489,282,376  $ & $   3,170,626,956  $ & $       27,760,130,976 $  \\
\hline
F3   & $   \comment{10,532  $ & $ }  66,623   $ & $  302,547   $ & $  3,185,895  $ & $   19,670,642  $ & $   101,703,667  $ & $      738,284,522 $  \\
\hline
F4   & $   \comment{7,732  $ & $ }  46,734   $ & $  209,747   $ & $  2,061,868  $ & $   12,000,187  $ & $   59,503,895  $ & $      406,963,112 $  \\
\hline
F5   & $   \comment{7,709  $ & $ }  46,531   $ & $  208,961   $ & $  2,054,891  $ & $   11,962,104  $ & $   59,353,374  $ & $     406,099,637 $  \\
\hline
F6   & $   \comment{7,071  $ & $ }  42,591   $ & $  198,192   $ & $  1,886,782  $ & $   11,004,125  $ & $   54,702,075  $ & $      367,847,304 $  \\
\hline
F7   & $   \comment{83  $ & $ }  329  $ & $   756   $ & $  3,439   $ & $  17,144  $ & $   55,929  $ & $      190,699 $ \\
\hline
F8   & $                    \comment{0  $ & $ }               0  $ & $                   0  $ & $                  0  $ & $                   0  $ & $                   0  $ & $  0 $ \\
\hline
\end{tabular}
\end{center}

\caption{The number of protocols that get past each strategy in the filter for $d=5$.}
\begin{center}
\begin{tabular}{|r||r|r|r|r|r|r|r|r|r|r|r|r|r|r|r|r|}
\hline
$d = 5$ 	  &   $\nu = 1/5$ 			  &    $\nu = 1/8$	  &   $\nu = 1/10$   &   $\nu = 1/12$	\\
\hline
\hline
Protocols 	  & $  2.52 \; \e \! + \! 08	 $ & $   6.00 \; \e \! + \! 10  $ & $  1.00 \; \e \! + \! 12  $ & $  1.09 \; \e \! + \! 13	 $  \\
\hline
Symmetry   & $   240,100   $ & $  29,539,225  $ & $   284,529,424  $ & $   2,485,919,881 $  \\
\hline
F1   & $   105,840  $ & $   9,467,770  $ & $   66,257,504  $ & $   567,544,997 $  \\
\hline
F2   & $   37,584  $ & $   2,687,906   $ & $  22,774,544  $ & $   203,983,360 $  \\
\hline
F3   & $   8,561  $ & $   241,420  $ & $   2,440,765  $ & $   17,794,655 $  \\
\hline
F4   & $   7,423  $ & $   201,569  $ & $   1,937,298   $ & $  13,682,059 $  \\
\hline
F5   & $   7,417   $ & $  200,965   $ & $  1,933,833  $ & $   13,665,087 $  \\
\hline
F6   & $   7,417   $ & $  189,144  $ & $   1,790,144  $ & $   13,117,165 $  \\
\hline
F7   & $                   0  $ & $  1,415   $ & $  10,790  $ & $   43,459 $  \\
\hline
F8       & $               0     $ & $                0       $ & $              0          $ & $           0 $  \\
\hline
\end{tabular}
\end{center}
\end{sidewaystable}

\begin{sidewaystable}
\caption{The number of protocols that get past each strategy in the filter for $d=6$.}
\quad
\begin{center}
\begin{tabular}{|r||r|r|r|r|r|r|r|r|r|r|r|r|r|r|r|r|}
\hline
$d = 6$ 	  &   $\nu = 1/7$   &   $\nu = 1/8$   &   $\nu = 1/9$   &   $\nu=1/10$   &   $\nu = 1/11$   &   $\nu = 1/12$ \\
\hline
\hline
Protocols   & $   3.93 \; \e \! + \! 11  $ & $  2.74 \; \e \! + \! 12  $ & $  1.60 \; \e \! + \! 13  $ & $  8.13 \; \e \! + \! 13  $ & $  3.64 \; \e \! + \! 14  $ & $  1.46 \; \e \! + \! 15 $  \\
\hline
Symmetry   & $   53,144,100   $ & $  265,950,864    $ & $  1,021,825,156   $ & $  3,534,302,500   $ & $  12,577,398,201   $ & $  46,107,255,076 $  \\
\hline
F1   & $   25,070,310   $ & $  107,583,876    $ & $  387,459,886   $ & $  1,034,786,700   $ & $  3,605,814,648   $ & $  13,370,558,568 $  \\
\hline
F2   & $   7,276,924   $ & $  23,294,007    $ & $  123,246,328   $ & $  287,251,218   $ & $  1,330,224,696   $ & $  3,841,063,848  $ \\
\hline
F3   & $   1,744,038   $ & $  2,811,374    $ & $  25,114,451   $ & $  42,503,208   $ & $  258,455,916   $ & $  468,218,324  $ \\
\hline
F4   & $   1,551,522   $ & $  2,526,900    $ & $  21,682,087   $ & $  36,628,517   $ & $  214,823,642   $ & $  390,846,158  $ \\
\hline
F5   & $   1,550,617   $ & $  2,524,052    $ & $  21,666,437   $ & $  36,594,682   $ & $  214,698,072   $ & $  390,649,931 $  \\
\hline
F6   & $   1,451,038   $ & $  2,419,474    $ & $  20,598,749   $ & $  34,117,986   $ & $  203,605,433   $ & $  377,899,946  $ \\
\hline
F7   & $  9,169  $ & $    13,976  $ & $   57,720  $ & $   174,118  $ & $   526,077  $ & $   1,153,864  $  \\
\hline
F8   & $       0                  $ & $   0                 $ & $    0                 $ & $    0                 $ & $    0                $ & $     0 $  \\
\hline
\end{tabular}
\end{center}

\caption{The number of protocols that get past each strategy in the filter for $d=7$.}
\quad
\begin{center}
\begin{tabular}{|r||r|r|r|r|r|r|r|r|r|r|r|r|r|r|r|r|}
\hline
$d = 7$ 	  &   $\nu = 1/5$   &   $\nu=1/6$   &   $\nu = 1/7$   &   $\nu = 1/8$   &   $\nu = 1/9$   &   $\nu = 1/10$ \\
\hline
\hline
Protocols   & $  4.55 \; \e \! + \! 10  $ & $
7.28 \; \e \! + \! 11  $ & $  8.67 \; \e \! + \! 12  $ & $  8.13 \; \e \! + \! 13  $ & $  6.27 \; \e \! + \! 14  $ & $  4.11 \; \e \! + \! 15  $  \\
\hline
Symmetry   & $   3,709,476  $ & $   46,963,609  $ & $   289,374,121  $ & $   1,730,643,201  $ & $   7,402,021,225  $ & $   30,490,398,225  $ \\
\hline
F1   & $   2,270,754  $ & $   26,952,849  $ & $   161,111,181  $ & $   841,297,023  $ & $   3,456,456,125  $ & $   10,915,707,495  $ \\
\hline
F2   & $   495,180  $ & $   3,154,266  $ & $   36,330,756  $ & $   136,788,372  $ & $   851,509,125  $ & $   2,419,940,743 $  \\
\hline
F3   & $   149,806  $ & $   369,434  $ & $   10,277,699  $ & $   20,469,535  $ & $   216,148,269  $ & $   449,464,967  $ \\
\hline
F4   & $   142,255  $ & $   351,290  $ & $   9,583,747  $ & $   19,200,670  $ & $   197,250,330  $ & $   409,366,494 $  \\
\hline
F5   & $   142,241  $ & $   351,219  $ & $   9,582,215  $ & $   19,194,692  $ & $   197,214,454  $ & $   409,185,885 $  \\
\hline
F6   & $   142,241  $ & $   351,219  $ & $   9,034,728  $ & $   18,734,072  $ & $   187,977,589  $ & $   383,402,064 $  \\
\hline
F7   & $                   0  $ & $                   0  $ & $   60,155  $ & $   91,787  $ & $   512,171  $ & $   1,804,382 $  \\
\hline
F8   & $                   0  $ & $                   0  $ & $                   0  $ & $                   0  $ & $                   0  $ & $                   0  $ \\
\hline
\end{tabular}
\end{center}
\end{sidewaystable}

\begin{sidewaystable}
\caption{The number of protocols that get past each strategy in the filter for $d=8$.}
\quad
\begin{center}
\begin{tabular}{|r||r|r|r|r|r|r|r|r|r|r|r|r|r|r|r|r|}
\hline
$d = 8$ 	  &   $\nu = 1/4$   &   $\nu = 1/5$   &   $\nu = 1/6$ &  $\nu = 1/7$  &  $\nu = 1/8$  &  $\nu = 1/9$ \\
\hline
\hline
Protocols 	   & $  1.18 \; \e \! + \! 10  $ & $  3.93 \; \e \! + \! 11  $ & $  8.67 \; \e \! + \! 12  $ & $  1.38 \; \e \! + \! 14  $ & $  1.71 \; \e \! + \! 15  $ & $   1.71 \; \e \! + \! 16 $   \\
\hline
Symmetry    & $    1,572,516  $ & $   11,532,816  $ & $   179,345,664  $ & $   1,293,697,024  $ & $   9,018,161,296  $ & $  42,352,405,209 $  \\
\hline
F1   & $   1,054,614  $ & $   7,797,216  $ & $   115,131,024  $ & $   814,855,040  $ & $   5,050,850,268  $ & $  23,061,817,617 $  \\
\hline
F2    & $   60,552  $ & $   1,356,936  $ & $   9,766,192  $ & $   142,862,430  $ & $   606,597,735  $ & $  4,417,668,742 $  \\
\hline
F3    & $                   0  $ & $   431,956  $ & $   1,254,420  $ & $   44,457,239  $ & $   106,851,420  $ & $  1,276,499,496 $  \\
\hline
F4    & $                   0  $ & $   417,759  $ & $   1,213,728  $ & $   42,541,702  $ & $   102,719,851  $ & $  1,204,238,273 $  \\
\hline
F5    & $                   0  $ & $   417,741  $ & $   1,213,629  $ & $   42,539,430  $ & $   102,710,139  $ & $  1,204,173,244 $  \\
\hline
F6    & $                   0  $ & $   417,741  $ & $   1,213,629  $ & $   40,425,272  $ & $   101,061,706  $ & $  1,151,097,965  $ \\
\hline
F7    & $                   0  $ & $                   0  $ & $                   0  $ & $   277,225  $ & $   452,792  $ & $  3,194,346  $ \\
\hline
F8    & $                   0  $ & $                   0  $ & $                   0  $ & $                   0  $ & $                   0  $ & $  0 $  \\
\hline
\end{tabular}
\end{center}

\caption{The number of protocols that get past each strategy in the filter for $d=9$.}
\quad
\begin{center}
\begin{tabular}{|r||r|r|r|r|r|r|r|r|r|r|r|r|r|r|r|r|}
\hline
$d = 9$ 	  &   $\nu = 1/3$   &   $\nu = 1/4$   &   $\nu = 1/5$   &   $\nu = 1/6$   &   $\nu = 1/7$   &   $\nu = 1/8$ \\
\hline
\hline
Protocols   & $  7.41 \; \e \! + \! 08  $ & $  6.00 \; \e \! + \! 10  $ & $  2.74 \; \e \! + \! 12  $ & $  8.13 \; \e \! + \! 13  $ & $  1.71 \; \e \! + \! 15  $ & $  2.74 \; \e \! + \! 16 $  \\
\hline
Symmetry   & $   164,025  $ & $   3,744,225  $ & $   32,069,569  $ & $   594,433,161  $ & $  4,957,145,649  $ & $  39,808,629,441 $  \\
\hline
F1   & $   131,625  $ & $   2,666,430  $ & $   23,348,549  $ & $   414,160,047  $ & $   3,423,681,189  $ & $  24,851,338,155 $  \\
 \hline
F2   & $   14,300  $ & $   115,752  $ & $   3,273,662  $ & $   26,075,045  $ & $   470,028,582  $ & $  2,216,082,560 $  \\
 \hline
F3   & $   2,700   $ & $                  0  $ & $   1,065,271  $ & $   3,484,092  $ & $   153,932,946  $ & $  432,754,976 $  \\
 \hline
F4   & $   2,639   $ & $                  0  $ & $   1,041,339  $ & $   3,405,532  $ & $  149,523,487  $ & $  421,903,500 $  \\
 \hline
F5   & $   2,639   $ & $                  0  $ & $   1,041,317  $ & $   3,405,403  $ & $   149,520,361  $ & $  421,889,260 $  \\
 \hline
F6   & $  2,639   $ & $                  0  $ & $   1,041,317  $ & $   3,405,403  $ & $   142,916,565  $ & $  416,869,327 $  \\
 \hline
F7   & $                   0           $ & $          0   $ & $                  0   $ & $                  0  $ & $   1,053,222  $ & $  1,809,800 $  \\
 \hline
F8   & $                   0           $ & $          0   $ & $                  0   $ & $                  0    $ & $                 0  $ & $  0  $ \\
\hline
\end{tabular}
\end{center}
\end{sidewaystable}


\newpage
\paragraph{Observations on the four-round search} \quad \\

We were able to search larger spaces than feasible with the SDP formulations alone. For example, suppose we took the $2.74 \times 10^{16}$ protocols from the $d = 9$, $\nu = 1/8$ search and checked to see if any of these had bias less than $0.7499$ by solving only the reduced SDPs. Since each SDP takes at least $0.08$ seconds to solve, this search would take at least $69$ million years to finish. By applying the techniques in this paper, we were able to run this search in a matter of days.

We see that symmetry helped dramatically reduce the number of protocols that needed to be tested. In the largest search, we were able to cut the $2.74 \times 10^{16}$ protocols down to $3.98 \times 10^{10}$. F1 and F2 perform very well, together cutting down the number of protocols by a factor of about $10$. An interesting observation is that F2 performs much better than F1, and is also $10$ times faster to compute. It may seem better to put F2 before F1 in the tests, however, we place F1 first since it is beneficial to have the more expensive strategy being computed first. This way, it only needs to be computed for every choice of $\beta_0$ and $\beta_1$. If we were to calculate F2 first, we would have to calculate F1 on $(\alpha_0, \alpha_0, \beta_0, \beta_1)$ for every $\alpha_0, \alpha_1$ that F2 did not filter out.

Being the first strategy to rely on all four probability distributions, F3 performs very well by reducing the number of protocols by another factor of $10$. F4, F5, and F6 do not perform well (F5 being the same as F4 but with $\beta_0$ swapped with $\beta_1$);
they cut down the number of protocols by a very small number. F7 and F8 perform so well that no SDPs were  needed to be solved.

These numbers suggest a conjecture along the lines of
\[ \min_{\alpha_0, \alpha_1, \beta_0, \beta_1 \in \prob^9} \max \set{\mathrm{F1}, \ldots, \mathrm{F8}} \geq 0.7499. \]
However, we shall soon see computational evidence in
Subsection~\ref{ssect:zoningin} showing that this may not be true if we replace $0.7499$ with $0.75$ and conduct zoning-in searches with
much higher precision.

\subsection{Six-round search}

We list the filter cheating strategies in the tables on the next two pages and give an estimate for how long it
takes to compute the success probability for each strategy by taking the average over $1000$ random instances. We then give tables of how well the filter performs for six-round protocols with $d \in \set{2, 3}$ and $\nu$ as small as $1/15$ for $d=2$ and $1/4$ for $d=3$.
The measure of performance of the filter that we use is as before.
For each prefix of cheating strategies in the filter, we
count the number of protocols in the mesh that are \emph{not\/} determined
to have bias greater than~$0.2499$ by that prefix.

\begin{sidewaystable}
\caption{Average running times for filter strategies in a six-round protocol for $d = 3$ over random protocol states (1 of 2).}
\label{6round_running_times1} 
\quad
\begin{center}
\begin{tabular}{|r|c|c|c|}
\hline
\comment{Strategy Description  $ & $  } Success Probability   &   Comp. Time (s)   &   Code \\
\hline
\hline
\quad   &   \quad   &   \quad \\
$\half + \half \sqrt{\rF(\beta_0, \beta_1)}$
  & $      0.000036128   $ &   G1 \\
 \quad   &   \quad   &   \quad \\
$\half + \half \Delta( \tr_{A_2}(\alpha_0), \tr_{A_2}(\alpha_1))$
  & $      0.000005552             $ &   G2  \\
\quad   &   \quad   &   \quad \\
$\half \lambda_{\max} \left( \kappa \sqrtt{\tr_{B_2} (\beta_0)} + \zeta \sqrtt{\tr_{B_2} (\beta_1)} \right)$
  & $    0.000015667      $ &   G3 \\
where $\kappa  := \sum_{x: \alpha_{0,x} \geq \alpha_{1,x}} \alpha_{0,x}$ and $\zeta := \sum_{x: \alpha_{0,x} < \alpha_{1,x}} \alpha_{1,x}$   &     &   \\
 \quad   &   \quad   &   \quad \\
$\left( \half + \half \sqrt{\rF(\tr_{A_2}(\alpha_0), \tr_{A_2}(\alpha_1))} \right) \left( \half + \half \Delta (\tr_{B_2}(\beta_0), \tr_{B_2}(\beta_1)) \right)$
  & $     0.000028408    $ &   G4  \\
 \quad   &   \quad   &   \quad \\
$\half \lambda_{\max} \left( \left( \displaystyle\sum_{y: \beta_{0,y} \geq \beta_{1,y}} \beta_{0,y} \right) \sqrtt{\alpha_{0}} + \left( \displaystyle\sum_{y: \beta_{0,y} < \beta_{1,y}} \beta_{1,y} \right) \sqrtt{\alpha_{1}} \right)$
  & $         0.000052325          $ &   G5  \\
 \quad   &   \quad   &   \quad \\
$\half \lambda_{\max} \left( \eta' \sqrtt{\tr_{A_2} (\alpha_{0})} + \tau' \sqrtt{\tr_{A_2} (\alpha_{1})} \right)$
  & $         0.000044243          $ &   G6  \\
 \quad   &   \quad   &   \quad \\
where
$\eta' := \sum_{y_1 \in B_1: [\tr_{B_2}(\beta_{0})]_{y_1} \geq  [\tr_{B_2}(\beta_{1})]_{y_1}}  [\tr_{B_2}(\beta_{0})]_{y_1}$,    &     &   \\
and
$\tau' := \sum_{y_1 \in B_1:  [\tr_{B_2}(\beta_{0})]_{y_1} <  [\tr_{B_2}(\beta_{1})]_{y_1}}  [\tr_{B_2}(\beta_{1})]_{y_1}$   &     &   \\
\hline
\end{tabular}
\end{center}
\end{sidewaystable}

\begin{sidewaystable}
\caption{Average running times for filter strategies in a six-round protocol for $d = 3$ over
random protocol states (2 of 2).}
\label{6round_running_times2} 
\quad
\begin{center}
\begin{tabular}{|r|c|c|c|}
\hline
\comment{Strategy Description  $ & $  } Success Probability   &   Comp. Time (s)   &   Code \\
\hline
\hline
 \quad   &   \quad   &   \quad \\
$\half \sum_{a \in A'_0} \rF\! \left( \sum_{x \in A} \alpha_{a,x} \tilde{p_2}^{(x)}, \beta_a \right) $
  & $      0.000879119    $ &   G7 \\
 \quad   &   \quad   &   \quad \\
\comment{Alice's improved eigenstrategy  $ & $  } $\half \sum_{a \in A'_0} \rF\! \left( \sum_{x \in A} \alpha_{a,x} \tilde{p_2}^{(x)}, \beta_{\bar a} \right) $
  & $     0.000797106     $ &   G8 \\
where $\tilde{p_2}^{(x)}$ is as defined in Theorem~\ref{BobFilter}
 \quad   &   \quad   &   \quad \\
 \quad   &   \quad   &   \quad \\
\comment{Alice's improved eigenstrategy  $ & $  } $\sum_{y \in B} \conc \set{ \frac{1}{2} \beta_{0,{y}} \rF(\cdot, \alpha_0), \frac{1}{2} \beta_{1,{y}} \rF(\cdot, \alpha_1)}(v)$
  & $    0.256377981      $ &   G9 \\
 \quad   &   \quad   &   \quad \\
\comment{Alice's improved eigenstrategy  $ & $  } $\sum_{y \in B} \conc \set{ \frac{1}{2} \beta_{1,{y}} \rF(\cdot, \alpha_0), \frac{1}{2} \beta_{0,{y}} \rF(\cdot, \alpha_1)}(v)$
  & $    0.249946219      $ &   G10 \\
($\sqrt{v}$ is the principal eigenvector of the matrix in G5)
 \quad   &   \quad   &   \quad \\
 \quad   &   \quad   &   \quad \\
\comment{Bob's optimal strategies  $ & $  } $P_{\rB,0}^*$
  & $     0.164744870    $ &   SDPB0 \\
 \quad   &   \quad   &   \quad \\
\comment{Bob's strategy from Kitaev  $ & $  } $\dfrac{1}{2 P_{\rB,0}^*}$
  & $    0.000000996     $ &   G11 \\
 \quad   &   \quad   &   \quad \\
\comment{Alice's optimal strategy  $ & $  } $P_{\rA,0}^*$
  & $    0.276034548     $ &   SDPA0 \\
 \quad   &   \quad   &   \quad \\
\comment{ $ & $  } $P_{\rB,1}^*$
  & $     0.162818974     $ &   SDPB1 \\
 \quad   &   \quad   &   \quad \\
\comment{Bob's strategy from Kitaev  $ & $  } $\dfrac{1}{2 P_{\rB,1}^*}$
  & $   0.000001075      $ &   G12 \\
 \quad   &   \quad   &   \quad \\
\comment{Alice's optimal strategy  $ & $  } $P_{\rA,1}^*$
  & $    0.271631913     $ &   SDPA1 \\
\hline
\end{tabular}
\end{center}
\end{sidewaystable}


Again, we choose which strategy to put first, G1 or G2. Preliminary tests show that placing G1 first results
in a much faster search, similar to the four-round case. Even though G5 takes longer to compute than G6, tests show that it is better to have G5 first. We calculate $P^*_{\rB,0}$ before $P^*_{\rA,0}$ since G9 and G10 are close approximations of $P^*_{\rA,0}$ and $P^*_{\rA,1}$, respectively. It will be evident that the order of solving the SDPs does not matter much.

We note here a few omissions as compared to the four-round tests. First, we have removed the two returning strategies, F4 and F5. These did not perform well in the four-round tests and preliminary tests show that they did not perform well in the six-round search either. Also, we do not have all the lower bounds for the eigenstrategies. Preliminary tests show that the lower bounds omitted take just as long or longer to compute than the corresponding upper bound, thus we just use the upper bound in the filter. Also, the marginal probabilities take approximately $5.49 \times 10^{-6}$ seconds to compute which is negligible compared to the other times. Thus, we need not be concerned whether the strategies rely on the full probability distributions or marginal distributions.

\begin{sidewaystable}
\caption{The number of protocols that get past each strategy in the filter for $d=2$.}
\quad
\begin{center}
\begin{tabular}{|r||r|r|r|r|r|r|r|r|r|r|r|r|r|r|r|r|}
\hline
$d = 2$	   &   $\nu = 1/3$   &   $\nu = 1/4$   &   $\nu = 1/5$   &   $\nu = 1/6$   &   $\nu = 1/7$   &   $\nu = 1/8$ \\
\hline
\hline
Protocols    & $    160,000   $ & $  1,500,625  $ & $  9,834,496   $ & $  49,787,136    $ & $   207,360,000   $ & $  7.41 \; \e \! + \! 08 $ \\
\hline
Symmetry & $ 6,400 $ & $ 59,049  $ & $    280,900   $ & $  1,517,824 $ & $    5,683,456  $ & $      19,713,600  $  \\
\hline
G1 & $         3,200  $ & $     20,412   $ & $    82,680  $ & $    389,312  $ & $   1,397,024 $ & $    4,115,880    $  \\
\hline
G2 & $         2,320    $ & $   12,516   $ & $    67,548  $ & $    272,392  $ & $   1,112,228 $ & $    3,057,246    $  \\
\hline
G3 & $         1,725      $ & $  9,627   $ & $    52,424  $ & $    223,034   $ & $   899,450 $ & $  2,526,712      $   \\
\hline
G4 & $          714      $ & $  4,206   $ & $    27,965  $ & $    105,050 $ & $     430,454 $ & $  1,240,106      $   \\
\hline
G5 & $          210     $ & $    684   $ & $     7,743   $ & $    20,373   $ & $   112,435 $ & $    228,274    $   \\
\hline
G6 & $          210    $ & $     684  $ & $      7,743  $ & $     20,373  $ & $    110,401 $ & $   228,274     $   \\
\hline
G7 & $           30     $ & $     48    $ & $    1,285  $ & $      1,856  $ & $     10,979 $ & $ 17,831       $   \\
\hline
G8 & $            0      $ & $     0    $ & $     466   $ & $      164   $ & $     3,427 $ & $  4,620      $   \\
\hline
G9 & $            0     $ & $      0    $ & $     466  $ & $       164   $ & $     3,419 $ & $     4,512   $   \\
\hline
G10 & $            0     $ & $      0    $ & $     466  $ & $       164   $ & $     3,369 $ & $   4,512     $   \\
\hline
SDPB0 & $            0     $ & $      0    $ & $       6    $ & $       0     $ & $     26 $ & $   20     $   \\
\hline
G11 & $            0     $ & $      0    $ & $       6    $ & $       0     $ & $     26 $ & $   20     $   \\
\hline
SDPA0 & $            0     $ & $      0    $ & $       0     $ & $      0     $ & $      0 $ & $   0     $   \\
\hline
\end{tabular}
\end{center}
\end{sidewaystable}

\begin{sidewaystable}
\caption{The number of protocols that get past each strategy in the filter for $d=2$.}
\quad
\begin{center}
\begin{tabular}{|r||r|r|r|r|r|r|r|r|r|r|r|r|r|r|r|r|}
\hline
$d = 2$ 	    &   $\nu = 1/9$   &   $\nu = 1/10$   &   $\nu = 1/11$   &   $\nu = 1/12$   &   $\nu = 1/13$   &   $\nu = 1/14$ &   $\nu = 1/15$  \\
\hline
\hline
Protocols   & $  2.34 \; \e \! + \! 09  $ & $  6.69 \; \e \! + \! 09  $ & $  1.75 \; \e \! + \! 10  $ & $  4.28 \; \e \! + \! 10  $ & $  9.83 \; \e \! + \! 10  $ & $   2.13 \; \e \! + \! 11 $ & $ 4.43 \; \e \! + \! 11 $ \\
\hline
Symmetry & $ 58,247,424 $ & $ 155,276,521 $ & $ 401,080,729 $ & $ 973,502,401 $ & $ 2,052,180,601 $ & $ 4,632,163,600 $ & $ 9,372,176,100 $ \\
\hline
G1 & $ 11,020,608 $ & $ 23,862,815 $ & $ 60,761,918 $ & $ 140,154,892 $ & $ 240,820,116 $ & $ 555,641,840 $ & $ 1,048,452,300 $  \\
\hline
G2 & $ 8,944,136 $ & $ 18,717,210 $ & $ 50,337,094 $ & $ 110,274,108 $ & $ 204,522,468 $ & $ 444,537,964 $  & $ 877,684,860 $ \\
\hline
G3 & $ 7,335,617 $ & $ 15,503,308 $ & $ 41,447,668 $ & $ 93,222,286 $ & $ 167,717,637 $ & $ 380,238,435 $ & $ 739,653,758 $  \\
\hline
G4 & $ 3,477,093 $ & $ 8,534,326 $ & $ 20,503,550 $ & $ 45,888,192 $ & $ 91,991,055 $ & $ 185,971,770 $  & $ 350,105,435 $ \\
\hline
G5 & $ 696,601 $ & $ 1,367,115 $ & $ 3,435,390 $ & $ 6,577,917 $ & $ 12,425,039 $ & $ 23,210,979 $  & $ 43,785,997 $ \\
\hline
G6 & $ 688,613 $ & $ 1,367,115 $ & $ 3,435,390 $ & $ 6,577,917 $ & $ 12,258,117 $ & $ 23,097,713 $  & $ 43,188,099 $ \\
\hline
G7 & $ 57,598 $ & $ 87,303 $ & $ 232,382 $ & $ 355,057 $ & $ 678,384 $ & $ 1,051,339 $  & $ 1,977,185 $ \\
\hline
G8 & $ 17,512 $ & $ 18,105 $ & $ 64,273 $ & $ 86,272 $ & $ 177,297 $ & $ 230,146 $  & $ 479,088 $ \\
\hline
G9 & $ 16,005 $ & $ 15,689 $ & $ 50,847 $ & $ 74,114 $ & $ 143,172  $ & $ 195,858 $  & $ 411,864 $ \\
\hline
G10 & $ 15,875 $ & $ 15,124 $ & $ 49,819 $ & $ 71,439 $ & $ 137,232 $ & $ 185,696 $  & $ 386,741 $ \\
\hline
SDPB0 & $ 68 $ & $ 58 $ & $ 152 $ & $ 126 $ & $ 492 $ & $ 346 $  & $ 594 $ \\
\hline
G11 & $ 68 $ & $ 58 $ & $ 152 $ & $ 126 $ & $ 492 $ & $ 346 $ & $ 594 $  \\
\hline
SDPA0 & $ 0 $ & $ 0 $ & $ 0 $ & $ 0 $ & $ 0 $ & $ 0 $ & $ 0 $  \\
\hline
\end{tabular}
\end{center}
\end{sidewaystable}

\begin{sidewaystable}
\caption{The number of protocols that get past each strategy in the filter for $d=3$.}
\quad
\begin{center}
\begin{tabular}{|r||r|r|r|r|r|r|r|r|r|r|r|r|r|r|r|r|}
\hline
$d = 3$	  & $\nu = 1/2$    &   $\nu = 1/3$   &   $\nu = 1/4$ \\
\hline
\hline
Protocols 	  & $ 4,100,625 $ & $ 741,200,625 $ & $ 60,037,250,625 $ \\
\hline
Symmetry   & $ 68,121 $ & $ 6,395,841 $ & $279,324,369$ \\
\hline
G1        & $ 42,282 $ & $ 5,222,385 $ & $180,500,400$ \\
\hline
G2        & $ 8,748 $ & $ 3,324,650 $ & $86,151,600$ \\
\hline
G3        & $ 5,643 $ & $ 1,958,070 $ & $58,038,667$ \\
\hline
G4         & $ 161 $ & $ 714,393 $ & $30,773,918$ \\
\hline
G5         & $ 0 $ & $ 464,538 $ & $15,310,116$ \\
\hline
G6         & $ 0 $ & $ 464,538 $ & $15,310,116$ \\
\hline
G7        & $ 0 $ & $ 310,518 $ & $6,557,007$ \\
\hline
G8         & $ 0 $ & $ 284,418 $ & $5,447,015$ \\
\hline
G9        & $ 0 $ & $ 284,418 $ & $5,393,911$ \\
\hline
G10       & $ 0 $ & $ 284,418 $ & $5,393,911$ \\
\hline
SDPB0       & $ 0 $ & $ 2,655 $ & $24,012$ \\
\hline
G11       & $ 0 $ & $ 2,655 $ & $24,012$ \\
\hline
SDPA0      & $ 0 $ & $ 0 $ & $0$ \\
\hline
\end{tabular}
\end{center}
\end{sidewaystable}


\newpage
\paragraph{Observations on the six-round search} \quad \\

We first note that the filter does not work as effectively as in the four-round case. The six-round search for $d=3$ ran for about a month. In comparison, all the four-round searches ran in the matter of days.

The symmetry arguments cut down the number of protocols we need to
examine significantly, by a factor of roughly $100$. Note that in the four-round case it was a factor of $1,000,000$ (for the $d=9$ case). This can be explained by the weaker index symmetry in the six-round version.

Cheating strategy G1 cut the number of protocols down by a factor of $10$ with G2 performing less well than the corresponding strategy in the four-round tests. G5 also performed well, but after this, G6 was not much help. G7 and G8 cut down the number of protocols by a factor of $10$ each in the $d=2$ case, but not as much in the $d=3$ case. The next notable strategy was G10, being G9 with $\beta_0$ and $\beta_1$ swapped, which performed very poorly. It seems that the swapped strategies do not help much in the filters, that is, there is not much discrepancy between cheating towards $0$ or $1$. SDPB0 almost filtered out the rest of the protocols, relying on SDPA0 to stop the rest. The implicit strategy from Kitaev's bound, G11, did not perform well after SDPB0 (note that it relies on SDPB0 so it is computed afterwards). Again, we notice that no protocols with bias less than $0.2499$ were found.

We notice that G9 and G10, the improved eigenstrategies for Alice,
hardly filter out any protocols, if any at all, in the low-precision
tests. In these strategies, we compute a value on the concave hull
\[ \conc \left\{ \frac{1}{2} \beta_{0,y} \, \rF(\cdot, \alpha_0), \frac{1}{2} \beta_{1,y} \, \rF(\cdot, \alpha_1) \right\}, \]
for every value of $y$. In the eigenstrategy, we approximate the concave hull with the one of the two that has the larger constant. When we choose these constants according to a coarse mesh, e.g., $\nu =1/3$ or $\nu = 1/4$, the one with the larger constant is a very good approximation of the concave hull. It
appears that, we need finer precisions to bring out the power of this strategy in the filter.

In all our searches, we did not find any protocols with bias less than $0.2499$, and it seems that $1/4$ might be the least bias
achievable by the class of protocols we study. To further test this conjecture, in the next two subsections we present two other kinds
of search.

\subsection{Random offset}

We would like to test more protocols, and also avoid anomalies that may have arisen in the previous tests due to the structure
of the mesh we use and also any special relation
the protocol states may have with each other due to low precision.
The six-round searches take a long time, which restricts the
precision~$\nu$ we can use.
The resulting mesh is also highly structured. We would like to test protocol parameters that do not necessarily have such regular entries.
With this end in mind, we offset all of the values in the search by some random additive term $\delta > 0$. For example, say the entries of $\alpha_0$, $\alpha_1$, $\beta_0$, and $\beta_1$ have been selected from the set $\set{0, \nu, 2\nu, \ldots, 1-\nu, 1}$.
With an offset parameter $\delta \in (0, \nu/2)$, we use the range
\[ \set{\delta, \delta + \nu, \delta + 2\nu, \ldots, \delta + 1 - \nu}. \]
Note that this destroys index symmetry. The simplest way to see this is to consider the $2$-dimensional probability distributions created in this way. They are
\[ \set{ \twovec{\delta}{1-\delta}, \twovec{\delta + \nu}{1 - \delta - \nu}, \twovec{\delta + 2 \nu}{1 - \delta - 2 \nu}, \ldots, \twovec{\delta + 1 - \nu}{\nu - \delta} }. \]
We see that the set of first entries is not the same as the set of second entries when $\delta > 0$. We choose the last entry in each vector to be such that the entries add to $1$. Since we generate all four of the probability distributions in the same manner, we can still apply the symmetry arguments to suppose $\alpha_0$ has the largest entry out of both $\alpha_0$ and $\alpha_1$ and similarly for $\beta_0$ and $\beta_1$.

Table~\ref{table:worstcase} (above) shows how well each strategy in the filter performs in the worst case and Table~\ref{table:averagecase} (on the next page) shows the average case over $100$ random choices of offset parameter $\delta \in [0, 1/100]$.

\begin{table}
\caption{The percentage of protocols that get stopped by each strategy in the worst case over $100$ random instances of offset parameter $\delta$.}
\quad
\begin{center}
\begin{tabular}{|r||r|r|r|r|r|r|r|r|r|r|r|r|r|r|r|r|}
\hline
$d = 2$  &   $\nu = 1/3$  &   $\nu = 1/4$  &   $\nu = 1/5$   &   $\nu = 1/6$ \\
\hline
\hline
G1 &   $71.87 \% $ & $82.35 \% $ & $84.06 \% $ & $86.63 \% $ \\
\hline
G2 &   $17.18 \% $ & $29.80 \% $ & $15.80 \% $ & $24.15 \% $ \\
\hline
G3 &   $8.17 \% $ & $10.73 \% $ & $13.46 \% $ & $12.12 \% $  \\
\hline
G4 &  $51.45 \% $ & $49.68 \% $ & $53.99 \% $ & $48.44 \% $  \\
\hline
G5 &  $70.00 \% $ & $83.29 \% $ & $78.02 \% $ & $82.86 \% $  \\
\hline
G6 &   $   0 \% $ & $0 \% $ & $0 \% $ & $0 \% $  \\
\hline
G7 &  $75.00 \% $ & $92.43 \% $ & $87.32 \% $ & $94.35 \% $  \\
\hline
G8 &   $100 \% $ & $100 \% $ & $49.10 \% $ & $100 \% $  \\
\hline
G9 &     &   & $0 \% $ &    \\
\hline
G10 &     &   & $0 \% $ &     \\
\hline
SDPB0 &     &   & $100 \% $ &    \\
\hline
\end{tabular}
\label{table:worstcase}
\end{center}
\end{table}

\begin{table}
\caption{The percentage of protocols that get stopped by each strategy in the average case over $100$ random instances of offset parameter $\delta$.}
\quad
\begin{center}
\begin{tabular}{|r||r|r|r|r|r|r|r|r|r|r|r|r|r|r|r|r|}
\hline
$d = 2$ 	  &   $\nu = 1/3$   &  $\nu = 1/4$   &   $\nu = 1/5$   &   $\nu = 1/6$ \\
\hline
\hline
G1 &  $85.75 \% $ & $87.30 \% $ & $89.42 \% $ & $90.47 \% $  \\
\hline
G2 &   $17.18 \% $ & $29.80 \% $ & $15.80 \% $ & $24.15 \% $ \\
\hline
G3 &   $10.85 \% $ & $13.15 \% $ & $14.53 \% $ & $12.35 \% $ \\
\hline
G4 &   $62.49 \% $ & $52.53 \% $ & $55.34 \% $ & $53.03 \% $ \\
\hline
G5 &   $70.00 \% $ & $87.11 \% $ & $93.46 \% $ & $93.29 \% $ \\
\hline
G6 &   $   0 \% $ & $0 \% $ & $0 \% $ & $0 \% $ \\
\hline
G7 &   $98.70 \% $ & $99.01 \% $ & $96.58 \% $ & $98.77 \% $ \\
\hline
\end{tabular}
\label{table:averagecase}
\end{center}
\end{table}

\newpage
\paragraph{Observations on the random offset tests} \quad \\

We notice that G6 performs very poorly on these tests. We need finer precision to see the effects of G6 in the filter. Also, G1 performs generally better as the filter precision increases. We see from the previous tables that it should stay at roughly $90 \%$. We see that G5 and G7 perform very well. G7 sometimes filters out the rest (why the average case table only displays up to G7). G8 performs well most of the time, except in the $\nu = 1/5$ case in the worst case table. Few protocols made it past the entire filter, and only SDPB0 needed to be solved of the four SDPs.
No protocols with bias at most~$0.2499$ were found.


\subsection{Computer aided bounds on bias}
\label{sec-cproof}

The search algorithm has the potential to give us computer aided
\emph{proofs\/} that certain coin-flipping protocols have bias within a
small interval. In this section, we describe the kind of bound we can
deduce under the assumption that the software provides us
an independently verifiable upper bound on the additive
error in terms of the objective value.

We begin by showing that any state~$\xi \in \R^D$ of the form used in
the protocols is suitably close to a state given by the mesh used in the
search algorithm. For an integer~$N \ge 1$, let~$\bM_N = \set{j/N :
j \in \Z, 0 \le j \le N }$.

\begin{lemma}
\label{thm-sapprox}
Let~$N \ge 1$ be an integer.
Consider the state~$\xi = \sum_{i = 1}^D \sqrt{\gamma_i}\, e_i$ in~$\R^D$,
where~$\gamma \in \Prob^D$. Then there is a probability
distribution~$\gamma' \in \Prob^D \cap \bM_N^D$ such that the
corresponding state~$\xi' = \sum_{i = 1}^D \sqrt{\gamma'_i} \, e_i$
satisfies~$\xi^*\xi' \ge 1 - D/2N$.
\end{lemma}
\begin{proof}
Let~$\tgamma_i = \floor{\gamma_i N}/N$ for~$i \in \set{1,2,\dotsc,D}$.
Note that~$\sum_{i = 1}^D \tgamma_i \le 1$, and that~$1 - \sum_{i = 1}^D
\tgamma_i = \sum_{i = 1}^D \gamma_i - \sum_{i = 1}^D \tgamma_i = j/N$, for
some~$j \in \set{0, 1, 2, \dotsc, D}$. We may obtain~$\gamma'$ by
adding~$1/N$ to~$j$ coordinates of~$\tgamma$. For concreteness,
let~$\gamma'_i = \tgamma_i + 1/N$ for~$i \in \set{1, 2, \dotsc, j}$
and~$\gamma'_i = \tgamma_i$ for~$i \in \set{j+1, \dotsc, D}$. We
therefore have~$\norm{\gamma - \gamma'}_1 \le D/N$, and
\[
\xi^*\xi' \quad = \quad \rF(\gamma,\gamma')^{1/2}
    \quad \ge \quad 1 - \frac{D}{2N} \enspace,
\]
by Proposition~\ref{thm-fvdg}.
\qed
\end{proof}

The above lemma helps us show that any protocol in the family we
consider is approximated by one given by the mesh.
\begin{lemma}
\label{thm-papprox}
Consider a bit-commitment based coin-flipping protocol~$\calA$ with
bias~$\epsilon$ of the form defined in Section~\ref{family}. Let~$\calA$
be specified by
the~$4$-tuple~$(\alpha_0,\alpha_1,\beta_0,\beta_1)$, where~$\alpha_i,\beta_i
\in \Prob^D$. There is a protocol~$\calA'$ with bias~$\epsilon'$ of the same
form, defined by a~$4$-tuple~$(\alpha'_0,\alpha'_1,\beta'_0,\beta'_1)$,
where~$\alpha'_i,\beta'_i \in \Prob^D \cap \bM_N^D$,
such that~$\size{\epsilon - \epsilon'} \le 2\sqrt{D/N}$.
\end{lemma}
\begin{proof}
The statement of the lemma is vacuous if~$1 - D/2N < 0$, we therefore assume $1
- D/2N \ge 0$.
We show that~$\epsilon' \le \epsilon + 2\sqrt{D/N}$. The other
inequality~$\epsilon \le \epsilon' + 2\sqrt{D/N}$ follows similarly.

Without loss in generality, assume that bias~$\epsilon'$ is achieved
when Bob cheats towards~$0$ in protocol~$\calA'$.
Recall
\begin{eqnarray*}
\psi & = & \frac{1}{\sqrt{2}} \left( e_0 \otimes e_0 \otimes \psi_0
            + e_1 \otimes e_1 \otimes \psi_1 \right) \enspace,
            \qquad \textrm{and} \\
\Pi_{\rA,0} & = & \sum_{b \in \set{0,1}} e_b e_b^* \otimes e_b e_b^*
                   \otimes \phi_b \phi_b^* \enspace.
\end{eqnarray*}
Let the probability distributions~$\alpha'_0, \alpha'_1
, \beta'_0, \beta'_1$ and states~$\psi'_0, \psi'_1, \phi'_0, \phi'_1$
corresponding to the
distributions~$\alpha_0, \alpha_1, \beta_0, \beta_1$, respectively,
be the ones guaranteed by Lemma~\ref{thm-sapprox}.  Let
\begin{eqnarray*}
\psi' & = & \frac{1}{\sqrt{2}} \left( e_0 \otimes e_0 \otimes \psi'_0
            + e_1 \otimes e_1 \otimes \psi'_1 \right) \enspace,
            \qquad \textrm{and} \\
\Pi'_{\rA,0} & = & \sum_{b \in \set{0,1}} e_b e_b^* \otimes e_b e_b^*
                   \otimes \phi'_b (\phi'_b)^* \enspace.
\end{eqnarray*}
We have $\psi^* \psi' \geq 1 - \frac{D}{2N}$, by Lemma~\ref{thm-sapprox}, and 
\begin{eqnarray*}
\norm{\psi' (\psi')^* - \psi \psi^*}_*
    & \le & 2 \left( 1 - (\psi^* \psi')^2 \right)^{1/2} \\
    & \le & 2 \sqrt{D/N} \enspace,
\end{eqnarray*}
by Proposition~\ref{thm-fvdg}.
Further,
\begin{eqnarray*}
\norm{\Pi'_{\rA,0} - \Pi_{\rA,0}}_{{\textup{op}}}
    & \le & \max \set{ \norm{\phi'_0 (\phi'_0)^* - \phi_0 \phi_0^*}_{{\textup{op}}},
            \norm{\phi'_1 (\phi'_1)^* - \phi_1 \phi_1^*}_{{\textup{op}}} } \\
    & \le & \sqrt{D/N} \enspace,
\end{eqnarray*}
using the identity~$\norm{ v v^* - u u^* }_{{\textup{op}}}  = \left( 1 - (v^* u)^2
\right)^{1/2}$ for normalized real vectors $v$ and $u$. Here, $\norm{X}_{\textup{op}}$ denotes
the operator norm of~$X$, namely the largest singular value of the matrix $X$.

For this analysis, we assume that the protocol~$\calA'$ is presented
in the form described in Section~\ref{sect:CF}, and the two parties
start with joint initial state~$e_0^{\otimes 4n}$,
apply~$U_{1}, U_{2}, \dotsc, U_{2n}$ alternately, and finally
measure their parts of the system to obtain the output.

Consider Bob's cheating strategy towards~$0$ (which we assumed
achieves bias~$\epsilon'$). As in the proof of
Lemma~\ref{thm-bob-sdp}, it follows that there are spaces~$\calH_i$ and
corresponding unitary operations~$U'_{i}$ on them for even~$i \le 2n$
that characterize his cheating strategy.
When Alice measures~$\zeta' = (U'_{2n} U_{2n-1} U'_{2n-2} \dotsb U_1)
e_0^{\otimes 4n}$, she obtains outcome~$0$ with
probability~$\norm{\Pi'_{\rA,0} \zeta'}^2_{2} = \tfrac{1}{2} + \epsilon'$.
(In the expression for the final
state~$\zeta'$, we assume that the unitary operations extend to the combined
state space by tensoring with identity over the other part.)

We consider the same cheating strategy for Bob in the protocol~$\calA$,
in which Alice starts with the commitment state~$\psi$, and performs the
measurement~$\set{\Pi_{\rA,0}, \Pi_{\rA,1}, \Pi_{\rA,\abort}}.$ \comment{where
\begin{eqnarray*}
\psi & = & \frac{1}{\sqrt{2}} \left( e_0 \otimes e_0 \otimes \psi_0
            + e_1 \otimes e_1 \otimes \psi_1 \right) \enspace,
            \qquad \textrm{and} \\
\Pi_{\rA,0} & = & \sum_{b \in \set{0,1}} e_b e_b^* \otimes e_b e_b^*
                   \otimes \phi_b \phi_b^* \enspace.
\end{eqnarray*}
} 
This
corresponds to a different initial unitary transformation for
Alice instead of~$U_1$. Let~$\zeta$ be the corresponding final
joint state. Note that $\psi$ is mapped to $\zeta$ using the same unitary transformation that maps $\psi'$ to $\zeta'$ since Bob is using the same cheating strategy. The probability of outcome~$0$ is~$\norm{\Pi_{\rA,0}
\zeta}^2_{2} \le \tfrac{1}{2} + \epsilon$, as the protocol~$\calA$
has bias~$\epsilon$. We may bound the difference in probabilities
as follows.
\begin{eqnarray*}
\epsilon' - \epsilon
    & \le & \tr \left(\Pi'_{\rA,0} \zeta' (\zeta')^* \right)
            - \tr\left(\Pi_{\rA,0} \zeta \zeta^* \right) \\
    & = & \tr\left( (\Pi'_{\rA,0} - \Pi_{\rA,0}) \zeta' (\zeta')^* \right)
            + \tr\left(\Pi_{\rA,0} (\zeta' (\zeta')^* - \zeta \zeta^*)
            \right) \\
    & \le & \norm{\Pi'_{\rA,0} - \Pi_{\rA,0}}_{{\textup{op}}} +
            \frac{1}{2} \norm{\zeta \zeta^* - \zeta' (\zeta')^*}_* \qquad
            \textrm{By Eq.~(\ref{eqn-trbound})} \\
    & = & \norm{\Pi'_{\rA,0} - \Pi_{\rA,0}}_{{\textup{op}}} +
            \frac{1}{2} \norm{\psi \psi^* - \psi' (\psi')^*}_* \\
    & \le & 2\sqrt{D/N} \enspace,
\end{eqnarray*}
as claimed.
\qed
\end{proof}

We may infer bounds on classes of protocols using the search algorithm
and the lemma above.
Suppose the computational approximation to the bias obtained by the
algorithm has net additive error~$\tau$ due to the protocol filter and
SDP solver and the finite precision arithmetic used in the computations.
If the algorithm reports that there are no protocols with bias at
most~$\epsilon^*$ given by a mesh with precision parameter~$N$,
then it holds that there are no~$4$-tuples, even
outside the mesh, with bias at most~$\epsilon^* - 2\sqrt{D/N} - \tau$. Here~$D$
is the dimension of Alice's (or Bob's) first $n$ messages (i.e., commitment states used, or
equivalently, the size of the support of an element of the $4$-tuple).

A quick calculation with~$\epsilon^* = 0.2499$ shows that
mesh fineness parameter~$N \ge 2184 \times d$ for four-round protocols
and~$N \ge 2184 \times d^2$ for six-round protocols with
message dimension~$d$, would be
sufficient for us to conclude that such protocols \emph{do not
achieve\/} optimal bias~$\approx 0.2071$. We would then
obtain \emph{computer aided\/} lower bounds for new classes of
bit-commitment based protocols. Thus, a refinement of
the search algorithm that allows finer meshes for messages of larger
dimension and over more rounds would be well worth pursuing.


\subsection{New bounds for four-round qubit protocols}
\label{ssect:newbounds}

We can derive analytical bounds on the bias of four-round protocols using the strengthened Fuchs-van de Graaf inequality for qubit states, below:

\begin{proposition}[\cite{SR01}]
\label{thm-fvdg-qubit}
For any quantum states~$\rho_1,\rho_2 \in \Pos^2$, i.e., qubits, we
have
\[ 1 \quad \leq \quad \Delta (\rho_1, \rho_2) + {\rF(\rho_1, \rho_2)}
\enspace. \]
\end{proposition}

Recall from Section~\ref{sect:filter} that Bob can cheat in a four-round protocol with probability bounded below by
\begin{equation}
P_{B,0}^* \quad \geq \quad \half + \half \sqrt{\rF(\beta_0, \beta_1)} \label{ineq:Bob1}
\end{equation}
and
\begin{equation}
P_{B,0}^* \quad \geq \quad \half + \half \Delta(\alpha_0, \alpha_1) \label{ineq:Bob2}
\end{equation}
and Alice can cheat with probability bounded below by
\begin{equation}
P_{A,0}^* \quad \geq \quad \left( \half + \half \sqrt{\rF(\alpha_0, \alpha_1)} \right) \left( \half + \half \Delta(\beta_0, \beta_1) \right) \enspace. \label{ineq:alice}
\end{equation}
If $\beta_0, \beta_1 \in \prob^2$, then by~(\ref{ineq:Bob1}) and Proposition~\ref{thm-fvdg-qubit}, we have
$\Delta(\beta_0, \beta_1) \geq 4 P_{B,0}^* (1 - P_{B,0}^* )$
and if $\alpha_0, \alpha_1 \in \prob^2$, then from~(\ref{ineq:Bob2}) and Proposition~\ref{thm-fvdg-qubit}, we have
$\rF(\alpha_0, \alpha_1) \geq 2 - 2 P_{B,0}^*$. 
Combining these two bounds with~(\ref{ineq:alice}), we get
\[ 4 P_{A,0}^* \quad \geq \quad \left( 1+ \sqrt{2-2 P_{B,0}^*} \right) \left( 1 + 4 P_{B,0}^* (1 - P_{B,0}^*) \right) \]
implying $\max \{ P_{A,0}^*, P_{B,0}^* \} \geq 0.7487 > 1/\sqrt{2} \approx 0.7071$. In fact, using the Fuchs-van de Graaf inequalities from Proposition~\ref{thm-fvdg}, we can get bounds when they are not both two-dimensional. If $\beta_0, \beta_1$ are two-dimensional and $\alpha_0, \alpha_1$ are not, we get a lesser bound of $\max \{ P_{A,0}^*, P_{B,0}^* \} \geq 0.7140 > 1/\sqrt{2}$. On the other hand, if $\alpha_0, \alpha_1$ are two-dimensional and $\beta_0, \beta_1$ are not, then we get $\max \{ P_{A,0}^*, P_{B,0}^* \} \geq 0.7040 \not >  1/\sqrt{2}$, so we do not rule out the possibility of optimal protocols with these parameters.

Note that tests where $\alpha_0, \alpha_1$ are two-dimensional are subsumed in the higher-dimensional tests we performed. However, future experiments could include  computationally testing the case where Alice's first message is two-dimensional and Bob's first message has dimension 10 or greater. 


\subsection{Zoning-in on near-optimal protocols} \label{ssect:zoningin}

The computational tests that we performed so far suggest that there are no protocols with cheating probabilities less than $0.7499$, that is, slightly smaller than the best known constructions. The tests also show that the number of protocols grows very large as the mesh precision increases. This poses the question of whether there are protocols that have optimal cheating probabilities just slightly less than $3/4$ when one considers increased mesh precisions. In this subsection, we focus on searching for such protocols.

There are a few obstacles to deal with in such a search. The first is that increasing the precision of the mesh drastically increases the number of protocols to be tested. To deal with this, we restrict the set of parameters to be tested by only considering protocols which are close to optimal, i.e., near-optimal protocols. In other words, we ``zone in'' on some promising protocols to see if there is any hope of improving the bias by perturbing some of the entries. To do this, we fix a near-optimal protocol and create a mesh over a small ball around the entries in each probability vector. We would like a dramatic increase in precision, so we  use a ball of radius $2 \, \nu$ (unless stated otherwise), yielding up to $5$ increments tested around each entry. This gives us the advantage of having a constant number of protocols to check, independent of the mesh precision. However, this comes at the cost that we lose symmetry, since we do not wish to permute the entries nor the probability distributions defining the protocol.

Another challenge is to find the near-optimal protocols. The approach we
take is to keep track of the best protocol found, updating the filter threshold accordingly. There are two issues with this approach. One is that increasing the threshold decreases the efficiency of the filter, so we are not able to search over the same mesh precisions given earlier in this section. The
second is that there is an abundance of protocols with cheating probabilities exactly equal to $3/4$. As was done in Section~\ref{sect:algo}, we can embed an optimal three-round protocol with optimal cheating probabilities $3/4$ into a four-round or six-round protocol. One way to do this is to set $\alpha_0 = \alpha_1$ (i.e. Alice's first $n$ messages contain \emph{no} information) or by setting $\beta_0 \perp \beta_1$ (i.e. Bob's first message reveals $b$, making the rest of his messages meaningless). So we already know many protocols with cheating probabilities equal to $3/4$, but can we find others? We now discuss the structure of near-optimal protocols in the case of four-round and six-round protocols, and how we zone in on them.

\paragraph{Four-round version} \quad \\

For the four-round search, we fix a message dimension $d=5$ and use precision parameters $\nu \in \set{1/7, 1/8, 1/9, 1/10, 1/11}$. This search yields a minimum (computer
verified) bias of $\epsilon = 0.2647$ when we rule out protocols with $\alpha_0 = \alpha_1$ or $\beta_0 \perp \beta_1$. In other words, we have that all of the protocols tested had one of the following three properties:
\newpage 
\begin{itemize}
\item $\alpha_0 = \alpha_1$,
\item $\inner{\beta_0}{\beta_1} = 0$,
\item $\max \set{P_{\rA,0}^*, P_{\rA,1}^*, P_{\rB,0}^*, P_{\rB,1}^*} \geq 0.7647$.
\end{itemize}
This suggests that near-optimal four-round protocols behave similarly to optimal three-round protocols. We now zone in on two protocols, one representing each of the first two conditions above. The first protocol is
\[ \alpha_0 = \half \left[ 0, 0, 0, 1, 1 \right]^\transpose, \quad \alpha_1 = \half \left[ 0, 0, 1, 0, 1 \right]^\transpose, \quad \beta_0 = \left[ 0, 0, 0, 0, 1 \right]^\transpose, \quad  \beta_1 =\left[ 0, 0, 0, 1, 0 \right]^\transpose \]
which satisfies ${\beta_0} \perp {\beta_1} = 0$ and has all four (computationally
verified) cheating probabilities equal to $3/4$. The second protocol is
\[ \alpha_0 = \left[ 0, 0, 0, 0, 1 \right]^\transpose, \quad \alpha_1 = \left[ 0, 0, 0, 0, 1 \right]^\transpose, \quad \beta_0 = \half \left[ 0, 0, 0, 1, 1 \right]^\transpose, \quad  \beta_1 = \half \left[ 0, 0, 1, 0, 1 \right]^\transpose \]
which satisfies $\alpha_0 = \alpha_1$ and has all four (computationally
verified) cheating probabilities equal to $3/4$. Tables~\ref{4_first_zoning_in} and \ref{4_second_zoning_in} display the zoning-in searches for these two protocols with threshold exactly $3/4$. Note we use mesh precisions up to $10^{-16}$ which, by Lemma~\ref{thm-papprox}, can guarantee us a change in bias up to $4 \times 10^{-8}$. A (computationally verified) change in bias of this magnitude could be argued to be an actual decrease in bias and not an error due to finite precision arithmetic.

\begin{sidewaystable}
\caption{The number of four-round protocols that get past each strategy when zoning-in on the first near-optimal protocol (showing F1 and only the other strategies that helped to weed out protocols).}
\label{4_first_zoning_in}
\begin{center}
\begin{tabular}{|r||r|r|r|r|r|r|r|r|r|r|r|r|r|r|r|r|rlrlrlrlrlrlrlrlrlrlrlrlrlrl}
\hline
$d=5$ & $\nu=1/10^{10}$ & $\nu=1/10^{11}$ & $\nu=1/10^{12}$ & $\nu=1/10^{13}$ & $\nu=1/10^{14}$ & $\nu=1/10^{15}$ & $\nu=1/10^{16}$ \\
\hline
\hline
F1  & $  119,574,225  $ & $   119,574,225  $ & $   119,574,225   $ & $  119,574,225  $ & $   119,574,225   $ & $  119,574,225   $ & $  119,574,225 $ \\
\hline
F2  & $   20,253,807 $ & $     20,253,807   $ & $   20,411,271   $ & $   21,067,371   $ & $   20,253,807   $ & $   20,253,807
$ & $    6,337,926 $ \\
\hline
F3   & $    493,557    $ & $    493,557     $ & $   493,557    $ & $    581,503    $ & $    493,557   $ & $     498,504    $ & $     33,279 $ \\
\hline
F6    & $   493,557    $ & $    493,557    $ & $    493,557   $ & $     576,819     $ & $   493,557    $ & $    480,276    $ & $     13,695 $ \\
\hline
F7    & $      981      $ & $     981      $ & $     981    $ & $      1,245     $ & $      981        $ & $   855     $ & $        0 $ \\
\hline
F8    & $        0        $ & $     0        $ & $     0       $ & $      0        $ & $     0      $ & $      29      $ & $       0 $ \\
\hline
F10   & $         0     $ & $       0       $ & $      0    $ & $         0        $ & $     0      $ & $       0       $ & $      0 $ \\
\hline
\end{tabular}
\end{center}

\caption{The number of four-round protocols that get past each strategy when zoning-in on the second near-optimal protocol (showing F1 and only the other strategies that helped to weed out protocols).}
\label{4_second_zoning_in}
\begin{center}
\begin{tabular}{|r||r|r|r|r|r|r|r|r|r|r|r|r|r|r|r|r|rlrlrlrlrlrlrlrlrlrlrlrlrlrl}
\hline
$d=5$ & $\nu=1/10^{10}$ & $\nu=1/10^{11}$ & $\nu=1/10^{12}$ & $\nu=1/10^{13}$ & $\nu=1/10^{14}$ & $\nu=1/10^{15}$ & $\nu=1/10^{16}$ \\
\hline
\hline
F1 &   $ 9,277,254   $ & $  9,277,254    $ & $ 9,277,254 $ & $    9,277,254   $ & $  9,277,254  $ & $   8,516,178     $ & $ 4,953,555 $ \\
\hline
F3   &  $ 907,608  $ & $    907,608   $ & $   913,496   $ & $   912,864   $ & $   907,608   $ & $   828,952   $ & $  1,030,152$ \\
\hline
F6    & $ 693,576  $ & $    693,576  $ & $    695,016  $ & $    713,424   $ & $   693,576  $ & $    412,392    $ & $    8,624$ \\
\hline
F7     & $ 45,376   $ & $    45,376   $ & $    45,376  $ & $     55,064 $ & $      43,264   $ & $     3,136     $ & $   5,056$ \\
\hline
F8       &  $  0      $ & $     0     $ & $      0     $ & $      0      $ & $     0    $ & $      68   $ & $     3,140$ \\
\hline
F9      &  $   0      $ & $     0      $ & $     0     $ & $      0      $ & $     0     $ & $      8     $ & $   2,072$ \\
\hline
F10     & $      0     $ & $      0   $ & $        0      $ & $     0    $ & $       0     $ & $      0      $ & $     0$ \\
\hline
\end{tabular}
\end{center}
\end{sidewaystable}

\newpage
Note that not all filter strategies are useful in the zoning-in tests. For example, if the strategy F1 $\approx 1/2 < 3/4$ for the protocol we are zoning-in on, then it never filters out any protocols with the precisions considered. Considering this, and by examining the tables, we see that most strategies filter out many protocols, or none at all. Also from the tables, we see that no protocols get through the entire filter. Notice that we needed to use more strategies than were needed in previous tables, namely F9 and F10. In the previous searches, F8 was the last filter strategy needed, thus demonstrating some protocols which F8 fails to filter out (noting a larger threshold was used here than in the previous tests). It is worth noting the efficiency of the four-round filter. The
algorithm did not need to solve for any optimal cheating values  in any of the four-round zoning-in tests.

These tables suggest that perturbing the entries of the parameters defining these two near-optimal protocols does
not yield better bias.

\paragraph{Six-round version} \quad \\

For the six-round search, we fix a message dimension $d=2$ and use precision parameters $\nu \in \set{1/7, 1/8, 1/9, 1/10, 1/11, 1/12}$. For $\nu > 1/12$, the test results were similar to the four-round version, that all of the protocols tested had one of the following three properties:
\begin{itemize}
\item $\alpha_0 = \alpha_1$,
\item $\inner{\beta_0}{\beta_1} = 0$,
\item $\max \set{P_{\rA,0}^*, P_{\rA,1}^*, P_{\rB,0}^*, P_{\rB,1}^*} \geq 0.7521$.
\end{itemize}

We choose the following two near-optimal protocols to represent the first two conditions:
\[ \alpha_0 = \half \left[ 0, 0, 1, 1 \right]^\transpose, \quad \alpha_1 = \half \left[ 0, 1, 0, 1 \right]^\transpose, \quad \beta_0 = \left[ 0, 0, 0, 1 \right]^\transpose, \quad  \beta_1 =\left[ 0, 0, 1, 0 \right]^\transpose, \]
which satisfies ${\beta_0} \perp {\beta_1} = 0$, and
\[ \alpha_0 = \left[ 0, 0, 0, 1 \right]^\transpose, \quad \alpha_1 = \left[ 0, 0, 0, 1 \right]^\transpose, \quad \beta_0 = \half \left[ 0, 0, 1, 1 \right]^\transpose, \quad  \beta_1 = \half \left[ 0, 1, 0, 1 \right]^\transpose, \]
which satisfies $\alpha_0 = \alpha_1$. Both of these protocols have all four (computationally
verified) cheating probabilities equal to $3/4$.

However, when $\nu = 1/12$, we found several protocols with a (computationally found) bias of $0.25$. We therefore searched for all protocols with bias $0.2501$ or less. We discovered the following $4$ protocols, no two of which are equivalent to each other with respect to symmetry. Note that these protocols bear no resemblance to any bias $1/4$ protocols previously discovered. These protocols are below: 

\[ \alpha_0 = \frac{1}{3} \left[ 0, 1, 1, 1 \right]^\transpose, \quad \alpha_1 = \frac{1}{3} \left[ 1, 1, 0, 1 \right]^\transpose, \quad \beta_0 = \frac{1}{12} \left[ 0, 3, 0, 9 \right]^\transpose, \quad  \beta_1 = \frac{1}{12} \left[ 0, 3, 9, 0 \right]^\transpose \]
and
\[ \alpha_0 = \frac{1}{3} \left[ 0, 1, 1, 1 \right]^\transpose, \quad \alpha_1 = \frac{1}{3} \left[ 1, 1, 0, 1 \right]^\transpose, \quad \beta_0 = \frac{1}{12} \left[ 1, 2, 0, 9 \right]^\transpose, \quad  \beta_1 = \frac{1}{12} \left[ 1, 2, 9, 0 \right]^\transpose \]
and
\[ \alpha_0 = \frac{1}{3} \left[ 0, 1, 1, 1 \right]^\transpose, \quad \alpha_1 = \frac{1}{3} \left[ 1, 1, 1, 0 \right]^\transpose, \quad \beta_0 = \frac{1}{12} \left[ 0, 3, 0, 9 \right]^\transpose, \quad  \beta_1 = \frac{1}{12} \left[ 0, 3, 9, 0 \right]^\transpose \]
and
\[ \alpha_0 = \frac{1}{3} \left[ 0, 1, 1, 1 \right]^\transpose, \quad \alpha_1 = \frac{1}{3} \left[ 1, 1, 1, 0 \right]^\transpose, \quad \beta_0 = \frac{1}{12} \left[ 1, 2, 0, 9 \right]^\transpose, \quad  \beta_1 = \frac{1}{12} \left[ 1, 2, 9, 0 \right]^\transpose. \]

Note that these four protocols have the property that all the filter strategies for them have cheating probabilities strictly less than $3/4$. Since many of these strategies are derived
from optimal three-round strategies, this property makes them especially interesting. (Other
six-round protocols were found. However, these were
equivalent to the ones above, under the equivalence relation described
in Section~\ref{sect:symmetry}.)

We now zone in on these six protocols as indicated in the following tables. Note that we decrease the radius of the balls to $\nu$ for the third, fourth, fifth, and sixth protocol (compared to $2 \nu$ for the other protocols). This is for two reasons. One is that most the entries are bounded away from $0$ or $1$, making the intersection of the ball and valid probability vectors large. Second, the filter has to work harder in this case since many of the filter cheating probabilities are bounded away from $3/4$ and thus more computationally expensive cheating probabilities need to be computed.

Preliminary tests show that when zoning-in on some of these 6 protocols, the default SDP solver precision is not enough to determine whether the bias is strictly less than $3/4$, or whether it is numerical round-off. To provide a further test, we add an extra step for those protocols that get through the filter and SDPs, we increase the SDP solver accuracy (set pars.eps = $0$ in SeDuMi) and let the solver run until no more progress is being made. The row "Better Accuracy" shows how many protocols get through this added step. Furthermore, we use the maximum of the primal and dual values when calculating the optimal cheating values since we are not guaranteed exact feasibility of both primal and dual solutions in these computational experiments.

\begin{sidewaystable}
\caption{The number of six-round protocols that get past each strategy when zoning-in on the first near-optimal protocol (showing G1 and only the other strategies that helped to weed out protocols).}
\begin{center}
\begin{tabular}{|r||r|r|r|r|r|r|r|r|r|r|r|r|r|r|r|r|rlrlrlrlrlrlrlrlrlrlrlrlrlrl}
\hline
$d=2$ & $\eta=1/10^{10}$ & $\eta=1/10^{11}$ & $\eta=1/10^{12}$ & $\eta=1/10^{13}$ & $\eta=1/10^{14}$ & $\eta=1/10^{15}$ & $\eta=1/10^{16}$ \\
\hline
\hline
G1   &  $1,476,225$  &   $1,476,225$  &   $1,476,225$   &  $1,476,225$  &   $1,476,225$  &   $1,476,225$ &    $1,476,225$ \\
\hline
G2   &   $874,800$ &     $874,800$   &   $874,800$  &    $879,174$  &    $874,800$  &    $874,800$  &    $601,425$ \\
\hline
G3    &  $533,439$    &  $533,439$  &    $533,655$  &    $538,326$    &  $533,439$   &   $448,065$   &   $149,040$ \\
\hline
G5    &   $20,434$     &  $20,434$   &    $20,434$    &   $21,250$  &     $20,434$    &   $14,494$  &       $359$ \\
\hline
G7      &   $656$     &    $656$     &    $668$      &   $685$   &      $579$   &      $455$    &       $0$ \\
\hline
G8       &   $70$    &      $70$     &     $70$     &     $76$    &      $42$    &      $21$    &       $0$ \\
\hline
G9         &  $0$   &        $0$     &      $0$   &        $0$     &      $0$      &     $0$     &      $0$ \\
\hline
\end{tabular}
\label{table:ZI6P1}
\vspace{1cm}
\end{center}

\caption{The number of six-round protocols that get past each strategy when zoning-in on the second near-optimal protocol (showing G1 and only the other strategies that helped to weed out protocols).}
\begin{center}
\begin{tabular}{|r||r|r|r|r|r|r|r|r|r|r|r|r|r|r|r|r|rlrlrlrlrlrlrlrlrlrlrlrlrlrl}
\hline
$d=2$ & $\eta=1/10^{10}$ & $\eta=1/10^{11}$ & $\eta=1/10^{12}$ & $\eta=1/10^{13}$ & $\eta=1/10^{14}$ & $\eta=1/10^{15}$ & $\eta=1/10^{16}$ \\
\hline
\hline
G1   & $    93,312  $ & $     93,312    $ & $   93,312    $ & $   93,312    $ & $   93,312    $ & $   86,022    $ & $   40,824 $ \\
\hline
G4    & $    38,061    $ & $   38,061    $ & $   38,061   $ & $    38,061   $ & $    38,061   $ & $    28,125    $ & $    4,995 $ \\
\hline
G5    & $     2,664   $ & $     2,664    $ & $    2,664    $ & $    2,716    $ & $    2,664     $ & $   1,418    $ & $       0 $ \\
\hline
G6    & $     2,376   $ & $     2,376    $ & $    2,376   $ & $     2,420   $ & $     2,376     $ & $   1,174 $ & $          0 $ \\
\hline
G9    & $     1,270   $ & $        0     $ & $      0      $ & $     0      $ & $     0     $ & $      0      $ & $     0 $ \\
\hline
G10  & $        774   $ & $        0     $ & $      0     $ & $      0     $ & $      0     $ & $      0  $ & $         0 $ \\
\hline
SDPA0  & $       538   $ & $        0     $ & $      0     $ & $      0    $ & $       0     $ & $      0     $ & $      0 $ \\
\hline
SDPA1  & $       474   $ & $        0       $ & $    0    $ & $       0      $ & $     0    $ & $       0      $ & $     0 $ \\
\hline
Better Accuracy  & $          0   $ & $        0    $ & $       0     $ & $      0    $ & $       0     $ & $      0     $ & $      0 $ \\
\hline
\end{tabular}
\label{table:ZI6P2}
\end{center}
\end{sidewaystable}

\begin{sidewaystable}
\caption{The number of six-round protocols that get past each strategy when zoning-in on the third, fourth, fifth and sixth near-optimal protocols (showing G1 and only the other strategies that helped to weed out protocols).}
\begin{center}
\begin{tabular}{|r||r|r|r|r|r|r|r|r|r|r|r|r|r|r|r|r|rlrlrlrlrlrlrlrlrlrlrlrlrlrl}
\hline
$d=2$ & $\eta=1/10^{10}$ & $\eta=1/10^{11}$ & $\eta=1/10^{12}$ & $\eta=1/10^{13}$ & $\eta=1/10^{14}$ & $\eta=1/10^{15}$ & $\eta=1/10^{16}$ \\
\hline
\hline
G1  & $     34,992    $ & $     34,992    $ & $     34,992   $ & $      34,992   $ & $      34,992   $ & $      34,992   $ & $      34,992 $ \\
\hline
SDPB0    & $     9,720     $ & $     9,720       $ & $   9,720       $ & $   9,720       $ & $   9,720     $ & $     9,720    $ & $     27,215 $ \\
\hline
SDPA0      & $      0    $ & $       0    $ & $         0    $ & $         0    $ & $         0     $ & $        0      $ & $       0 $ \\
\hline
\end{tabular}
\label{table:ZI6P3}
\end{center}

\vspace{0.25cm}

\begin{center}
\begin{tabular}{|r||r|r|r|r|r|r|r|r|r|r|r|r|r|r|r|r|rlrlrlrlrlrlrlrlrlrlrlrlrlrl}
\hline
$d=2$ & $\eta=1/10^{10}$ & $\eta=1/10^{11}$ & $\eta=1/10^{12}$ & $\eta=1/10^{13}$ & $\eta=1/10^{14}$ & $\eta=1/10^{15}$ & $\eta=1/10^{16}$ \\
\hline
\hline
G1    & $    99,144   $ & $    99,144   $ & $      93,312     $ & $    99,144   $ & $      99,144    $ & $     99,144      $ & $   99,144 $ \\
\hline
SDPB0     & $       0     $ & $        0    $ & $         0      $ & $       0      $ & $       0      $ & $       0         $ & $    0 $ \\
\hline
\end{tabular}
\label{table:ZI6P4}
\end{center}

\vspace{0.25cm}

\begin{center}
\begin{tabular}{|r||r|r|r|r|r|r|r|r|r|r|r|r|r|r|r|r|rlrlrlrlrlrlrlrlrlrlrlrlrlrl}
\hline
$d=2$ & $\eta=1/10^{10}$ & $\eta=1/10^{11}$ & $\eta=1/10^{12}$ & $\eta=1/10^{13}$ & $\eta=1/10^{14}$ & $\eta=1/10^{15}$ & $\eta=1/10^{16}$ \\
\hline
\hline
G1   & $     34,992    $ & $     34,992  $ & $       34,992   $ & $      34,992    $ & $     34,992    $ & $     34,992      $ & $   34,992 $ \\
\hline
SDPB0    & $     9,720    $ & $      9,720   $ & $       9,720     $ & $     9,720   $ & $       9,720     $ & $     9,720    $ & $     27,215 $ \\
\hline
SDPA0    & $        0    $ & $         0     $ & $        0        $ & $     0    $ & $         0     $ & $        0      $ & $       0 $ \\
\hline
\end{tabular}
\label{table:ZI6P4}
\end{center}

\vspace{0.25cm}

\begin{center}
\begin{tabular}{|r||r|r|r|r|r|r|r|r|r|r|r|r|r|r|r|r|rlrlrlrlrlrlrlrlrlrlrlrlrlrl}
\hline
$d=2$ & $\eta=1/10^{10}$ & $\eta=1/10^{11}$ & $\eta=1/10^{12}$ & $\eta=1/10^{13}$ & $\eta=1/10^{14}$ & $\eta=1/10^{15}$ & $\eta=1/10^{16}$ \\
\hline
\hline
G1   & $     99,144     $ & $    99,144     $ & $    93,312     $ & $    99,144      $ & $   99,144     $ & $    99,144     $ & $    99,144 $ \\
\hline
SDPB0    & $        0     $ & $        0   $ & $          0         $ & $    0      $ & $       0      $ & $       0         $ & $    0 $ \\
\hline
\end{tabular}
\label{table:ZI6P4}
\end{center}

\end{sidewaystable}

\newpage 
We see in Tables~\ref{table:ZI6P1}, ~\ref{table:ZI6P2}, and~\ref{table:ZI6P3} that zoning-in on the six protocols yields no protocols with bias less than $1/4$. The zoning-in tests for the second near-optimal protocol are the only ones where we needed the added step of increasing the SDP solver accuracy. We see that this added step removed the remaining protocols.

We remark on the limitations of using such fine mesh precisions. For example, when zoning-in on the fourth and sixth protocol, only two strategies were used, G1 and SDPB0. These are both strategies for Bob which suggests that there are some numerical precision issues. We expect that some perturbations would decrease Bob's cheating probability, for example when $\alpha_0$ and $\alpha_1$ become ``closer'' and $\beta_0$ and $\beta_1$ remain the same. However, the precisions used in these searches do not find any such perturbations. 

From the outcome of the zoning-in tests, along with the computational evidence from all the other tests we conducted, we conjecture that any strong coin-flipping protocol based on
bit-commitment as defined formally in Section~\ref{family} has bias at
least~$1/4$ (Conjecture~\ref{conj-bias} in Section~\ref{sec-results}).


\section{Conclusions} \label{sect:conclusions}

We introduced a parameterized family of quantum coin-flipping protocols based
on bit-commitment, and formulated the cheating probabilities of Alice and Bob as simple semidefinite programs. Using these semidefinite programming formulations, we designed an algorithm to search for parameters yielding a protocol with small bias. We exploited symmetry and developed cheating strategies to create a protocol filter so that a wider array of protocols can be searched. For example, without the heuristics used in this paper, it would have taken over $69$ million years to search the same $3 \times 10^{16}$ protocols that
we tested.

Using the search algorithm, we searched four and six-round protocols
from a mesh over the parameter space, with messages
of varying dimension and with varying fineness for the mesh. After the initial systematic searches, no protocols having all four cheating probabilities less than $0.7499$ were found. We then performed a search
over a randomly translated mesh to avoid any anomalies that may have occurred while testing structured parameter sets. These tests also did not find any protocols with cheating probabilities less than $0.7499$. Our final tests zoned-in on protocols with maximum cheating probability $3/4$ to test whether there
are protocols with cheating probabilities  between $0.7499$ and $0.75$. A computational search to find such protocols yielded $8$
equivalence classes of protocols representing all the protocols with cheating probabilities equal to $3/4$. Four of these protocols bear no resemblance to previously known protocols with bias $1/4$. Zoning-in on these protocols showed that we cannot improve the bias by perturbing the parameters defining the protocols. Improvements to the algorithm may yield computer aided proofs of bounds on the bias of new sets of protocols. 

An obvious open problem is to resolve the conjecture that
all the protocols in the family we study have bias at least~$1/4$. It seems the smallest bias does not decrease when the number of messages increases from four rounds to six. We conjecture the smallest bias does not decrease even if more messages are added. One way to show this is to find closed-form expressions of the optimal objective values of the SDP formulations. This would be of
great theoretical
significance since very few highly interactive protocols (such as those examined in this paper) have
been characterized by closed-form expressions for their bias or even by a description of optimal cheating strategies.


\section*{Acknowlegdements}

We thank Andrew Childs, Michele Mosca, Peter H\o yer, and John Watrous for their comments and suggestions. A.N.'s research was supported in part by NSERC Canada, CIFAR, an ERA (Ontario),
QuantumWorks, and MITACS. A part of this work was completed
at Perimeter Institute for Theoretical Physics. Perimeter Institute
is supported in part by the Government of Canada
through Industry Canada and by the Province of Ontario through MRI.
J.S.'s research is supported by NSERC Canada, MITACS, ERA (Ontario), ANR project ANR-09-JCJC-0067-01, and ERC project
QCC 306537. L.T.'s research is supported in part by Discovery Grants from NSERC.


\nocite{ABD+04}
\nocite{Amb01}
\nocite{Amb02}
\nocite{ATVY00}
\nocite{BB84}
\nocite{Blu81}
\nocite{CK09}
\nocite{Moc04}
\nocite{Moc05}
\nocite{Moc07}
\nocite{NS03}
\nocite{SDP}
\nocite{SR01}
\nocite{Blu81}
\nocite{KN04}
\nocite{LC97}
\nocite{LC97a}
\nocite{May97}

\bibliographystyle{alpha}
\bibliography{paper}


\appendix

\section{SDP characterization of cheating strategies}
\label{sec-cheating-sdp}

In this section, we present proofs for Lemmas~\ref{thm-bob-sdp}
and~\ref{thm-alice-sdp}, originally due to Kitaev.

\paragraph{Proof of Lemma~\ref{thm-bob-sdp}:}
The matrix constraints in the SDP may readily be rewritten as linear
constraints on the variables~$\rho_j$, so the optimization problem is an
SDP. The variables are the density matrices of qubits under Alice's control after each of Bob's messages. The partial trace is trace-preserving, so any feasible solution satisfies
\[ \tr(\rho_F) = \tr(\rho_n) = \cdots = \tr(\rho_1) = \tr(\psi \psi^*) = 1. \]
Since $\rho_1, \ldots, \rho_n, \rho_F$ are constrained to be positive semidefinite, they are quantum states.

Bob sends the ${B_{1}}$ qubits to Alice replacing the ${A_{1}}$ part already sent to him. Being the density matrix Alice has after Bob's first message, $\rho_1$ satisfies
\[ \tr_{B_{1}}(\rho_{1}) = \tr_{A_{1}} ( \psi \psi^* ), \]
since the state of the qubits other than those in~$A_1, B_1$
remains unchanged.
Similarly, we have the constraint
\[ \tr_{B_{j}}(\rho_{j}) = \tr_{A_{j}} (\rho_{j-1}), \quad \text{ for } \quad j \in \{ 2, \ldots, n \}, \]
for each $\rho_j$ after Bob's $j$'th message. Also $\rho_F$, the state Alice has at the end of the protocol, satisfies
\[ \tr_{B' \times B'_{0}}(\rho_{F}) = \tr_{A' \times A'_{0}}(\rho_{n}). \]
She then measures $\rho_{F}$ and accepts $c$ with probability $\inner{\rho_{F}}{\Pi_{{\rA,c}}}$.

These constraints are necessary conditions on the states under Alice's control.
We may further restrict the states to be real matrices: the real parts
of any complex feasible solution also form a feasible solution with the
same objective function value.

We now show that every feasible solution to the above problem yields a valid cheating strategy for Bob with success probability equal to the objective function value of the feasible solution. He can find such a strategy by maintaining a purification of each density matrix in the feasible solution. For example, suppose the protocol starts in the state $\psi \otimes \phi'$, where $\phi' \in \C^K := \C^{B_0} \otimes \C^{B'_0} \otimes \C^B \otimes \C^{B'} \otimes \C^{K'}$ where $\C^{K'}$ is extra space Bob uses to cheat. Consider $\tau \in \C^{A_0} \otimes \C^{A'_0} \otimes \C^A \otimes \C^{A'} \otimes \C^K$ a purification of $\rho_1$ and $\eta := \psi \otimes \phi'$ a purification of $\psi \psi^*$. Since $\tr_{B_1}(\rho_1) = \tr_{A_1} (\psi \psi^*)$,
\[ \tr_{A_1 \times K} (\tau \tau^*)
= \tr_{B_1}(\rho_1)
= \tr_{A_1} (\psi \psi^*)
= \tr_{A_1 \times K} (\eta \eta^*). \]
Thus, there exists a unitary $U$ which acts on $\C^{A_1} \otimes \C^{K}$ which maps $\tau$ to $\eta$. If Bob applies this unitary after Alice's first message and sends
the ${B_1}$ qubits back then he creates $\rho_1$ under Alice's control. The same argument can be applied to the remaining constraints.

The states corresponding to honest Bob yield a feasible solution.
Attainment of an optimal solution then follows from continuity of the objective function and from the compactness
of the feasible region.
An optimal solution yields an optimal cheating strategy.  \qed

The characterization of Alice's cheating strategies is almost the same as that for cheating Bob; we
only sketch the parts that are different.

\paragraph{Proof of Lemma~\ref{thm-alice-sdp}:}
 There are two key differences from the proof of
Lemma~\ref{thm-bob-sdp}.  One is that Alice sends the first message and Bob sends the last, explaining the slightly different constraints. Secondly, Bob measures only the $\C^{B_{0}} \otimes \C^{A'_{0}} \otimes \C^A \otimes \C^{A'}$ part of his state after Alice's last message, i.e., he measures $\tr_{B'_0 \times B'} (\sigma_F)$. Note that the adjoint of the partial trace can be written as $\tr_{B'_0
\times B'}^*(Y) = Y \otimes \id_{B'_0 \times B'}$. Therefore we
have
\[ \inner{\tr_{B'_0 \times B'} (\sigma_F)}{\Pi_{\rB,c}} = \inner{\sigma_F}{\Pi_{\rB,c} \otimes \id_{B'_0 \times B'}}
\enspace, \]
which explains the objective function. \qed


\section{Derivations of the reduced cheating strategies} \label{Alice}

In this appendix, we show the derivation of Alice's reduced cheating strategy (the derivation of Bob's is very similar and the arguments are the same). We show that if we are given an optimal solution to Alice's cheating SDP, then we can assume it has a special form while retaining the same objective function value. Then we show this special form for an optimal solution can be written in the way desired.

\paragraph{Technical lemmas} \quad \\

We now discuss some of the tools used in the proofs in the rest of the appendix.

\begin{lemma} \label{diag}
Suppose $A$ is a finite set. Suppose  $p = \sum_{x \in A} {p_x} \, e_x \otimes e_x \in \Prob^{A \times A}$ and $\sigma \in \Herm_{+}^A$ is a density matrix. Then we have
\[ \max_{\rho \in \Herm_{+}^{A \times A}} \set{ \inner{\sqrtt{p}}{\rho} : \tr_{A}(\rho) = \sigma} \\
\leq
\max_{\rho \in
\Herm_{+}^{A \times A}} \set{ \inner{\sqrtt{p}}{\rho} : \tr_{A}(\rho) = \Diag(\sigma)},
\]
where $\Diag$ restricts to the diagonal of a square matrix. Moreover, an optimal solution to the problem on the right is $\overline{\rho} := \sqrtt{q}$, where $q = \sum_{x \in A} [\sigma]_{x,x} \, e_x \otimes e_x \in \Prob^{A \times A}$, yielding an objective function value of $\rF(p,q)$.
\end{lemma}

\begin{proof} Consider~$\bar{\rho}$ as defined in the statement of the lemma.
Since~$\tr_A(\bar{\rho}) = \Diag(\sigma)$, it suffices to show that
for any density matrix~$\rho \in \Herm^{A \times A}_+$ satisfying either $\tr_A(\rho)
= \sigma$ or~$\tr_A(\rho) = \Diag(\sigma)$, we have~$\inner{\sqrtt{p}}{\rho}
\le \inner{\sqrtt{p}}{\bar{\rho}} = \rF(p,q)$.

Expanding the first inner product, and using the Cauchy-Schwartz
inequality, we get
\[ \inner{\sqrtt{p}}{\rho} \! =  \!\! \sum_{x,y \in A} \sqrt{p_x p_y} (e_x \tensor e_x)^\transpose
          \rho \, (e_y \tensor e_y) \! \le \!\! \sum_{x,y \in A} \sqrt{p_x p_y} \norm{ \sqrt{\rho} \, (e_x
            \tensor e_x)} \cdot \norm{\sqrt{\rho} \, (e_y \tensor e_y)}. \]
\comment{
\begin{eqnarray*}
\inner{\sqrtt{p}}{\rho}
    & = & \sum_{x,y \in A} \sqrt{p_x p_y} (e_x \tensor e_x)^\transpose
          \rho \, (e_y \tensor e_y) \\
    & \le & \sum_{x,y \in A} \sqrt{p_x p_y} \norm{ \sqrt{\rho} \, (e_x
            \tensor e_x)} \cdot \norm{\sqrt{\rho} \, (e_y \tensor e_y)}
            \enspace.
\end{eqnarray*}
} 
We can simplify this by noting
\[ \norm{ \sqrt{\rho} \, (e_x \tensor e_x)}^2 \! = \! (e_x \tensor
e_x)^\transpose \rho \, (e_x \tensor e_x) \le \sum_{z \in A}
(e_z \tensor e_x)^\transpose \rho \, (e_z \tensor e_x)
= e_x^\transpose \tr_{A}(\rho) e_x
= [\sigma]_{x,x} \]
implying
$\inner{\sqrtt{p}}{\rho}
\leq \sum_{x,y \in A} \sqrt{p_x p_y} \left( [\sigma]_{x,x} [\sigma]_{y,y} \right)^{\frac{1}{2}}
=
\left( \sum_{x \in A} \sqrt{p_x [\sigma]_{x,x}} \right)^2
= \rF(p,q)$,
as desired. \qed
\end{proof}

\comment{
\begin{proof}
We first show that $\inf_{W \in \Herm^A} \set{\inner{\Diag(\sigma)}{W}: W \otimes \id_{A} \succeq \sqrtt{p}}$, the dual of the problem on the right, has a sequence of feasible solutions all of which are diagonal and the objective function values converge to the optimal value. Consider the following primal-dual pair
\[ \begin{array}{rrrclcrrrclrrrrrrrrrrrr}
& \sup                         & \inner{\sqrtt{p}}{\rho}  &    &    & & & \inf                        & \inner{\Diag(\sigma)}{W}               \\
                     & \textrm{subject to}  &               \tr_{A}(\rho)  & = &  \Diag(\sigma), & &                     & \textrm{subject to} & W \otimes \id_{A} & \succeq & \sqrtt{p},  \\
                     &                                &             \rho    & \in & \Pos^{A \times A}, & & & & W & \in & \Herm^A. \\
\end{array} \]
We know that $\overline{\rho}$ is a feasible solution for the primal problem with objective value $\rF(p,q)$. We now define a sequence of dual feasible solutions with objective value approaching this bound. Define positive definite $W = \sum_{x \in A} w_x \, e_x e_x^*$ where the values $w_x > 0$ are determined later.
Then
\begin{eqnarray*}
W \otimes \id_{A} \succeq \sqrtt{p}
& \iff & \id \succeq (W \otimes \id_{A})^{-1/2} \sqrtt{p} (W \otimes \id_{A})^{-1/2} \\
& \iff & 1 \geq \sqrt{p}^\transpose (W \otimes \id_{A})^{-1} \sqrt{p} \\
& \iff & 1 \geq \sum_{x \in A} \frac{p_x}{w_x}.
\end{eqnarray*}
For $\delta > 0$, define
\[ w_x := \left\{ \begin{array}{rcl}
(\sqrt{\rF(p,q)} + \delta) \dfrac{\sqrt{p_x}}{\sqrt{q_x}} & \text{ if } & p_x, q_x > 0, \\
\dfrac{\delta}{\sqrt{\rF(p,q)} + \delta} & \text{ if } & q_x = 0, \\
\epsilon & \text{ if } & p_x = 0, q_x > 0.
\end{array} \right. \]
Now,
\[
\sum_{x \in A} \frac{p_x}{w_x}
= \sum_{x \in A: \, p_x, q_x > 0} p_x \dfrac{\sqrt{q_x}}{\sqrt{p_x} (\sqrt{\rF(p,q)} + \delta)} +
\sum_{x \in A: \, p_x > 0, q_x = 0} p_x \dfrac{\delta}{\sqrt{\rF(p,q)} + \delta} \leq 1,
\]
and
\begin{eqnarray*}
\inner{\Diag(\sigma)}{W}
& = & \sum_{x \in A} q_x w_x \\
& = & \sum_{x \in A: \, q_x, p_x > 0} q_x (\sqrt{\rF(p,q)} + \delta) \dfrac{\sqrt{p_x}}{\sqrt{q_x}} \; +
\sum_{x \in A: \, q_x > 0, p_x = 0} q_x \, \delta \\
& \to & \rF(p,q) \text{ as } \delta \to 0.
\end{eqnarray*}
Therefore, we have a sequence of diagonal dual feasible solutions approaching optimality.

Using these diagonal solutions, we have the following by applying strong duality twice
\begin{eqnarray*}
& & \max_{\rho \in \Pos^{A \times A}} \set{\inner{\sqrtt{p}}{\rho}: \tr_{A}(\rho) = \sigma} \\
& = & \inf_{W \in \Herm^A} \set{\inner{\sigma}{W}:  W \otimes \id_{A} \succeq \sqrtt{p}} \\
& \leq & \inf_{W \in \Herm^A} \set{\inner{\sigma}{W}: W \otimes \id_{A} \succeq \sqrtt{p}, W = \Diag(W)} \\
& = & \inf_{W \in \Herm^A} \set{\inner{\Diag(\sigma)}{W}: W \otimes \id_{A} \succeq \sqrtt{p}, W = \Diag(W)} \\
& = & \inf_{W \in \Herm^A} \set{\inner{\Diag(\sigma)}{W}: W \otimes \id_{A} \succeq \sqrtt{p}} \\
& = & \max_{\rho \in \Pos^{A \times A}} \set{\inner{\sqrtt{p}}{\rho}: \tr_{A}(\rho) = \Diag(\sigma)},
\end{eqnarray*}
as desired. \qed
\end{proof}
} 

\begin{definition}
We define the \emph{partial $\Diag$ operator} over the subspace $\C^A$,
denoted $\Diag_{A}$, as the operator that projects density matrices over~$\C^B
\tensor \C^A$ onto the diagonal only on the subspace $\C^A$:
\[
\Diag_A (\rho) \quad = \quad
    \sum_{x \in A} (\id_B \tensor e_x^{\transpose}) \, \rho \, (\id_B \tensor e_x) \otimes e_x e_x^{\transpose}.
\]

\comment{
for finite sets $A$ and $B$, and matrix $M = \sum_{x,x',y,y'} [M]_{x,x',y,y'} \, e_x e_{x'}^\transpose \otimes e_y e_{y'}^\transpose \in \Herm^{A \times B}$ we apply the partial $\Diag$ operator over the subspace $\C^A$ to get
\[ \Diag_{A}(M) := \sum_{x,y,y'} [M]_{x,x,y,y'} \, e_x e_x^\transpose \otimes e_y e_{y'}^\transpose \in \Herm^{A \times B}. \]
} 
\end{definition}
We may write $\Diag_A$  as the superoperator $\mathbb{I} \otimes \Diag_A$, where $\mathbb{I}$ is the identity superoperator acting on the rest of the space.
Similarly, we may write the partial trace over $A$ as the superoperator $\tr_A := \mathbb{I} \otimes \tr(\cdot)$ where $\tr(\cdot)$ acts only on $\C^A$. Using this perspective, we see that the partial trace and the partial $\Diag$ operators commute when they act on different subspaces. Also, $\tr_A \circ \Diag_A = \tr_A$ since the trace only depends on the diagonal elements.

We also make use of the following lemma.
\begin{lemma} \label{rank1}
Consider a matrix $\rho \in \Pos^{A \times B}$. If $\tr_{A}(\rho) = \psi \psi^*$ for some vector $\psi \in \C^B$, then $\rho$ can be written as $\rho = \tilde{\rho} \otimes \psi \psi^*$, for some $\tilde{\rho} \in \Pos^A$.
\end{lemma}

This is easily proven using the fact that the half-line emanating through a rank one positive semidefinite matrix forms an extreme ray of the cone of positive semidefinite matrices,
or more directly by expressing~$\rho$ using an orthogonal basis for~$\C^B$ that
includes~$\psi$.


\paragraph{Derivations of Alice's reduced cheating strategy} \quad \\ 

Assume $(\sigma_1, \sigma_2, \ldots, \sigma_n, \sigma_F)$ is optimal for Alice's cheating SDP. We now define
\[ (\sigma'_1, \sigma'_2, \ldots, \sigma'_n, \sigma'_F)
= (\sigma_1, \Diag_{B'_1} (\sigma_2), \ldots, \Diag_{B'_1 \times \cdots \times B'_{n-1}} (\sigma_n), \Diag_{B' \times A'_0}(\sigma_F)) \]
and show it is also optimal. All we need to show is feasibility since the objective function value is preserved because $\Pi_{\rB,c} \otimes \id_{B'_0 \times B'}$ is diagonal in the space $\pos^{B' \times A'_0}$.

The first constraint is satisfied since $\sigma'_1 = \sigma_1$ is part of a feasible solution. From Lemma~\ref{rank1}, we can write $\sigma'_1 = \phi \phi^* \otimes \tilde{\sigma}_1$ for some $\tilde{\sigma}_1 \in \Pos^{A_1}$. We can write
\[ \tr_{B_1}(\sigma'_1) = \sum_{y_1 \in B'_1} e_{y_1} e_{y_1}^* \otimes \phi_{y_1} \phi_{y_1}^* \otimes \tilde{\sigma}_1, \]
where $\phi_{y_1, \ldots, y_j} := \dsum_{b \in B_0} \sum_{y_{j+1} \in B'_{j+1}} \cdots \dsum_{y_n \in B'_n} \dfrac{1}{\sqrt 2} \sqrt{\beta_{b,y}} \, e_{b} \otimes e_{b} \otimes e_{y_{j+1}} \otimes e_{y_{j+1}} \otimes \cdots \otimes e_{y_n} \otimes e_{y_n}$,
which is in $\C^{B_0 \times B'_0 \times B_{j+1} \times B'_{j+1} \times \cdots \times B_n \times B'_{n}}$.
Therefore, $\tr_{B_1}(\sigma'_1)$ is diagonal in $B'_1$ and
\begin{equation} \tr_{B_1}(\sigma'_1)
= \Diag_{B'_1} (\tr_{B_1}(\sigma'_1))
= \Diag_{B'_1} (\tr_{B_1}(\sigma_1))
= \Diag_{B'_1} (\tr_{A_2}(\sigma_2))
= \tr_{A_2}(\sigma'_2). \label{eqn:reduction}
\end{equation}
Therefore, the second constraint is satisfied. Since $\sigma'_2$ is diagonal in $B'_1$ we can write it as
\[ \sigma'_2 = \sum_{y_1 \in B'_1} e_{y_1} e_{y_1}^* \otimes {\sigma}_{2,y_1}, \quad \textup{ for some } \quad \sigma_{2,y_1} \in \Pos^{B_0 \times B'_0 \times A_1 \times A_2 \times B_2 \times \cdots \times B_n \times B'_2 \times \cdots \times B'_n}. \]
By feasibility,
${\tr_{A_2} (\sigma'_2) \! = \!\! \dsum_{y_1 \in B'_1} e_{y_1} e_{y_1}^* \otimes \tr_{A_2} ({\sigma}_{2,y_1})
\! = \! \tr_{B_1}(\sigma'_1)
\! = \!\! \dsum_{y_1 \in B'_1} e_{y_1} e_{y_1}^* \otimes \phi_{y_1} \phi_{y_1}^* \otimes \tilde{\sigma}_1}$,
therefore $\sigma'_2 = \sum_{y_1 \in B'_1} e_{y_1} e_{y_1}^* \otimes \phi_{y_1} \phi_{y_1}^* \otimes \tilde{\sigma}_{2,y_1}$,
where $\tilde{\sigma}_{2,y_1} \in \Pos^{B_0 \times B'_0 \times A_1 \times A_2}$ satisfies $\tr_{A_2}(\tilde{\sigma}_{2,y_1}) = \tilde{\sigma}_{1}$ for all $y_1 \in B'_1$. Using
similar arguments, we may show that the rest of the first $n$ constraints are satisfied. For every $j \in \set{3, \ldots, n}$, we have
\[ \sigma'_j = \sum_{y_1 \in B'_1} \cdots \sum_{y_{j-1} \in B'_{j-1}} e_{y_1} e_{y_1}^* \otimes \cdots \otimes e_{y_{j-1}} e_{y_{j-1}}^* \otimes \phi_{y_1, \ldots, y_{j-1}} \phi_{y_1, \ldots, y_{j-1}}^* \otimes \tilde{\sigma}_{j, y_1, \ldots, y_{j-1}}, \]
where $\tilde{\sigma}_{j, y_1, \ldots, y_{j-1}} \in \Pos^{B_0 \times B'_0 \times A_1 \times \cdots \times A_{j}}$ satisfies ${\tr_{A_j}(\tilde{\sigma}_{j, y_1, \ldots, y_{j-1}}) = \tilde{\sigma}_{j-1, y_1, \ldots, y_{j-2}}}$ for all $y_1 \in B'_1, \ldots, y_{j-1} \in B'_{n-1}$.
Note that ${\tr_{B_n}(\sigma'_n) = \sum_{y \in B'} e_{y} e_{y}^* \otimes \phi_y \phi_y^* \otimes \tilde{\sigma}_{n, y_1, \ldots, y_{n-1}}}$
which is helpful in proving feasibility of the last constraint. For the
last constraint, we can use a similar reduction as in Equation~(\ref{eqn:reduction}) to show $\tr_{A' \times A'_0}(\sigma'_F) = \tr_{B_n}(\sigma'_n)$ proving $(\sigma'_1, \ldots, \sigma'_n, \sigma'_F)$ is feasible. We now use this feasible solution to simplify the problem.

\comment{
\begin{eqnarray*}
\tr_{A' \times A'_0}(\sigma'_F)
& = & \tr_{A' \times A'_0} (\Diag_{B' \times A'_0}(\sigma_F)) \\
& = & \Diag_{B'} (\tr_{A' \otimes A'_0} (\sigma_F)) \\
& = & \Diag_{B'} (\tr_{B_n}(\sigma_n)) \\
& = & \Diag_{B'_n} (\tr_{B_n}(\sigma'_n)) \\
& = & \tr_{B_n}(\sigma'_n) \enspace.
\end{eqnarray*}
} 

We can clean up $\sigma'_F$ by noting that it is diagonal in $\C^{B'}$ and $\C^{A'_0}$ and write it as
\[ \sigma'_F = \sum_{a \in A'_0} \sum_{y \in B'} e_a e_a^* \otimes e_y e_y^* \otimes \sigma_{F,a,y}, \quad \text{ for some } \quad \sigma_{F,a,y} \in \Pos^{B_0 \times B'_0 \times A \times A'}. \]
Thus, $\tr_{A' \times A'_0} (\sigma'_F) = \sum_{a \in A'_0} \sum_{y \in B'} e_y e_y^* \otimes \tr_{A'}(\sigma_{F,a,y}) = \sum_{y \in B'} e_y e_y^* \otimes \left( \sum_{a \in A'_0} \tr_{A'}(\sigma_{F,a,y}) \right)$.
Similarly, by feasibility, we have $\tr_{A' \times A'_0} (\sigma'_F) = \tr_{B_n}(\sigma'_n) = \sum_{y \in B'} e_y e_y^* \otimes \phi_y \phi_y^* \otimes \sigma_{n, y_1, \ldots, y_{n-1}}$.
Thus,
\[ \sigma'_F = \sum_{a \in A'_0} \sum_{y \in B'} \kb{a} \otimes \kb{y} \otimes \phi_y \phi_y^* \otimes \tilde{\sigma}_{F,a,y}, \]
by writing $\sigma_{F,a,y} = \phi_y \phi_y^* \otimes \tilde{\sigma}_{F,a,y}$ where $\tilde{\sigma}_{F, a, y} \in \Pos^{A \times A'}$ satisfies $\dsum_{a \in A'_0} \tr_{A'}(\tilde{\sigma}_{F,a,y}) = \sigma_{n, y_1, \ldots, y_{n-1}}$ for all $a \in A'_0$ and $y \in B'$.

The objective function becomes $\inner{\sigma'_F}{\Pi_{\rB,0} \otimes \id_{B'_0 \times B'}} = \frac{1}{2} \sum_{a \in A'_0} \sum_{y \in B'} \beta_{a,y} \inner{\tsigma_{F,a,y}}{\psi_a \psi_a^*}$. At this point, we note that $\inner{\sigma'_F}{\Pi_{\rB,1} \otimes \id_{B'_0 \times B'}} =
\half \sum_{a \in A'_0} \sum_{y \in B'} \beta_{\bar{a},y} \inner{\tsigma_{F,{a},y}}{\psi_{a} \psi_{a}^*}$,
proving that evaluating Alice's success probability of cheating towards $0$ or $1$ with this strategy is a matter of switching Bob's two probability distributions.

\comment{
\begin{eqnarray*}
& & \inner{\sigma'_F}{\Pi_{\rB,0} \otimes \id_{B'_0 \times B'}} \\
& = &
\inner{\tr_{B'_0 \otimes B'}(\sigma'_F)}{\Pi_{\rB,0}} \\
& = & \inner{\tr_{B'_0 \times B'} \left( \sum_{a \in A'_0} \sum_{y \in B'} \kb{a} \otimes \kb{y} \otimes \phi_y \phi_y^* \otimes \tsigma_{F,a,y} \right)}{\Pi_{\rB,0}} \\
& = & \inner{\sum_{a \in A'_0} \sum_{y \in B'} \kb{a} \otimes \left( \sum_{b \in B_0} \frac{1}{2} \beta_{b,y} \, \kb{b} \right) \otimes \tsigma_{F,a,y}}{\Pi_{\rB,0}} \\
& = & \half \sum_{y \in B'} \sum_{b \in B_0} \sum_{a \in A'_0} \sum_{a' \in A'_0} \beta_{b,y} \inner{ \kb{a} \otimes \kb{b} \otimes \tsigma_{F,a,y}}{\kb{a'} \otimes \kb{a'} \otimes \psi_{a'} \psi_{a'}^*} \\
& = & \half \sum_{a \in A'_0} \sum_{y \in B'} \beta_{a,y} \inner{\tsigma_{F,a,y}}{\psi_a \psi_a^*}. \\
\end{eqnarray*}
} 

Carrying on with $P_{\rA,0}^*$, we get the following SDP
\[ \begin{array}{rrrcllllllllllllll}
& \sup                         & \half \displaystyle\sum_{a \in A'_0, \, y \in B'} \beta_{a,y} \, \inner{\tsigma_{F,a,y}}{\psi_a \psi_a^*} \\
                     & \textrm{subject to } & \tr_{A_1} (\tsigma_1) & = & 1, \\                                                                                             &  & \tr_{A_j} (\tsigma_{j,y_1, \ldots, y_{j-1}}) & = & \tsigma_{j-1,y_1, \ldots, y_{j-2}}, & \forall j \in \set{2, \ldots, n}, \\
& & & & & \forall y_{1} \in B'_{1}, \\
& & & & & \quad \quad \vdots \\
& & & & & \forall y_{j-1} \in B'_{j-1}, \\
                     &                                   & \sum_{a \in A'_0} \tr_{A'}(\tsigma_{F,a,y}) & = & \tsigma_{n, y_1, \ldots, y_{n-1}}, & \forall y \in B', \\
                     & & \tsigma_{j, y_{1}, \ldots, y_{j-1}} & \in & \Pos^{A_1 \times \cdots \times A_{j}}, & \forall j \in \{ 1, \ldots, n \}, \\
& & & & & \forall y_{1} \in B'_{1}, \\
& & & & & \quad \quad \vdots \\
& & & & & \forall y_{j-1} \in B'_{j-1}, \\
& & \tsigma_{F,a,y} & \in & \pos^{A' \times A}, & \forall a \in A'_0, y \in B'. \\
\end{array} \]
By Lemma~\ref{diag}, the following restrictions can only improve the objective function value:
\begin{eqnarray*}
s_1 := \diag(\tsigma_1), & \quad & \\
s_2^{(y_1)} := \diag(\tsigma_{2,y_1}), & \quad & \forall y_1 \in B'_1, \\
\vdots \quad \quad & & \\
s_n^{(y_1, \ldots, y_{n-1})} := \diag(\tsigma_{n,y_1, \ldots, y_{n-1}}), & \quad & \forall y_1 \in B'_1, \ldots, y_{n-1} \in B'_{n-1}, \\
s^{(a,y)} := \diag(\tr_{A'}(\tsigma_{F,a,y})), & \quad & \forall a \in A'_0, \, y \in B', \\
\tr_{A'}(\tsigma_{F,a,y}) = \Diag({s^{(a,y)}}), & \quad & \forall a \in A'_0, \, y \in B',
\end{eqnarray*}
where the superscripts are the restrictions of the vectors as before. With these new variables, and using Lemma~\ref{diag}, we can write the new objective function as
\[ \half \sum_{a \bit} \sum_{y \in B} \beta_{a,y} \, \rF( s^{(a,y)}, \alpha_a), \]
where $(s_1, \ldots, s_n, s) \in \calP_\rA$. Any feasible solution to
the reduced SDP also gives us a feasible solution to the original SDP,
so their optimal values are equal. \qed


\comment{
\section{Derivation of Bob's reduced cheating strategy}\label{Bob}

The proof for Bob's reduced cheating problem follows in the same manner as cheating Alice. We show that we can assume an optimal solution to have a particular form, then exploit its structure.

Let $(\rho_1, \rho_2, \ldots, \rho_n, \rho_F)$ be an optimal solution for Bob's cheating SDP. Define
\[ (\rho'_1, \rho'_2, \ldots, \rho'_n, \rho'_F) = (\Diag_{A'_1}(\rho_1), \Diag_{A'_1 \times A'_2}(\rho_2), \ldots, \Diag_{A'}(\rho_n), \Diag_{B'_0 \times A_0}(\rho_F)). \]
We show this is also an optimal solution. Note that $\Pi_{\rA,c}$ is diagonal in $\C^{A_0 \times B'_0}$, so this solution has the same objective function value.

To show feasibility, we have that
\[ \tr_{A_1} (\psi \psi^*) = \sum_{x_1 \in A'_1} e_{x_1} e_{x_1}^* \otimes \psi_{x_1} \psi_{x_1}^*, \]
where
\[ \psi_{x_1, \ldots, x_j} := \sum_{a \in A_0} \sum_{x_{j+1} \in A'_{j+1}} \cdots \sum_{x_n \in A'_n} \frac{1}{\sqrt 2} \sqrt{\alpha_{a,x}} \, e_{a} \otimes e_{a} \otimes e_{x_{j+1}} \otimes e_{x_{j+1}} \otimes \cdots \otimes e_{x_n} \otimes e_{x_n}, \]
which is in $\C^{A_0 \times A'_0 \times A_{j+1} \times A'_{j+1} \times \cdots \times A_n \times A'_{n}}$.
Since $\tr_{A_1}(\psi \psi^*)$ is diagonal in $A'_1$, we can write
\[ \tr_{B_1}(\rho'_1) = \Diag_{A'_1} (\tr_{B_1} (\rho_1)) = \Diag_{A'_1} (\tr_{A_1} (\psi \psi^*)) = \tr_{A_1} (\psi \psi^*), \]
showing the first constraint is satisfied.
Now we show that $\rho'_1$ has a special form. Let $\rho'_1 = \sum_{x_1 \in A'_1} \kb{x_1} \otimes {\rho}_{1,x_1}$, for some ${\rho}_{1,x_1} \in \Pos^{A_0 \times A'_0 \times B_1 \times A_2 \times \cdots \times A_n \times A'_2 \times \cdots \times A'_n}$.
Due to feasibility, we have
\[ \tr_{B_1}(\rho'_1) = \sum_{x_1 \in A'_1} \kb{x_1} \otimes \tr_{B_1}({\rho}_{1,x_1}) = \tr_{A_1} (\psi \psi^*) = \sum_{x_1 \in A'_1} e_{x_1} e_{x_1}^* \otimes \psi_{x_1} \psi_{x_1}^*. \]
Thus,
\[ \rho'_1 = \sum_{x_1 \in A'_1} \kb{x_1} \otimes {\psi}_{x_1} {\psi}_{x_1}^* \otimes {\tilde{\rho}}_{1,x_1}, \]
where $\tilde{\rho}_{1,x_1} \in \pos^{B_1}$ satisfies $\tr_{B_1}(\tilde{\rho}_{1,x_1}) = 1$ for all $x_1 \in A'_1$.
Now we can write
\[ \tr_{A_2}(\rho'_1) = \sum_{x_1 \in A'_1} \sum_{x_2 \in A'_2} \kb{x_1} \otimes \kb{x_2} \otimes {\psi}_{x_1, x_2} {\psi}_{x_1, x_2}^* \otimes {\tilde{\rho}}_{1,x_1}, \]
which is diagonal in $A'_2$ and
\[ \tr_{A_2}(\rho'_1)
= \Diag_{A'_2} \tr_{A_2}(\rho'_1)
= \Diag_{A'_1 \times A'_2} \tr_{A_2}(\rho_1)
= \Diag_{A'_1 \times A'_2} \tr_{B_2}(\rho_2)
= \tr_{B_2}(\rho'_2).
\]
Thus, the second constraint is satisfied.
We can similarly argue the first $n$ constraints are satisfied and, for all $j \in \set{2, \ldots, n}$,
\[ \rho'_j = \sum_{x_1 \in A'_1} \cdots \sum_{x_j \in A'_j} \kb{x_1} \otimes \cdots \otimes \kb{x_j} \otimes \psi_{x_1, \ldots, x_j} \psi_{x_1, \ldots, x_j}^* \otimes \tilde{\rho}_{j, x_1, \ldots, x_j}, \]
where $\tilde{\rho}_{j, x_1, \ldots, x_j} \in \Pos^{A_0 \times A'_0 \times B_1 \times \cdots \times B_j}$ satisfies the condition $\tr_{B_j}(\tilde{\rho}_{j, x_1, \ldots, x_j}) = \tilde{\rho}_{j-1, x_1, \ldots, x_{j-1}}$, for all $x_1 \in A'_1, \ldots, x_j \in A'_j$.
In particular,
\[ \tr_{A' \times A'_0}(\rho'_n) = \half \sum_{a \in A_0} \sum_{x \in A'} \alpha_{a,x} \, \kb{a} \otimes {\tilde{\rho}}_{n,x}. \]
Notice $\tr_{A' \times A'_0}(\rho'_n)$ is diagonal in $A_0$, thus
\begin{eqnarray*}
\tr_{A' \times A'_0}(\rho'_n)
& = & \Diag_{A_0} (\tr_{A' \times A'_0}(\rho'_n)) \\
& = & \Diag_{A_0} (\tr_{A' \times A'_0}(\rho_n)) \\
& = & \Diag_{A_0} (\tr_{B' \times B'_0}(\rho_F)) \\
& = & \tr_{B' \times B'_0}(\rho'_F),
\end{eqnarray*}
showing the last constraint is also satisfied.

We can write $\rho'_F = \sum_{a \in A_0} \sum_{b \in B'_0} \kb{a} \otimes \kb{b} \otimes \tilde{\rho}_{F,a,b}$ for some $\rho_{F,a,b} \in \Pos^{B \times B'}$.
By feasibility we have
\[ \tr_{B'_0 \times B'}(\rho'_F) = \sum_{a \in A_0} \kb{a} \otimes \sum_{b \in B'_0} \tr_{B'} (\tilde{\rho}_{F,a,b}) = \tr_{A' \times A'_0}(\rho'_n) = \half \sum_{a \in A_0} \sum_{x \in A'} \alpha_{a,x} \, \kb{a} \otimes {\tilde{\rho}}_{n,x}. \]
Thus, we have the constraints $\sum_{b \in B'_0} \tr_{B'} (\tilde{\rho}_{F,a,b}) = \half \sum_{x \in A} \alpha_{a,x} \, {\tilde{\rho}}_{n,x}$. For forcing outcome $0$, we can assume $b=a$ from the definition of $\Pi_{\rA,0}$ and for forcing outcome $1$ we can assume $b= \bar{a}$ from the definition of $\Pi_{\rA,1}$. For forcing outcome $0$, we have the constraint
$\tr_{B'} (\tilde{\rho}_{F,a,a}) = \half \sum_{x \in A} \alpha_{a,x} \, \tilde{{\rho}}_{n,x}$. Thus, we have the following SDP:
\[ \begin{array}{rrrcllllllllllllll}
& \sup                         & \sum_{a \in A_0} \inner{\tilde{\rho}_{F,a,a}}{\phi_a \phi_a^*}
\quad \\
                     & \textrm{subject to} & \tr_{B_1}({\tilde{\rho}_{1,x_1}}) & = & 1, & \forall x_1 \in A'_1, \\
                     & & \tr_{B_2}({\tilde{\rho}_{2, x_1, x_2}}) & = & {\tilde{\rho}_{x_1}}, & \forall x_1 \in A'_1, x_2 \in A'_2, \\
                     & & & \vdots \\
                     & & \tr_{B_n}({\tilde{\rho}_{n, x}}) & = & {\tilde{\rho}_{n-1, x_1, \ldots, x_{n-1}}}, & \forall x \in A', \\
                     &  & \tr_{B'}(\tilde{\rho}_{F,a,a}) & = & \half \sum_{x \in A} \alpha_{a,x} \, {\tilde{\rho}_{n,x}}, & \forall a \in A'_0, \\
                     & & \tilde{\rho}_{j,x_1, \ldots, x_j} & \in & \Pos^{B_1 \times \cdots \times B_j}, & \forall j \in \set{1, \ldots, n}, \\
                     & & & & & \forall x_1 \in A'_1, \ldots, x_{j} \in A'_j, \\
                     & & \tilde{\rho}_{F,a,a} & \in & \Pos^{B \times B'}, & \forall a \in A'_0.
\end{array} \]

Now, by Lemma~\ref{diag}, the following restriction can only help the objective function value:
\begin{eqnarray*}
p_1^{(x_1)} := \diag({\tilde{\rho}_{1, x_1}}), & \quad & \forall x_1 \in A'_1, \\
p_2^{(x_1, x_2)} := \diag({\tilde{\rho}_{2, x_1, x_2}}), & \quad & \forall x_1 \in A'_1, x_2 \in A'_2, \\
\vdots \quad \quad & & \\
p_n^{(x)} := \diag({\tilde{\rho}_{n,x}}), & \quad & \forall x \in A', \\
{\tilde{\rho}_{F,a,a}} := \sqrtt{\half \sum_{x \in A} \alpha_{a,x} \, p_n^{(x)}}, & \quad & \forall a \in A'_0.
\end{eqnarray*}
This yields the new objective function
\[ \sum_{a \in A_0} \rF\! \left( \half \sum_{x \in A} \alpha_{a,x} \, p_n^{(x)}, \beta_a \right) = \half \sum_{a \bit} \rF\! \left( (\alpha_a \otimes \id_{B})^\transpose p_n, \beta_a \right), \]
for forcing outcome $0$, and
\[ \sum_{a \in A_0} \rF\! \left( \half \sum_{x \in A} \alpha_{a,x} \, p_n^{(x)}, \beta_{\bar a} \right) = \half \sum_{a \bit} \rF\! \left( (\alpha_a \otimes \id_{B})^\transpose p_n, \beta_{\bar a} \right), \]
for forcing outcome $1$, with $(p_1, \ldots, p_n) \in \calP_\rB$. Moreover,
a feasible solution for the reduced SDP gives us a feasible solution to
the original SDP with the same objective function value. \qed
} 

This proof shows that the reduced cheating problem does not eliminate all of the optimal solutions of the corresponding SDP. We can also show that the reduced problems capture optimal solutions to the corresponding SDPs by examining the dual SDPs. However, the primal SDPs are more important for the purposes of this paper and this proof is more illustrative.


\section{Developing the strategies in the filter}

In this appendix, we prove Theorems~\ref{AliceFilter}~and~\ref{BobFilter} by developing the cheating strategies used in the filter for Alice and Bob.


\paragraph{Cheating Alice} \quad \\

Recall Alice's optimization problem
\[ P_{\rA,0}^* = \max \left\{ \half \sum_{a \bit} \sum_{y \in B} \beta_{a,y} \; \rF(s^{(a,y)}, \alpha_{a}) \; : \;
(s_{1}, \ldots, s_{n}, s) \in \calP_\rA \right\}. \]

To get a feasible solution, suppose Alice guesses $b$ before she reveals $a$ in the following way. If Bob reveals $y \in B$, then Alice guesses $b=0$ if
$\beta_{0,y} \geq \beta_{1,y}$ and $b=1$ if $\beta_{0,y} < \beta_{1,y}$. Let Alice's guess be denoted by $f(y)$, so
\[ f(y) = \arg \max_a \set{\beta_{a,y}} \in \zo, \]
and we set $f(y) = 0$ in the case of a tie.
We have chosen a way to satisfy the last constraint in Alice's cheating polytope, but we can choose how Alice sends her first $n$ messages $s_1, \ldots, s_n$. We make one more restriction, we set $s_n = d \otimes e_{B_1 \times \cdots \times B_{n-1}}$ and optimize over $d \in \prob^A$. We can easily satisfy the rest of the constraints given any $d$ by choosing each variable as the corresponding marginal probability distribution.

Under these restrictions, we have that Alice's reduced problem can be written as
\[ \max_{d \in \Prob^A} \set{\half \sum_{y \in B} \beta_{f(y), y} \rF(d, \alpha_{f(y)})} = \max_{d \in \Prob^A} \set{\eta \, \rF(d, \alpha_{0}) + \tau \, \rF(d, \alpha_{1})}. \]
We can simplify this using the following lemma.

\begin{lemma} \label{FidLemma}
For nonnegative vectors $\set{z_1, \ldots, z_n} \subset \R_+^n$, we have that
\[ \max \set{\sum_{i=1}^n \rF(p, z_i) : p \in \Prob^n} = \lambda_{\max} \left( \sum_{i=1}^n \sqrtt{z_i} \right). \]
Furthermore, an optimal solution is the entry-wise square of the normalized principal eigenvector.
\end{lemma}

\begin{proof}
Since $\sum_{i=1}^n \rF(p, z_i) = \sum_{i=1}^n \inner{\sqrtt{p}}{\sqrtt{z_i}} = \sqrt{p}^\transpose \left( \sum_{i=1}^n \sqrtt{z_i} \right) \sqrt{p}$, \\ where $\sqrt{\cdot}$ is the entry-wise square root, the maximization problem reduces to
\[ \max \set{ \sqrt{p}^\transpose \left( \sum_{i=1}^n \sqrtt{z_i} \right) \sqrt{p} : p \in \prob^n}. \]
Let $\hat{x} \in \R^m$ be the restriction of a vector $x$ onto $\cup_{i = 1}^n \supp(z_i)$. Then the optimal objective value of the above optimization problem is equal to that of
\[ \max \left\{ \sqrt{\hat{p}}^\transpose \left( \sum_{i=1}^n \sqrtt{\hat{z_i}} \right) \sqrt{\hat{p}} : \hat{p} \in \prob^{\displaystyle\cup_{i = 1}^n \supp(z_i)} \right\}. \]
If the nonnegativity constraint were not present, the optimum value would be attained by setting $\sqrt{\hat{p}}$ to be the normalized principal eigenvector of the matrix $\sum_{i=1}^n \sqrtt{\hat{z_i}}$. Because $\sum_{i=1}^n \sqrtt{\hat{z_i}}$ has positive entries, we know the principal eigenvector is also positive by the Perron-Frobenius Theorem. Since this does not violate the nonnegativity constraint in the problem, $\hat{p}$, where $\sqrt{\hat{p}}$ is the normalized principal eigenvector, is an optimal solution yielding an optimal objective value of $\lambda_{\max} \left( \sum_{i=1}^n \sqrtt{\hat{z_i}} \right)$. Notice that $\sum_{i=1}^n \sqrtt{\hat{z_i}}$ is the matrix obtained by removing the zero rows and columns from $\sum_{i=1}^n \sqrtt{{z_i}}$ and thus has the same largest eigenvalue. $\quad \square$
\end{proof}

Using this lemma, Alice can cheat with probability
\[ \half \lmax{\eta \sqrtt{\alpha_0} + \tau \sqrtt{\alpha_1}}, \]
which we call Alice's \emph{eigenstrategy}.

We can find a lower bound on this value using the following two lemmas.

\begin{lemma} \label{mn}
For $\beta_0$, $\beta_1$, $\eta$, and $\tau$ defined above, we have $\eta + \tau = 1 + \Delta(\beta_0, \beta_1)$.
\end{lemma}

\begin{proof}
Notice that we can write ${\dsum_{y \in B} \displaystyle\max_{a \bit} \set{\beta_{a,y}} + \dsum_{y \in B} \displaystyle\min_{a \bit} \set{\beta_{a,y}} = 2}$ and we can also write
${\dsum_{y \in B} \displaystyle\max_{a \bit} \set{\beta_{a,y}} - \dsum_{y \in B} \displaystyle\min_{a \bit} \set{\beta_{a,y}} = 2 \Delta(\beta_0, \beta_1)}$. With this, we can conclude that $\eta + \tau = \dsum_{y \in B} \displaystyle\max_{a \bit} \set{\beta_{a,y}} = 1 + \Delta(\beta_0, \beta_1)$, as desired. \qed
\end{proof}

The above lemma can be restated as $\sum_{y \in B} \max_{a \bit} \set{\beta_{a,y}} = 1 + \Delta(\beta_0, \beta_1)$ for any probability distributions $\beta_0$ and $\beta_1$. This is helpful when looking at Bob's cheating strategies as well.

\begin{lemma} \label{lmaxfid2}
For $\eta, \tau \in \R$ and $p, q \in \Prob^n$, we have
\[ \lmax{ \eta \sqrtt{p} + \tau \sqrtt{q} } \\
= \frac{1}{2} \left( \eta + \tau + \sqrt{ (\eta - \tau)^2 + 4 \eta \tau \, \rF(p, q) } \right). \]
\end{lemma}

\begin{proof}
Since we can write $\rF(p,q) = \left( \sqrt{p}^{\transpose} \sqrt{q} \right)^2$, we can apply a unitary to both $\sqrt{p}$ and $\sqrt{q}$ and both sides of the equality we want to prove are unaffected. Choose a unitary $U$ such that
\[ U \sqrt{p} = [1,0, 0, \ldots, 0]^{\transpose} \quad \text{ and } \quad U \sqrt{q} = [\sin \theta, \cos \theta, 0, \ldots, 0]^{\transpose}, \]
for some $\theta \in [0, 2 \pi)$. Then we can write $\rF(p,q) = \sin^2 \theta$. Let $\lambda_{\max}$ be the largest eigenvalue of $\eta \sqrtt{p} + \tau \sqrtt{q}$, or equivalently, of $\eta U \sqrtt{p} U^* + \tau U \sqrtt{q} U^*$, and let $\lambda_{2}$ be the second largest eigenvalue. Then
\[ \lambda_{\max} + \lambda_{2} = \tr(\eta \sqrtt{p} + \tau \sqrtt{q}) = \eta + \tau \]
and, by taking the determinant of the only nonzero  block, we get
\[ \lambda_{\max} \cdot \lambda_{2} = \eta \tau \cos^2 \theta = \eta \tau (1-\rF(p,q)) \]
implying $\lambda_{\max} = \half \left( \eta + \tau + \sqrt{(\eta - \tau)^2 + 4 \eta \tau \rF(p,q)} \right)$, as desired. \qed
\end{proof}
Note that Lemma~\ref{lmaxfid2} shows that switching the roles of $\eta$ and $\tau$ does not affect the largest eigenvalue.

Using the above two lemmas, we have
\begin{eqnarray*}
\lefteqn{\half \lmax{\eta \sqrtt{\alpha_0} + \tau \sqrtt{\alpha_1}} } \\
& = & \frac{1}{4} \left( \eta + \tau + \sqrt{ (\eta - \tau)^2 + 4 \eta \tau \, \rF(\alpha_0, \alpha_1) } \right) \\
& \geq & \frac{1}{4} \left( \eta + \tau + \sqrt{ (\eta - \tau)^2 \,\rF(\alpha_0, \alpha_1) + 4 \eta \tau \, \rF(\alpha_0, \alpha_1) } \right) \\
& = & \frac{1}{4} \left( \left( 1 + \sqrt{\rF(\alpha_0, \alpha_1)} \right) \left( \eta + \tau \right) \right) \\
& = & \left( \half + \half \sqrt{\rF(\alpha_0, \alpha_1)} \right) \left( \half + \half \Delta(\beta_0, \beta_1) \right).
\end{eqnarray*}

This lower bound has a natural interpretation. This is the strategy where Alice ignores all of Bob's messages until $\C^{B_n}$ is sent. Then she measures it to learn $b$ with probability $\half + \half \Delta(\beta_0, \beta_1)$. Conditioned on having the correct value for $b$, she tries to get past Bob's cheat detection and can do so with probability $\half + \half \sqrt{\rF(\alpha_0, \alpha_1)}$. We call this Alice's \emph{three-round strategy} since it combines optimal strategies for the three-round protocol example in Subsection~\ref{ex}. It makes sense that this is a lower bound on the success probability of Alice's eigenstrategy since her eigenstrategy is optimized from the same restrictions that apply to her three-round strategy.

We can also examine how Alice can choose her last message optimally supposing she has already sent her first $n$ messages in a particular way. I.e., suppose
${s_n := c \otimes e_{B_1 \times \cdots \times B_{n-1}}}$ for some $c \in \prob^A$ (as in the eigenstrategy). From this we can find $s_1, \ldots, s_{n-1}$ satisfying the first $n-1$ constraints of her cheating polytope by taking the corresponding marginal distributions of $c$. We want to optimize over $s$ satisfying $\tr_{A'_0}(s) = s_n \otimes e_{B_n} = c \otimes e_B$. In this case, this constraint can be written as $\sum_{a \bit} s^{(a,y)} = c$, for each $y \in B$, where again, $s^{(a,y)}$ is the restriction of $s$ with $a$ and $y$ fixed. Now we get the following optimization problem

\[ \begin{array}{rrrcllllllllllllll}
 & \max                         & \half \sum_{a \bit} \sum_{y \in B} \beta_{a,y}\, \rF(s^{(a,y)}, \alpha_a) \\
                     & \textrm{subject to} & \sum_{a \bit} s^{(a,y)} & = & c, & \forAll y \in B, \\
                     &                                   & s^{(a,y)} & \geq & 0,
\end{array} \]
where $c$ is now constant. If we rewrite this as

\[ \begin{array}{rrrcllllllllllllll}
 & \max                         & \half \sum_{y \in B} \sum_{a \bit} \rF(s^{(a,y)}, \beta_{a,y} \alpha_a) \\
                     & \textrm{subject to} & \sum_{a \bit} s^{(a,y)} & = & c, & \forAll y \in B, \\
                     &                                   & s^{(a,y)} & \geq & 0,
\end{array} \]
we have a separable problem over $y \in B$. That is, for each fixed $\tilde{y} \in B$, Alice needs to solve the optimization problem

\[ G_{\tilde{y}}(c) := \max \left\{ \half \sum_{a \bit} \rF(s^{(a,\tilde{y})}, \beta_{a,\tilde{y}} \alpha_a): \sum_{a \bit} s^{(a,\tilde{y})} = c, \, s^{(a,\tilde{y})} \geq 0, \, \forall a \bit \right\}. \]

This optimization problem has a special structure.

\begin{definition}
The \emph{infimal convolution} of the convex functions $f_1, f_2, \ldots, f_n$, where \\ $f_1, \ldots, f_n: \R^m \to \R \cup \set{\infty}$, is
\[ (f_1 \infconv f_2 \infconv \cdots \infconv f_n)(d) := \inf_{x_1, \ldots, x_n \in \R^m} \left\{\sum_{i = 1}^n f_i(x_i): \sum_{i = 1}^n x_i = d \right\}. \]
\end{definition}

We do not need to worry about the nonnegativity constraints on the variables since we can define our convex function $-\rF(p,q) = + \infty$ if $p$ or $q$ is not nonnegative. Note for every $p \in \R_+^m$, that $-\rF(p, \cdot )$ is a \emph{proper, convex function}, i.e., it is convex and $-\rF(p,q) < + \infty$ for some $q \in \R_+^m$ and $-\rF(p,q) > - \infty$ for every $q \in \R_+^m$. Proper, convex functions have many useful properties as detailed in this section. Using these properties and the fact that $-\rF(p, \cdot )$ is positively homogeneous, we show a way to express $G_{\tilde{y}}$.

Recall that for proper, convex functions ${f_1, \ldots, f_n} : \R^m \to \R \cup \set{\infty}$, the convex hull of $\{ f_1, \ldots, f_n \}$ is the greatest convex function $f$ such that $f(x) \leq f_1(x), \ldots, f_n(x)$ for every $x \in \R^m$. To write down explicitly what the convex hull is, we use the following theorem.

\begin{theorem}[{[Roc70, page 37]}] \label{convexhull1}
Let ${f_1, \ldots, f_n} : \R^m \to \R \cup \set{\infty}$ be proper, convex functions. Then we have
\[ \conv \set{f_1, \ldots, f_n}(d) = \inf \left\{ \sum_{i=1}^n \lambda_i f_i(x_i) : \sum_{i=1}^n \lambda_i x_i = d \right\}. \]
\end{theorem}

For a positively homogeneous function $f$, we have $\lambda \, f \left( \lambda^{-1} x \right) = f(x)$, for $\lambda > 0$. Therefore, we have the following corollary.

\begin{corollary}
Let ${f_1, \ldots, f_n} : \R^m \to \R \cup \set{\infty}$ be positively homogeneous, proper, convex functions. Then we have
\[ \conv \set{f_1, \ldots, f_n} = f_1 \infconv f_2 \infconv \cdots \infconv f_n. \]
\end{corollary}

Therefore, we can write Alice's cheating probability using concave hulls as shown below

\begin{eqnarray*}
G_{\tilde{y}}(c)
& = & \max \left\{ \half \sum_{a \bit} \rF(s^{(a,\tilde{y})}, \beta_{a,\tilde{y}} \alpha_a): \sum_{a \bit} s^{(a,\tilde{y})} = c, \, s^{(a,\tilde{y})} \geq 0, \, \forall a \bit \right\} \\
& = & - \min \left\{ -\frac{1}{2} \sum_{a \bit} \rF(s^{(a,\tilde{y})}, \beta_{a,\tilde{y}} \alpha_a): \sum_{a \bit} s^{(a,\tilde{y})} = c, \, s^{(a,\tilde{y})} \geq 0, \, \forall a \bit \right\} \\
& = & - \left( -\frac{1}{2} \rF( \cdot, \beta_{0,\tilde{y}} \alpha_0) \right) \infconv \left( -\frac{1}{2} \rF( \cdot, \beta_{1,\tilde{y}} \alpha_1) \right)(c) \\
& = & - \conv \set{\frac{-1}{2} \beta_{0,{ \tilde y}} \rF(\cdot, \alpha_0), \frac{-1}{2} \beta_{1,{\tilde y}} \rF(\cdot, \alpha_1)}(c) \\
& = & \conc \set{\frac{1}{2} \beta_{0,{ \tilde y}} \rF(\cdot, \alpha_0), \frac{1}{2} \beta_{1,{\tilde y}} \rF(\cdot, \alpha_1)}(c).
\end{eqnarray*}

Thus, for each $c \in \prob^A$, we can write Alice's cheating probability as
\[ \sum_{y \in B} \conc \set{\frac{1}{2} \beta_{0,{y}} \rF(\cdot, \alpha_0), \frac{1}{2} \beta_{1,{y}} \rF(\cdot, \alpha_1)}(c).  \]
Note this way of optimizing the last message works for any strategy. For a general strategy, we would have a different $c$ for every $y_1, \ldots, y_{n-1}$.

Thus, we have Alice's \emph{improved eigenstrategy} which is when Alice chooses her first $n$ messages according to her eigenstrategy, yet reveals $a$ optimally.

\paragraph{Cheating Alice in six-round protocols} \quad \\

In six-round protocols, Alice's goal is to maximize the objective function
\[ \dfrac{1}{2} \dsum_{a \in \zo} \dsum_{y_1 \in B_1} \dsum_{y_2 \in B_2} \beta_{a,y_1 y_2} \rF(s^{(a,y_1 y_2)}, \alpha_a)
\] over $(s_1, s_2, s)$ satisfying:
\[ \begin{array}{rrrcllllllllllllll}
                     &                               & \tr_{A_1}(s_1) & = & 1, \\
                     &                               & \tr_{A_2}(s_2) & = & s_1 \otimes e_{B_{1}}, \\
                     &                               & \tr_{A'_{0}}(s) & = & s_2 \otimes e_{B_{2}}, \\
                     \quad \\
                     &                               & s_{1} & \in & \R_{+}^{A_{1}}, \\
                     &                               & s_{2} & \in & \R_{+}^{A_{1} \times B_{1} \times A_{2}}, \\
                     & & s & \in & \R_{+}^{A_1 \times A_2 \times B_1 \times B_2 \times A'_{0}}.
\end{array} \]
We suppose that Alice chooses her commitment $a$ based on the most likely choice of $b$ after seeing $y_1$ from Bob's first message. Let $f'(y_1) = \arg\max_{a \in A'_0} \set{[\tr_{B_2}(\beta_{a})]_{y_1}}$ and $0$ in the case of a tie. The last constraint can be written as $\sum_{a \in A'_0} s^{(a,y_1 y_2)} = s_2^{(y_1)}$, for all $y_1 \in B_1$, where $s_2^{(y_1)}$ is the projection of $s_2$ with the index $y_1$ fixed. We set $s^{(a, y_1, y_2)} = s_2^{(y_1)}$, if $a = f'(y_1)$, and $0$ otherwise. Now we set $s_2^{(y_1)} = s_2^0$, if $f'(y_1)
= 0$, and $s_2^{(y_1)} = s_2^1$, if $f'(y_1) = 1$, where we optimize $s_2^0, s_2^1 \in \R_+^{A_1 \times A_2}$. The new objective function can be written as
\begin{eqnarray*}
\half \sum_{a \in A'_0} \sum_{y_1 \in B_1, y_2 \in B_2} \beta_{a,y_1 y_2} \rF(s^{(a,y_1 y_2)}, \alpha_a)
& = & \half \sum_{y_1 \in B_1} \left[ \sum_{y_2 \in B_2} \beta_{f'(y_1),y_1 y_2} \right] \rF(s_2^{f'(y_1)}, \alpha_{f'(y_1)}) \\
& = & \half \eta' \, \rF(s_2^{0}, \alpha_{0}) + \half \tau' \, \rF(s_2^{1}, \alpha_{1}).
\end{eqnarray*}
Since the only constraints remaining are $\tr_{A_2}(s_2^0) = s_1 = \tr_{A_2}(s_2^1)$, we now optimize over each choice of $s_2^0$ and $s_2^1$ separately using the following lemma.
\begin{lemma} \label{PartialFidelityLemma}
For $\alpha \in \R_+^{A_1 \times A_2}$ \comment{ with~$|A_2| \ge |A_1|$,} and $c \in \R_+^{A_1}$, we have
\[ \max \set{\rF(p, \alpha) : \tr_{A_2}(p) = c, \, p \geq 0} \geq \rF(c, \tr_{A_2}(\alpha)). \]
\end{lemma}
The inequality can be shown to hold with equality by Uhlmann's theorem. However, we prove the inequality by exhibiting a feasible solution which is also useful for the analysis of cheating Bob.

\begin{proof}
For each $x_1 \in A_1$, $x_2 \in A_2$, define $p_{x_1, x_2}$ as
\[ p_{x_1, x_2} := \left\{ \begin{array}{ccl}
c_{x_1} \frac{\alpha_{x_1, x_2}}{[\tr_{A_2}(\alpha)]_{x_1}} & \text{ if } [\tr_{A_2}(\alpha)]_{x_1} > 0, \\
\quad \\
c_{x_1} \frac{1}{|A_2|} & \text{ if } [\tr_{A_2}(\alpha)]_{x_1} = 0.
\end{array} \right. \]
Then we have $p \geq 0$ is feasible since $[\tr_{A_2}(p)]_{x_1} = c_{x_1}$ and it has objective function value $\rF(p, \alpha) = \rF(c, \tr_{A_2}(\alpha))$, as desired. \qed
\comment{
\begin{eqnarray*}
\rF(p, \alpha)
& = & \left( \sum_{x_1, x_2 : \, \alpha_{x_1, x_2} > 0} \sqrt{p_{x_1, x_2}} \sqrt{\alpha_{x_1, x_2}} \right)^2 \\
& = & \left( \sum_{x_1, x_2} \sqrt{c_{x_1} \frac{\alpha_{x_1, x_2}}{[\tr_{A_2}(\alpha)]_{x_1}}} \sqrt{\alpha_{x_1, x_2}} \right)^2 \\
& = & \left( \sum_{x_1} \sqrt{c_{x_1}} \sqrt{[\tr_{A_2}(\alpha)]_{x_1}} \right)^2, \\
& = & \rF(c, \tr_{A_2}(\alpha)),
\end{eqnarray*}
} 
\end{proof}

Using the lemma, we can write the problem as
\[ \max_{c \in \prob^{A_1}} \eta' \, \rF(c, \tr_{A_2}(\alpha_{0})) + \tau' \, \rF(c, \tr_{A_2}(\alpha_{1})) \]
which has optimal value $\half \lambda_{\max} \left( \eta' \sqrtt{\tr_{A_2}(\alpha_0)} + \tau' \sqrtt{\tr_{A_2}(\alpha_1)} \right)$
and is lower bounded by $\left( \half + \half \sqrt{\rF(\tr_{A_2}(\alpha_0), \tr_{A_2}(\alpha_1))} \right) \left( \half + \half \Delta (\tr_{B_2}(\beta_0), \tr_{B_2}(\beta_1)) \right)$.
Again, this last quantity has context. This is the strategy where Alice measures the first message to learn $b$ early and then tries to change the value of $a$. She can learn $b$ with probability $\half + \half \Delta(\tr_{B_2}(\beta_0), \tr_{B_2}(\beta_1))$. She can successfully change the value of $a$ with probability $\half + \half \sqrt{\rF(\tr_{A_2}(\alpha_0), \tr_{A_2}(\alpha_1))}$. Thus, she can cheat with probability at least
\[ \left( \half + \half \sqrt{\rF(\tr_{A_2}(\alpha_0), \tr_{A_2}(\alpha_1))} \right) \left( \half + \half \Delta(\tr_{B_2}(\beta_0), \tr_{B_2}(\beta_1)) \right). \]


\paragraph{Cheating Bob} \quad \\

Bob's returning strategy is to send Alice's messages right back to her (if the dimensions agree). This way, the state that Alice checks at the end of the protocol is her own state. This is a good strategy when Alice and Bob share the same starting states, i.e., for a  protocol with parameters $\alpha_0 = \beta_0$ and $\alpha_1 = \beta_1$. To calculate the cheating probability of this strategy, for any choice of parameters, it is easier to use the original cheating SDP as opposed to the reduced cheating SDP. This cheating strategy corresponds to the feasible solution
\[ \bar{\rho}_1 = \bar{\rho}_2 = \cdots = \bar{\rho}_n = \bar{\rho}_F = \psi \psi^* \]
which has success probability given by the objective function value
\[ \inner{\bar{\rho}_F}{\Pi_{\rA,0}} = \inner{\psi \psi^*}{\Pi_{\rA,0}} = \half \sum_{a \bit} \rF(\alpha_a, \beta_a). \]
This is clearly optimal when $\alpha_0 = \beta_0$ and $\alpha_1 = \beta_1$.

Recall Bob's reduced problem below
\[ P_{\rB,0}^* = \max \left\{ \half \sum_{a \bit} \, \rF \left( (\alpha_{a} \otimes \id_{B})^{\transpose}p_{n}, \, \beta_a \right) : (p_{1}, \ldots, p_{n}) \in \calP_\rB \right\}. \]
There is a strategy for Bob that works for any $n$ and is very important in the search algorithm. This is the strategy where Bob ignores all of Alice's messages and tries to choose $b$ after learning $a$ from Alice. By ignoring Alice's messages, he effectively sets $p_n = e_{A} \otimes d$, for some $d \in \prob^B$, which we optimize. Under this restriction, he can cheat with probability
\begin{eqnarray*}
\max_{d \in \prob^B} \half \sum_{a \bit} \, \rF \left( (\alpha_{a} \otimes \id_{B})^{\transpose}(e_{A} \otimes d), \, \beta_a \right)
& = &
\max_{d \in \prob^B} \half \sum_{a \bit} \, \rF \left( d, \beta_a \right) \\
& = & \half \lambda_{\max} \left( \sqrtt{\beta_0} + \sqrtt{\beta_1} \right) \\
& = & \half + \half \sqrt{\rF(\beta_0, \beta_1)}
\end{eqnarray*}
using Lemma~\ref{FidLemma} and Lemma~\ref{lmaxfid2}. Note this is similar to the three-round case (discussed in Subsection~\ref{ex}). The reason this strategy is important is that it is easy to compute, only depends on half of the parameters, and is effective in pruning
sub-optimal protocols. We call this Bob's \emph{ignoring strategy}.

Another strategy for Bob is to measure Alice's first message, choose $b$ accordingly, then play honestly. This is called Bob's \emph{measuring strategy} and succeeds with probability
\[ \half + \half \Delta(\tr_{A_2 \times \cdots \times A_n} \left( \alpha_0 \right), \tr_{A_2 \times \cdots \times A_n} \left( \alpha_1 \right)), \]
when $n \geq 2$.

\paragraph{Cheating Bob in four-round protocols} \quad \\

There are cheating strategies that apply to four-round protocols, that do not
extend to a larger number of rounds.
For example, Bob has all of Alice's $\C^A$ space before he sends any messages. We show that Bob can use this to his advantage. One example is Bob's measuring strategy,
which leads to a cheating probability of
\[ \half + \half \Delta(\alpha_0, \alpha_1) \enspace. \]

Similar to cheating Alice, we can develop an eigenstrategy for Bob. For the special case of four-round protocols, notice that Bob's cheating polytope contains only the constraints $\tr_{B}(p) = e_{A}$ and $p \in \R_+^{A \times B}$. This can be rewritten as
$p_x \in \Prob^B$ for all $x \in A$. Also, $\rF \left( (\alpha_{a} \otimes \id_{B})^{\transpose} p_{n}, \, \beta_a \right)$ can be written as $\rF\! \left( \sum_{x \in A} \alpha_{a,x} p_n^{(x)} , \beta_a \right)$, where $p_n^{(x)}$ is the projection of $p_n$ with $x$ fixed. Thus, we can simplify Bob's reduced problem as
\[ P_{\rB,0}^* = \max \set{ \half \sum_{a \bit} \rF\! \left( \sum_{x \in A} \alpha_{a,x} p_n^{(x)} , \beta_a \right) : p_n^{(x)} \in \prob^B, \, \forAll x \in A}. \]
Since fidelity is concave, we have that $\rF\! \left( \sum_{x \in A} \alpha_{a,x} p_n^{(x)}, \beta_a \right) \geq \sum_{x \in A} \alpha_{a,x}
\, \rF(p_n^{(x)}, \beta_a)$.
Therefore Bob's optimal cheating probability is bounded below by
\[ \max \set{ \half \sum_{x \in A} \sum_{a \bit} \alpha_{a,x} \, \rF(p_n^{(x)}, \beta_a) : p_n^{(x)} \in \prob^B, \, \forAll x \in A} \]
which separates over $x \in A$. That is, we choose each $p_n^{(x)} \in \prob^B$ separately to maximize $\dsum_{a \bit} \alpha_{a,x} \, \rF(p_n^{(x)}, \beta_a)$, which has optimal objective value
$\lambda_{\max}\!\left( \dsum_{a \bit} \alpha_{a,x} \sqrtt{\beta_a} \right)$ using Lemma \ref{FidLemma}.
Thus, we know that
\[ P_{\rB,0}^* \geq \half \sum_{x \in A} \lambda_{\max} \left( \sum_{a \bit} \alpha_{a,x} \sqrtt{\beta_a} \right). \]

Since we use the concavity of the objective function, the bound we get
may not be tight. Notice that solving the smaller separated problems yields a solution which is feasible for the original problem. Therefore, we can substitute this into the original objective function to get a better lower bound on Bob's optimal cheating probability. We call this Bob's \emph{eigenstrategy}.

Since eigenvalues are expensive to compute, we can bound this quantity by
\begin{eqnarray*}
\frac{1}{2} \sum_{x \in A} \lambda_{\max}\!\left( \sum_{a \bit} \alpha_{a,x} \sqrtt{\beta_a}  \right)
& \geq & \min_{\beta_0, \beta_1 \in \prob^B} \frac{1}{2} \sum_{x \in A} \lambda_{\max} \left( \sum_{a \bit} \alpha_{a,x} \sqrtt{\beta_a}  \right) \\
& = & \half \sum_{x \in A} \max_{a \bit} \set{\alpha_{a,x}} \\
& = & \frac{1}{2} + \frac{1}{2} \Delta(\alpha_0, \alpha_1) \enspace,
\end{eqnarray*}
where the last equality follows from Lemma~\ref{mn}.

Since $\lambda_{\max}(X+Y) \leq \lambda_{\max}(X) + \lambda_{\max}(Y)$ for all matrices $X$ and $Y$, we have that
\begin{eqnarray*}
\frac{1}{2} \sum_{x \in A} \lambda_{\max}\!\left( \sum_{a \bit} \alpha_{a,x} \sqrtt{\beta_a}  \right)
& \geq &
\frac{1}{2} \lambda_{\max}\!\left( \sum_{x \in A} \sum_{a \bit} \alpha_{a,x} \sqrtt{\beta_a}  \right) \\
& = &
\frac{1}{2} \lambda_{\max}\!\left( \sum_{a \bit} \sqrtt{\beta_a}  \right) \\
& = &
\half + \half \sqrt{\rF(\beta_0, \beta_1)} \enspace.
\end{eqnarray*}
Therefore, Bob's eigenstrategy performs better than both his measuring strategy and ignoring strategy.

\paragraph{Cheating Bob in six-round protocols} \quad \\

In six-round protocols, Bob's goal is to maximize the objective function
\[ \dfrac{1}{2} \dsum_{a \in \zo \comment{A'_0}} \rF((\alpha_a \otimes \id_{B_1 \times B_2})^\transpose p_2, \beta_a) \]
over $(p_1, p_2)$ satisfying:
\[ \begin{array}{rrrcllllllllllllll}
                     &                               & \tr_{B_1}(p_1) & = & e_{A_1}, \\
                     &                               & \tr_{B_2}(p_2) & = & p_1 \otimes e_{A_{2}}, \\
                     &                               & p_{1} & \in & \R_{+}^{A_{1} \times B_{1}}, \\
                     &                               & p_{2} & \in & \R_{+}^{A_{1} \times B_{1} \times A_{2} \times B_{2}}.
\end{array} \]
Like in four-round protocols, we can lower bound the objective function as
\[ \half \sum_{a \in A'_0} \rF\!\left( \sum_{x \in A} \alpha_{a,x} p_2^{(x)}, \beta_a \right) \geq \half \sum_{x \in A} \sum_{a \in A'_0} \rF(p_2^{(x)}, \alpha_{a,x} \beta_a) \]
and focus our attention on optimizing the function $\dsum_{a \in A'_0} \rF(p_2^{(x)}, \alpha_{a,x} \beta_a)$.
We use the following lemma.

\begin{lemma} \label{pfidlemma2}
For $\beta_0, \beta_1 \in \R_+^{B_1 \times B_2}$ and $c \in \R_+^{B_1}$, we have
\[ \max \set{ \sum_{a \bit} \rF(p, \beta_a): \tr_{B_2}(p) = c, \, p \geq 0 } \geq \rF(c, \tr_{B_2}(\beta_{\tilde{a}})), \]
for \emph{any} $\tilde{a} \bit$.
\end{lemma}

\begin{proof}
Fix any $\tilde{a}$ and choose $p \in \arg\max \set{ \rF(p, \beta_{\tilde{a}}): \tr_{B_2}(p) = c, \, p \geq 0 }$. Since the fidelity is nonnegative, the result follows by Lemma~\ref{PartialFidelityLemma}. \qed
\end{proof}

By setting $p_1 = c \otimes e_{A_1}$, we have the constraint $\tr_{B_2}(p^{(x)}) = c$ for all $x \in A$. We now apply Lemma~\ref{pfidlemma2} to get
\[ \max_{p_2^{(x)}} \left\{ \sum_{a \in A'_0} \rF(p_2^{(x)}, \alpha_{a,x} \beta_a) \right\} \geq \alpha_{g(x), x} \, \rF(c, \tr_{B_2}(\beta_{g(x)})), \]
where $g(x) := \arg\max_{a \in A'_0} \set{\alpha_{a,x}}$, and $0$ in the case of a tie.

Substituting this into the relaxed objective function above, we have
\begin{eqnarray}
\lefteqn{ \max_{c \in \prob^{B_1}} \frac{\kappa}{2} \, \rF(c, \tr_{B_2}(\beta_{0})) + \frac{\zeta}{2} \, \rF(c, \tr_{B_2}(\beta_1))
} \nonumber \\
& = & \half \lambda_{\max} \left( \kappa \sqrtt{\tr_{B_2}(\beta_{0})} + \zeta \sqrtt{\tr_{B_2}(\beta_{1})} \right) \\
& \geq & \left( \half + \half \Delta(\alpha_0, \alpha_1) \right)  \left( \half + \half \sqrt{\rF(\tr_{B_2}(\beta_{0}), \tr_{B_2}(\beta_{1}))} \right). \label{B15}
\end{eqnarray}
The quantity $\textup{(\ref{B15})}$ corresponds to the strategy where Bob measures Alice's second message to try to learn $a$ early, then tries to change the value of $b$. He can learn $a$ after Alice's second message with probability $\half + \half \Delta(\alpha_0, \alpha_1)$. He can change the value of $b$ with probability $\half + \half \sqrt{\rF(\tr_{B_{2}}(\beta_0), \tr_{B_{2}}(\beta_1))}$. Thus, he can cheat with probability at least
\[ \left( \half + \half \sqrt{\rF(\tr_{B_{2}}(\beta_0), \tr_{B_{2}}(\beta_1))} \right) \left( \half + \half \Delta(\alpha_0, \alpha_1) \right). \]
We call this Bob's \emph{three-round strategy}.

Although we used many bounds in developing the quantity $\textup{(\ref{B6})}$, such as concavity and the lower bound in Lemma~\ref{pfidlemma2}, we can recover some of the losses by generating its corresponding feasible solution and computing its objective function value for the original objective function. For example, we can calculate $c$ as the entry-wise square of the normalized principal eigenvector of
\[ \half \lambda_{\max} \left( \kappa \sqrtt{\tr_{B_2}(\beta_{0})} + \zeta \sqrtt{\tr_{B_2}(\beta_{1})} \right), \]
then calculate $p_2^{(x)}$ for each value of $x$ from the construction of the feasible solution in the proof of Lemma~\ref{PartialFidelityLemma}. We call this Bob's \emph{eigenstrategy}.

\end{document}